\documentclass[leqno,10pt,a4paper]{amsart}
\usepackage[margin=2.5cm]{geometry}
\usepackage{cite}
\usepackage{graphicx}
\usepackage{amscd}
\usepackage{amssymb}
\usepackage[usenames,dvipsnames]{color}
\usepackage{hyperref}
\usepackage{xcolor}
\definecolor{c1}{RGB}{10,100,155}
\definecolor{c2}{RGB}{50,135,90}
\hypersetup{
pagebackref=true,
colorlinks=true,
linktocpage=true,
linkbordercolor= c2,
citecolor=c1, 
linkcolor=c2,
urlcolor=gray
}
\usepackage{mleftright}
\usepackage{booktabs}
\usepackage[shortlabels]{enumitem}
\usepackage{mathtools}
\usepackage{subcaption}
\usepackage{threeparttable}
\usepackage{alphalph}
\usepackage{comment}
\usepackage[colorinlistoftodos,bordercolor=gray,backgroundcolor=white,linecolor=black,textsize=scriptsize]{todonotes}

\setlist[enumerate]{labelsep=*, leftmargin=1.5pc}
\setlist[enumerate]{label=\normalfont(\roman*), ref=\roman*}
\DeclareCaptionSubType*[arabic]{figure}

\newtheorem{theo}{Theorem}[section]

\newtheorem{prop}[theo]{Proposition}
\newtheorem{coro}[theo]{Corollary}

\theoremstyle{definition}
\newtheorem{dfn}[theo]{Definition}
\newtheorem{ex}[theo]{Example}
\newtheorem{rem}[theo]{Remark}
\newtheoremstyle{noparens}%
  {}{}%
  {\itshape}{}%
  {\bfseries}{.}%
  { }%
  {\thmname{#1}\thmnumber{ #2}\thmnote{ #3}}
\theoremstyle{noparens}
\newtheorem{exnopar}[theo]{Example}
\newtheorem{propnopar}[theo]{Proposition}
\renewcommand{\emptyset}{\varnothing}
\DeclareMathOperator{\vol}{vol}
\DeclareMathOperator{\conv}{conv}
\DeclareMathOperator{\Hom}{Hom}
\newcommand*{\defeq}{\coloneqq}
\newcommand{\CC}{\mathbb{C}}
\newcommand{\PP}{\mathbb{P}}
\newcommand{\ZZ}{\mathbb{Z}}
\newcommand{\QQ}{\mathbb{Q}}
\newcommand{\RR}{\mathbb{R}}
\newcommand{\T}{\mathbb{T}}
\newcommand{\NQ}{N_\QQ}
\newcommand{\MQ}{M_\QQ}
\DeclareMathOperator{\supp}{supp}
\DeclareMathOperator{\refl}{ref}
\DeclareMathOperator{\FI}{FI}
\DeclareMathOperator{\can}{can}
\DeclareMathOperator{\ord}{ord}
\DeclareMathOperator{\vertex}{vert}
\newcommand{\modb}[1]{\left(\mathrm{mod}\ {#1}\right)}
\newcommand{\calh}{\mathcal{H}}
\newcommand{\mylabel}[2]{#2\def\@currentlabel{#2}\label{#1}}
\newcommand{\canonicalFano}{$674,\!688$}

\graphicspath{{images/}}
\begin{document}
\author[V.\,Batyrev]{Victor Batyrev}
\address{Mathematisches Institut, Universit\"at T\"ubingen, Auf der Morgenstelle 10, 72076 T\"ubingen, Germany}
\email{victor.batyrev@uni-tuebingen.de}

\author[A.\,M.\,Kasprzyk]{Alexander Kasprzyk}
\address{School of Mathematical Sciences, University of Nottingham, Nottingham, NG7~2RD, UK}
\email{a.m.kasprzyk@nottingham.ac.uk}

\author[K.\,Schaller]{Karin Schaller}
\address{Mathematisches Institut, Freie Universit\"at Berlin, Arnimallee 3, 14195 Berlin, Germany}
\email{karin.schaller@fu-berlin.de}

\title{On the Fine interior of three-dimensional canonical~Fano~polytopes}
\begin{abstract}
The Fine interior $\Delta^{\FI}$ of a $d$-dimensional lattice polytope $\Delta$ is a rational subpolytope of $\Delta$ which is important for 
constructing minimal birational models of non-degenerate hypersurfaces defined by Laurent polynomials with Newton polytope $\Delta$. 
This paper presents some computational results on the Fine interior of all~\canonicalFano\ three-dimensional canonical Fano polytopes.
\end{abstract}
\maketitle
\section{Introduction}

Let $M \cong \ZZ^d$ be a free abelian group of rank $d$.
We set $\MQ \defeq M \otimes \QQ$
and denote by $N$ the dual group $\Hom(M, \ZZ)$ in the dual
$\QQ$-linear vector space
$\NQ\defeq \Hom(M, \QQ)$. Let $\langle \cdot, \cdot \rangle:
 \MQ \times \NQ \to \QQ$ be the natural
pairing.

A convex compact $d$-dimensional polytope
$\Delta \subseteq \MQ$ is called {\em lattice $d$-tope} if all vertices
of $\Delta$ belong to the lattice $M \subseteq \MQ$, \emph{i.e.}, $\Delta$ equals
the convex hull $\conv(\Delta \cap M)$ of all lattice points in
$\Delta$.
The usual interior $\Delta^{\circ}$ of $\Delta$ is
the complement $\Delta \setminus \partial \Delta$, where
$\partial \Delta$ is the boundary of $\Delta$.
Another interior of a lattice
polytope $\Delta$ was introduced by J. Fine~\cite{Fine83,Rei87,Ish99,Bat18}:

\begin{dfn}
Let $\Delta \subseteq \MQ$ be a lattice $d$-tope. Denote by
${\ord}_\Delta$ the piecewise linear function $\NQ \to \QQ$ with
\[ {\ord}_\Delta (y) \defeq \min_{ x \in \Delta} \langle x, y \rangle \;\; (y \in \NQ).\]
Then the convex subset
\[ \Delta^{\FI}\defeq \bigcap_{n \in N \setminus \{0 \}} \{ x \in \MQ \, \mid \,
\langle x, n \rangle \geq {\ord}_\Delta(n) + 1 \} \]
is called {\em Fine interior} of $\Delta$.
\end{dfn}

One can show that only finitely many linear inequalities
$\langle x, n \rangle \geq {\ord}_\Delta(n) + 1$ are necessary to define
$\Delta^{\FI}$. Therefore, $\Delta^{\FI}$ is a convex hull of finitely many
rational points $p \in \MQ$. Moreover, any
lattice point $p \in \Delta^\circ \cap M$ in the usual interior of $\Delta$
is contained in $\Delta^{\FI}$. Therefore,
$\Delta^{\FI}$ contains the convex hull of $\Delta \cap M$, \emph{i.e.},
we get the inclusion $\conv(\Delta^{\circ} \cap M)
\subseteq \Delta^{\FI}$. In particular,
$\Delta^{\FI}$ is non-empty if $\Delta^{\circ} \cap M$ is non-empty. Moreover,
for any lattice polytope $\Delta$ of dimension $d \leq 2$
one has the equality $\conv(\Delta^{\circ} \cap M) = \Delta^{\FI}$~\cite{Bat18}. The Fine interior $\Delta^{\FI}$ of a lattice polytope $\Delta$
of dimension $d \geq 3$ may happen
to be strictly larger than the convex hull
$\conv(\Delta^{\circ} \cap M)$. The simplest famous example of such a situation
is due to M. Reid. Other similar examples based on hollow $3$-topes
can be found in Appendix \ref{appendixB}:

\begin{exnopar}[{\cite[Example 4.15]{Rei87}}] \label{G5} \emph{
Let $M \subseteq \QQ^4$ be $3$-dimensional affine lattice defined by
 \[ M \defeq \{ ( m_1,m_2, m_3,m_4 ) \in \ZZ^4 \, \mid \, \sum_{i=1}^4 m_i = 5, \;
\sum_{i=1}^4 i m_i \equiv 0\ \modb{5}\}.\]
Consider the $M$-lattice $3$-tope $\Delta \subseteq \MQ$ defined as
the convex hull of $4$ lattice points
\[ (5,0,0,0), \;(0,5,0,0), \;(0,0,5,0), \text{ and}\; (0,0,0,5) \in M. \]
Then $\conv(\Delta^{\circ} \cap M) = \emptyset$, but
$\Delta^{\FI}$ is the $3$-dimensional $M$-rational simplex
\[ \conv((2,1,1,1), (1,2,1,1), (1,1,2,1), (1,1,1,2)) \]
and $\Delta^{\FI} \cap M$ is empty.}
\end{exnopar}

In this paper, we are interested in
 lattice $d$-topes $\Delta \subseteq \MQ$ obtained as Newton polytopes
of Laurent polynomials $f_\Delta$ in $d$ variables $x_1, \ldots, x_d$, \emph{i.e.},
\[ f_\Delta({\bf x}) = \sum_{m \in \Delta \cap M} a_m {\bf x}^m, \]
where $a_m \in \CC$ are sufficiently general complex numbers. The importance
of Fine interior is explained by the following theorem~\cite{Rei87,Ish99,Bat18}:

\begin{theo} \label{mintheo}
Let $\mathcal Z_\Delta \subseteq \T^d$ be a non-degenerate affine hypersurface
 in the $d$-dimensional algebraic torus $\T^d$
 defined
by a Laurent polynomial $f_\Delta$ with Newton $d$-tope $\Delta$.
Then the following conditions are equivalent:
\begin{enumerate}[label=(\roman*)]
\item[\emph{(i)}] a smooth projective compactification $\mathcal V_\Delta$ of $\mathcal Z_\Delta$
has non-negative Kodaira dimension, \emph{i.e.}, $\kappa \geq 0$;
\item[\emph{(ii)}] $\mathcal Z_\Delta$ is birational to a minimal model
${\mathcal S}_\Delta$ with abundance;
\item[\emph{(iii)}] the Fine interior $\Delta^{\FI}$ of $\Delta$ is non-empty.
\end{enumerate}
\end{theo}

\begin{rem}
By well known results
of Khovanskii~\cite{Kho78}, one has
vanishing of the cohomology groups
 \[ h^{i}({\mathcal O}_{\mathcal V_\Delta}) = 0 \;\; (1 \leq i \leq d-2) \]
 and the equation $h^{d-1}({\mathcal O}_{\mathcal V_\Delta}) =
| \Delta^{\circ} \cap M|$. The numbers $h^{i}({\mathcal O}_{\mathcal V_\Delta})$ are
birational invariants of $\mathcal Z_\Delta$, because they do not depend
on a smooth projective compactification $\mathcal V_\Delta$ of $\mathcal Z_\Delta$. In particular,
the number
$| \Delta^{\circ} \cap M|$
is the geometric genus $p_g$ of the affine hypersurface $\mathcal Z_\Delta \subseteq \T^d$.
\end{rem}

Smooth projective compactifications of non-degenerate hypersurfaces
in $\T^d$ can be obtained using the theory of toric varieties~\cite{Kho78}.

Let $\Delta \subseteq \MQ$
be a lattice $d$-tope.
We consider
 the {\em normal fan} $\Sigma^{\Delta}$ \index{fan!normal} of $\Delta$
in the dual space $\NQ$, \emph{i.e.},
\label{normalfan}
$\Sigma^{\Delta}
\defeq \left\{\sigma^{\theta} \, \middle\vert \, \theta \preceq \Delta \right\}$,
where
$\sigma^{\theta}$ is the cone generated by all inward-pointing facet
normals of facets containing the face
$\theta \preceq \Delta$ of $\Delta$. One has $\dim(\sigma^{\theta}) +
\dim(\theta) =d$ for any face $\theta \preceq \Delta$.
We denote by $X_\Delta$ the normal projective toric variety constructed via the
normal fan $\Sigma^{\Delta}$.
In particular, the above function $ {\ord}_\Delta\; :\; \NQ \to \QQ$ is a piecewise
linear function with respect to this fan defining an ample
Cartier divisor on $X_\Delta$.
In particular,
the cones $\sigma^\theta \in \Sigma^\Delta$ are be defined
as
$$\sigma^{\theta}
= \left\{y \in \NQ\, \middle\vert \,
{\ord}_\Delta(y) = \left<x,y \right> \; \text{ for all } x \in \theta
\right\}. $$

\begin{rem}
Using the normal fan $\Sigma^\Delta$, one can compute the fundamental
group $\pi_1(\mathcal V_\Delta)$ of a smooth projective birational model $\mathcal V_\Delta$ of a non-degenerate
affine hypersurface
$\mathcal Z_\Delta$ (given as in Theorem \ref{mintheo}).
The fundamental group $\pi_1(\mathcal V_\Delta)$ does not depend on the choice of the
smooth birational model and it
is isomorphic to the quotient
of the lattice $N$ modulo the sublattice $N'$ generated by all lattice points in
$(d-1)$-dimensional cones $\sigma^\theta$
of the normal fan $\Sigma^\Delta$~\cite{BK06}.
\end{rem}

\begin{ex}
The minimal model ${\mathcal S_\Delta}$ of a non-degenerate affine
surface $\mathcal Z_\Delta$ defined by a
Laurent polynomial with the Newton polytope $\Delta$ from Example \ref{G5} is a
 {\em Godeaux surface}. It is a surface of general type
with $p_g = q=0$, $K^2 =1$, and $\pi_1({\mathcal S_\Delta}) \cong \ZZ/5 \ZZ$.
\end{ex}

\begin{dfn}
A lattice $d$-tope
$\Delta$ is called {\em canonical Fano $d$-tope} if $| \Delta^{\circ} \cap M| =1$. Up to a shift by a lattice vector, we will assume without
loss of generality that $0 \in M$ is the
single lattice point in the interior $\Delta^{\circ}$ of the canonical Fano
$d$-tope $\Delta$, \emph{i.e.},
$\Delta^{\circ} \cap M = \{ 0 \}$.
\end{dfn}

All canonical Fano $3$-topes have been classified~\cite{Kas10}. There exists exactly~\canonicalFano\ canonical Fano $3$-topes $\Delta$.
The aim of this paper is to present computational results of their
Fine interiors $\Delta^{\FI}$ and some related combinatorial invariants.
These data are important for computing minimal
smooth projective surfaces
$\mathcal S_\Delta$ with $p_g=1$ and $q =0$ which are birational to affine
non-degenerate hypersurfaces $\mathcal Z_\Delta \subseteq \T^3 \cong (\CC^\times)^3$.

The simplest description of the minimal surface $\mathcal S_\Delta$ has been
 obtained when $\Delta$ is a reflexive $3$-tope~\cite{Bat94}.

 \begin{dfn}
 A $d$-dimensional lattice
polytope $\Delta \subseteq \MQ$
containing the origin
$0 \in M$ in its interior
 is called {\em reflexive} if the dual polytope
\[ \Delta^* \defeq \{ y \in N\, \vert \, \langle x, y \rangle \geq -1 \; \text{ for all } x \in \Delta \} \subseteq \NQ\]
is a lattice polytope.
 \end{dfn}

There exists exactly $4,\!319$ reflexive
$3$-topes classified by Kreuzer and Skarke~\cite{KS98} and they form
a small subset in the list of all~\canonicalFano\ canonical Fano $3$-topes~\cite{Kas10}.
Reflexive $4$-topes are also classified by Kreuzer and Skarke~\cite{KS00}.
There exist $473,\!800,\!776$ reflexive $4$-topes, but the complete list
of all canonical Fano $4$-topes is unknown and expected
to be much bigger.

 If $\Delta$ is a reflexive
$d$-tope, then $X_\Delta$ is a Gorenstein toric Fano $d$-fold and
the Zariski closure $\overline{Z}_\Delta$ in $X_\Delta$
is a Gorenstein Calabi-Yau $(d-1)$-fold. If $d =3$, then $\overline{Z}_\Delta$ is a
$K3$-surface with at worst finitely many
Du Val singularities of type $A_k$.
The minimal
surface $\mathcal S_\Delta$ is a smooth $K3$-surface which is obtained as the
minimal (crepant)
desingularization of $\overline{Z}_\Delta$~\cite{Bat94}.

One motivation for the present paper is due to Corti and Golyshev, who have found
$9$ interesting examples of canonical Fano $3$-simplices $\Delta$ such
that the affine surfaces $\mathcal Z_\Delta$ are birational
to elliptic surfaces of Kodaira dimension $\kappa =1$~\cite{CG11}.

The computation of the Fine interior $\Delta^{\FI}$
for all canonical Fano $3$-topes
$\Delta \subseteq \MQ$ has shown that the dimension of the Fine interior
$\Delta^{\FI}$ has only three values: $0$, $1$,
and $3$. It is rather surprising
that there are no canonical Fano $3$-topes $\Delta$
with $\dim( \Delta^{\FI}) = 2$.

The condition $\dim (\Delta^{\FI}) =0$ holds
 if and only if $\Delta^{\FI}$ equals the lattice
point $0 \in M$.
There exist exactly $665,\!599$ canonical Fano
$3$-topes with $\Delta^{\FI} = \{0 \}$,
where $0 \in M$ is the only interior lattice point of $\Delta$.
These polytopes are characterized in~\cite[Proposition 3.4]{Bat18}
by the condition that $0 \in N$ is an interior lattice point of the $3$-dimensional
lattice polytope
\[ [\Delta^*]\defeq \conv( \Delta^* \cap N).\]

\begin{rem}
If $\Delta$ is a canonical Fano $3$-tope, then $\Delta^{\FI} = \{0 \}$
 if and only if the non-degenerate affine surface $\mathcal Z_\Delta$ is
birational to a $K3$-surface~\cite[Theorem 2.26]{Bat18}.
\end{rem}

The case $\dim( \Delta^{\FI}) =1$ splits in two subcases.
There exists exactly $20$
canonical Fano $3$-topes $\Delta$ such that $0 \in M$
is the midpoint of the Fine interior
$\Delta^{\FI}$. Therefore, we call this Fine interior {\em symmetric}. Canonical
Fano $3$-topes with $1$-dimensional symmetric Fine interior are characterized
by the condition that $[\Delta^*]$ is a $2$-dimensional
reflexive polytope.
The Fine interior of the remaining $9,\!020 $ canonical Fano $3$-topes with
$\dim( \Delta^{\FI}) =1$ contains $0 \in M$ as a
vertex. Therefore, we call this Fine interior {\em asymmetric}.
Canonical Fano $3$-topes with $1$-dimensional asymmetric Fine interior are
combinatorially characterized by the condition that $0 \in N$ is contained in the
relative interior of a facet
$\Theta \preceq [\Delta^*]$
of the lattice $3$-tope $[\Delta^*]$. The minimal surfaces $\mathcal S_\Delta$ corresponding
to canonical Fano $3$-topes with $1$-dimensional Fine interior (symmetric and
asymmetric) are elliptic surfaces of Kodaira dimension $\kappa=1$.

There exist exactly $49$ canonical Fano $3$-topes with $\dim (\Delta^{\FI}) =3$. These
polytopes are characterized by the condition that $0 \in N$ is a vertex
of the $3$-dimensional lattice polytope $[\Delta^*]$.
The surfaces $\mathcal S_\Delta$ corresponding
to canonical Fano $3$-topes $\Delta$
with $3$-dimensional Fine interior $\Delta^{\FI}$ are of general type
(\emph{i.e.}, $\mathcal S_\Delta$ has
maximal Kodaira dimension $\kappa=\dim(\mathcal S_\Delta)=2$).

\begin{rem}
The Fine interior computations were done using
\[ \Delta^{\FI}
= \bigcap_{\theta \preceq \Delta} \; \bigcap_{n \in \calh(\sigma^{\theta})}
\left\{ x \in \MQ \, \middle \vert \, \langle x,n \rangle
\geq {\ord}_\Delta(n) +1 \right\}, \]
where $\calh(\sigma^{\theta})$ denotes the set of all irreducible elements in the
monoid $\sigma^{\theta} \cap N$. It is the minimal generating set of the monoid
$\sigma^{\theta} \cap N$ and is called {\em Hilbert basis} of $\sigma^{\theta} \cap N$.
\end{rem}

In the next sections we consider examples and
discuss additional properties
of canonical Fano $3$-topes $\Delta$
in dependence of their Fine interiors $\Delta^{\FI}$.
All computations were done using the
\href{http://www.grdb.co.uk}{Graded Ring Database}\footnotemark[1], including the data of all~\canonicalFano\ canonical Fano $3$-topes and \href{http://magma.maths.usyd.edu.au/magma/}{MAGMA}\footnotemark[2].
Therefore, all statements have been checked by computer calculations.
The canonical Fano $3$-topes used as examples in this chapter appear with an
\href{http://www.grdb.co.uk/forms/toricf3c}{ID} that is the example's ID in the Graded Ring Database.\footnotemark[3]

\footnotetext[1]{\url{http://www.grdb.co.uk}}
\footnotetext[2]{\url{http://magma.maths.usyd.edu.au/magma/}}
\footnotetext[3]{\url{http://www.grdb.co.uk/forms/toricf3c}}

\section{Almost reflexive polytopes of dimension $3$ and $4$}

\begin{dfn}
A canonical Fano $d$-tope $\Delta \subseteq \MQ$ is called {\em almost reflexive}
if the convex hull of all $N$-lattice points in the dual polytope $\Delta^*$ is
reflexive.
\end{dfn}

It is {easy} to show the following statement:

\begin{prop}
If a canonical Fano $d$-tope $\Delta$ is almost reflexive, then
\[ \Delta^{FI} = \{0 \}. \]
\end{prop}

\begin{proof}
If $[ \Delta^*]$ is reflexive, then $\Delta = (\Delta^*)^*$ is contained in the
dual reflexive polytope $[\Delta^*]^*$. Therefore, the Fine interior of
$\Delta$ is contained in the Fine interior of the reflexive polytope $[\Delta^*]^*$
and $([\Delta^*]^*)^{\rm FI} =\{ 0 \}$. Thus, $\Delta^{\rm FI} = \{0 \}$.
\end{proof}

The converse statement is not true in general for $d \geq 5$, but
there exist many equivalent characterizations of reflexive and
almost reflexive
$d$-topes among canonical Fano $d$-topes if $d =3$ or $d=4$.

Let us recall some combinatorial
invariants of arbitrary lattice $d$-topes.

\begin{dfn}
The {\em Ehrhart power series} of an arbitrary
 lattice $d$-tope
$\Delta \subseteq \MQ$ is defined as
\[ P_\Delta(t) \defeq \sum_{k \geq 0}
\left\vert k \Delta \cap M\right\vert t^k, \]
where $|k \Delta \cap M|$ denotes the number of
lattice points in the $k$-th dilate $k \Delta$
of $\Delta$.
\end{dfn}

This Ehrhart series is a rational function of the form
\[ P_\Delta(t) = \frac{\psi_d (\Delta)t^ d + \cdots + \psi_1 (\Delta)t +
\psi_0 (\Delta)}{(1 - t)^{d+1}}, \]
where $\psi_i(\Delta)$ are non-negative integers
for all $0 \leq i \leq d$~\cite{Sta80} such
that
$\psi_0(\Delta) =1$ and $\psi_1(\Delta) = |\Delta \cap M| - d -1$.
Moreover,
 $\sum_{i=0}^d \psi_i(\Delta) = v(\Delta)$, where
 $v(\Delta) \defeq d! \cdot \vol(\Delta)$ denotes the
{\emph normalized volume} of $\Delta$.

One has the following characterization of reflexive $d$-topes:

\begin{propnopar}[{\cite[Theorem 4.6]{BR06}}]
A canonical Fano $d$-tope $\Delta$ is reflexive if and only if
\[ \psi_i(\Delta) = \psi_{d-i}(\Delta) \; \; (0 \leq i \leq d). \]
\end{propnopar}

The Ehrhart reciprocity
implies
that the power series
 \[ Q_\Delta(t) \defeq \sum_{k \geq 1}
\left\vert (k \Delta)^\circ \cap M\right\vert t^k \]
is a rational function
\[ Q_\Delta(t) = \frac{\varphi_{d+1}(\Delta)t^{d+1} + \cdots + \varphi_2 (\Delta)t +
\varphi_1 (\Delta)t + \varphi_0 (\Delta)}{(1 - t)^{d+1}}, \]
where $\varphi_0 (\Delta)=0$ and $\varphi_1 (\Delta) = |\Delta^\circ \cap M|$. Using Serre duality, one obtains
\[ \varphi_i(\Delta) = \psi_{d+1-i}(\Delta) \; \; (1 \leq i \leq d+1), \]
 \emph{i.e.}, in particular
 \[ \psi_d(\Delta)= \varphi_1(\Delta) = |\Delta^\circ \cap M|\]
 and
\[ \psi_{d-1}(\Delta) =\varphi_2(\Delta) = |2 \Delta^\circ \cap M| -
(d+1) |\Delta^\circ \cap M| \]~\cite[Section 4, 5.11]{DK86}.
Therefore, the lattice $d$-tope $\Delta$ is a canonical Fano $d$-tope
if and only
if $\psi_d(\Delta) =1$.
Moreover,
\[ \psi_{d-1}(\Delta) = |(2 \Delta)^\circ \cap M| - (d+1) \]
if $\Delta$ is a canonical Fano $d$-tope.

Applying the above equations, one
immediately obtains
the following criterion for reflexivity of canonical Fano $d$-topes
in the case $d =3, 4$:

\begin{prop} \label{3-4-ref}
Let $\Delta \subseteq \MQ$ be a canonical Fano $d$-tope with $d \in \{ 3,4\}$.
Then for $d =3$, one has
\[ P_\Delta(t) = \frac{t^ 3 + (|(2 \Delta)^\circ \cap M| -4)t^2 + ( |\Delta \cap M| - 4)t
 + 1}{(1 - t)^{4}} \]
and for $d =4$, one obtains
\[ P_\Delta(t) = \frac{t^ 4 + (|(2 \Delta)^\circ \cap M| -5)t^3 + \psi_2(\Delta) t^2 +
( |\Delta \cap M| - 5)t
 + 1}{(1 - t)^{5}}. \]
In particular, $\Delta$ is reflexive if and only if
\[ \left\vert \Delta \cap M\right\vert=
\left\vert (2 \Delta)^\circ \cap M\right\vert. \]
\end{prop}

\begin{prop}
Let $\Delta \subseteq \MQ$ be a canonical Fano $d$-tope
with $d \in \{3,4\}$ such
that $0 \in N$ is an interior lattice point of $[\Delta^*]$. Then $ [\Delta^*]$
is reflexive, {\em i.e.}, $\Delta$ is almost reflexive.
\end{prop}

\begin{proof}
Let $n \in N$ be an interior lattice point of $[\Delta^*]$. Then
 $ \langle x, n \rangle \geq 0$ for all $x \in \Delta \cap M$ because
\[ \Delta^* = \{ y \in \NQ \; |\; \langle x, y \rangle \geq -1 \;\; \text{ for all }
x \in \Delta \} \]
and $ \langle x, n \rangle$ is an integer.
Since $0 \in \Delta^\circ \cap M$, $\MQ$ is the set of all non-negative $\QQ$-linear
combinations of all lattice points in $\Delta \cap M$. This implies
$ \langle x', n \rangle \geq 0$ for all $x' \in \MQ$, \emph{i.e.}, $n =0$.
Therefore,
$[\Delta^*]$ has only one interior lattice point $0 \in N$, \emph{i.e.},
$[\Delta^*]$ is a canonical Fano $d$-tope.

It is clear that
$[\Delta^*]$ is contained in the interior of $2[\Delta^*]$. Therefore,
we have $[\Delta^*] \cap N \subseteq (2[\Delta^*])^\circ \cap N$. On the other
hand, for any lattice point $n \in (2[\Delta^*])^\circ$,
$ \langle x, n \rangle > -2 $ for all $x \in \Delta \cap M$. Since
$ \langle x, n \rangle $ is an integer, $n \in \Delta^* \cap N$, \emph{i.e.},
\[ [\Delta^*] \cap N = (2[\Delta^*])^\circ \cap N . \]
Using Proposition \ref{3-4-ref}, $[\Delta^*]$ is reflexive.
\end{proof}

\begin{coro} \label{min-ref}
Let $\Delta \subseteq \MQ$ be a canonical Fano $d$-tope with $d \in \{3,4\}$
such that $0 \in N$ is an interior lattice point of $[\Delta^*]$. Then $[\Delta^*]^*$
is the
smallest $($referring to inclusion$)$
reflexive polytope containing $\Delta$.
\end{coro}

\begin{proof}
 Let $\Delta' \subseteq \MQ$ be a reflexive $d$-tope such that
 $\Delta \subseteq \Delta'$.
Then $(\Delta')^* \subseteq \Delta^*$. Since $(\Delta')^*$ is a lattice
polytope, it is contained in $[\Delta^*]$. Thus, $[\Delta^*]^*$ is contained
in $((\Delta')^*)^* = \Delta'$.
\end{proof}

\begin{rem}
If $\Delta$ is a reflexive $d$-tope, then $[2\Delta^\circ] = \Delta$.
 If $\Delta$
is a canonical Fano $d$-tope with $d \in \{3,4\}$ such that $\Delta^{\FI}=\{0\}$
and $\Delta$ is contained in a reflexive
$d$-tope $\Delta'$,
then $[2\Delta^\circ]$ is contained in
$[(2\Delta')^\circ] = \Delta'$. Therefore, $[2\Delta^\circ]$ is contained
in the smallest reflexive polytope $[\Delta^*]^*$ containing $\Delta$, \emph{i.e.},
\[ [2\Delta^\circ] \subseteq [\Delta^*]^*. \]
Computations showed that among all $665,\!599$ canonical Fano $3$-topes 
$\Delta$ with $\Delta^{\rm FI} = \{0 \}$ there exist exactly $211,\!941$ canonical Fano $3$-tops 
such that $[2\Delta^\circ]$ is reflexive.  For the remaining canonical Fano $3$-topes $\Delta$ the lattice 
$3$-topes $[2\Delta^\circ] $ are larger than $\Delta$, but are not equal to the reflexive hull 
$[\Delta^*]^*$.
\end{rem}

\begin{rem}
Let $\Delta$ be an almost reflexive $3$-tope. 
We denote by $\tau(\Delta)$
the lattice $d$-tope $[2\Delta^\circ]$ .  If  
$\tau(\Delta)$ is not reflexive, then it is almost reflexive and we can consider 
the larger lattice $d$-tope $\tau^2(\Delta) \defeq \tau(\tau(\Delta)) \subseteq  [\Delta^*]^*$. 
After at morst five steps, $\tau^k(\Delta)$ is equal to the reflexive 
hull $[\Delta^*]^*$ of $\Delta$.
\end{rem}

In dimension $4$, the situation is comparable:

\begin{ex}
Let $\Delta \subseteq \RR^4$ be the almost reflexive  $4$-tope defined
by the inequalities
$x_i \geq -1$ $(1 \leq i \leq 4)$,
 $x_1 \leq 2$, and $x_1 + x_2 + x_3 + x_4 \leq 1$. Then $\Delta^{FI} = \{0\}$ and
the smallest reflexive $4$-tope containing $\Delta$ is the $4$-simplex
$ [\Delta^*]^*$
defined by the inequalities $x_i \geq -1$ $(1 \leq i \leq 4)$
 and $x_1 + x_2 + x_3 + x_4 \leq 1$. It is easy to see that $\tau(\Delta)$ is not the reflexive 
 $4$-tope $[\Delta^*]^*$ because the vertex
$(4,-1,-1,-1) \in \vertex([\Delta^*]^*)$ is not in
$2\Delta^\circ$. 
However, $\tau^2(\Delta) = [\Delta^*]^*$. 
\end{ex}

\section{Canonical Fano $3$-topes with $\Delta^{\FI} = \{0\}$}

We note that the set of all reflexive $3$-topes forms a
rather small part of the set of all canonical Fano $3$-topes.
The majority of canonical Fano $3$-topes
belongs to the subset of almost reflexive $3$-topes.
The proof of the following statement is
based on the result of Skarke~\cite{Ska96}
and the explanations in the previous section:

\begin{prop} A canonical Fano $3$-tope $\Delta$ is almost reflexive
if one of the following equivalent conditions is satisfied:
\begin{enumerate}[label=(\roman*)]
\item[{\em (i)}] $\Delta^{\FI} =\{0\};$
\item[{\em (ii)}] $0 \in N$ is an interior lattice point of $[\Delta^* ]$;
\item[{\em (iii)}] $\Delta$ is contained in some reflexive $3$-tope;
\item[{\em (iv)}] $\tau^k(\Delta)$ is the reflexive $3$-tope $[\Delta^*]^*$ for some sufficiently large $k$ $(1 \leq k \leq 5)$;
\item[{\em (v)}] the lattice $3$-tope
$[2\Delta^\circ]$ has exactly one interior lattice point;
\item[{\em (vi)}] the non-degenerate affine hypersurface $\mathcal Z_\Delta$
defined by a Laurent
polynomial with Newton polytope $\Delta$ is birational to
a smooth $K3$-surface.
\end{enumerate}
\end{prop}

\begin{figure}
	\centering
	\begin{subfigure}{6.5cm}
	\fbox{
	\begin{minipage}[c][6cm]{6cm}
		\centering \includegraphics[width=6cm]{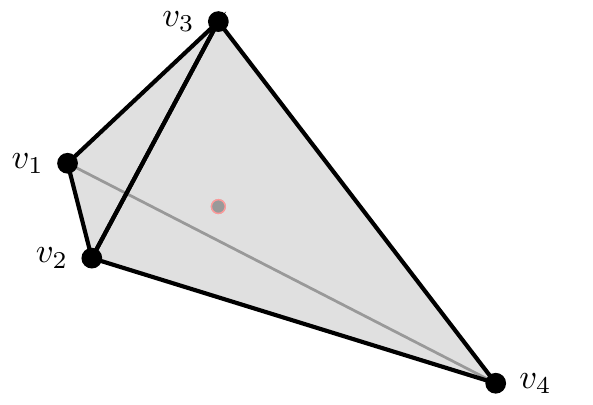}
	\end{minipage}
	}
	\caption{\href{http://www.grdb.co.uk/forms/toricf3c}{ID $547386$}.}
	\label{fig:P(1,1,1,1)}
	\end{subfigure}
\hspace{0.5cm}
	\begin{subfigure}{6.5cm}
	\fbox{
	\begin{minipage}[c][6cm]{6cm}
	\centering \includegraphics[width=6cm]{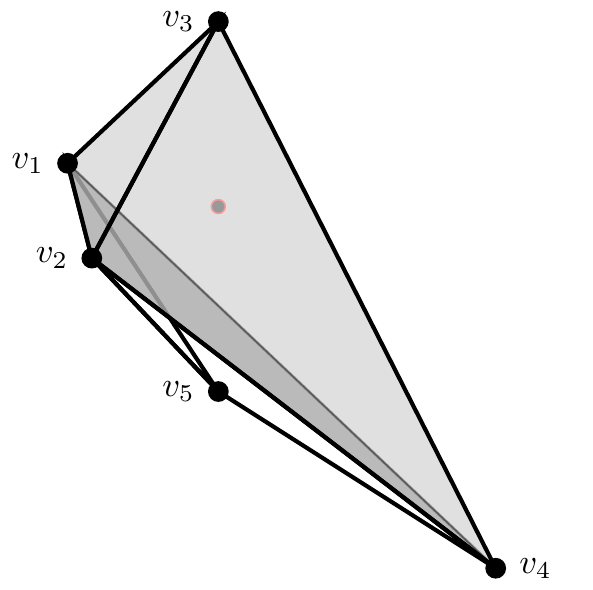}
	\end{minipage}
	}
	\caption{\href{http://www.grdb.co.uk/forms/toricf3c}{ID $547385$}.}
	\label{fig:P(1,1,1,2)}
	\end{subfigure}
	\\
	\caption[]
	{{\bf Canonical Fano {\boldmath $3$}-topes $\Delta$ with {\boldmath$ \Delta^{\FI}=\{0\} $}.} Shaded faces are occluded and
	the Fine interior $\{0\}$ is coloured grey with a red margin.
The whole polytope is the canonical hull
	$\Delta^{\can}$ as well as the reflexive hull $\Delta^{\refl}$ and the grey coloured polytope is $\Delta$.
	 	{\boldmath $(a)$} Reflexive polytope
	$\Delta= \conv(v_1,v_2,v_3,v_4)$ with
	$v_1=(1,0,0)$, $v_2=(0,1,0)$, $v_3=(0,0,1)$, $v_4= (-1,-1,-1)$, and
	$\Delta^{\refl}=\Delta^{\can}= \Delta$. All facets of $\Delta$ have lattice distance $1$ to the origin.
	{\boldmath $(b)$} Almost reflexive polytope
	$\Delta= \conv(v_1,v_2,v_3,v_4)$ with
	$v_1=(1,0,0)$, $v_2=(0,1,0)$, $v_3=(0,0,1)$, $v_4= (-1,-1,-2)$, and
	$\Delta^{\refl}=\Delta^{\can}= \conv(v_1,v_2,v_3,v_4,v_5)$ with
	$v_5=(0,0,-1)$ reflexive. The dark grey coloured facet of $\Delta$ has lattice distance $2$ and all other facets have lattice distance $1$ to the origin.
	}
	\label{fig:picturestrefct}
\end{figure}

Computations show that there exist exactly $665,\!599$ almost 
reflexive canonical Fano $3$-topes. The set of almost reflexive $3$-topes
includes all $4, \!319$ reflexive $3$-topes.
We have  shown that for any almost reflexive $3$-tope $\Delta$,
 the reflexive polytope $\Delta^{\refl}\defeq [\Delta^*]^*$ is
the smallest reflexive $3$-tope containing $\Delta$.
 We call $\Delta^{\refl}$
the {\em reflexive hull} of $\Delta$.
Thus we obtain a natural
 surjective map $ \Delta \mapsto \Delta^{\refl}$ from the
set of almost reflexive $3$-topes
 to the set of reflexive $3$-topes,
 which is the identity on the set of reflexive $3$-topes.
 The minimal surface $\mathcal S_\Delta$
is a $K3$-surface if and only if
 $\Delta$ is an almost reflexive $3$-tope. If $\Delta$ is an almost reflexive $3$-tope,
but not reflexive, then the minimal surface $\mathcal S_\Delta$
is a crepant desingularization of the Zariski closure of $\mathcal Z_\Delta$ in the Gorenstein
toric Fano threefold $X_{\Delta^{\refl}}$ defined by the reflexive hull of $\Delta$.

A generalization of the reflexive hull of almost reflexive $3$-topes
for arbitrary lattice $d$-topes with non-empty Fine interior
can be obtained using the notion of support of the Fine interior $\Delta^{\FI}$.

 \begin{dfn}
Let $\Delta \subseteq \MQ$ a lattice $d$-tope
with $\Delta^{\FI} \neq \emptyset$.
Then the set
\[ \supp(\Delta^{\FI}) \defeq \{ n \in N \, \vert \, \text{ there exists }x \in \Delta^{\FI} \; {\rm with }\;
\langle x, n \rangle = {\ord}_{\Delta}(n) +1 \} \]
is called {\em support of the Fine interior} of $\Delta$.
\end{dfn}

\begin{ex}
If $\Delta$ is a reflexive $d$-tope, then the support of the
Fine interior of $\Delta$ is the set of all non-zero lattice points
in $\Delta^* \cap N$.
\end{ex}

\begin{rem}
It is easy to show that one always has
\[ \Delta^{\FI} = \bigcap_{n \in \supp(\Delta^{\FI})} \{ x \in \MQ \, \vert \,
\langle x, n \rangle \geq {\ord}_{\Delta}(n) +1 \}. \]
\end{rem}

\begin{dfn}
Let $\Delta \subseteq \MQ$ a lattice
$d$-tope with $\Delta^{\FI} \neq \emptyset$. Then
the rational polytope
\[ \Delta^{\can}\defeq \bigcap_{n \in \supp(\Delta^{\FI})} \{ x \in \MQ \, \vert \,
\langle x, n \rangle \geq {\ord}_{\Delta}(n) \} \]
contains $\Delta$ and is called {\em canonical hull} of $\Delta$.
\end{dfn}

\begin{ex}
If $\Delta$ is an almost reflexive $3$-tope, then $\supp(\Delta^{\FI})$ is the
set $(\Delta^* \cap N) \setminus \{0\}$ of boundary lattice points
 in the reflexive
$3$-tope $[ \Delta^*]$ and the
canonical hull $\Delta^{\can}$ equals the reflexive hull
$\Delta^{\refl}$ of the
polytope $\Delta$, \emph{i.e.}, $\Delta^{\can}=\Delta^{\refl}=[\Delta^*]^*$.
In particular, in this case $\Delta^{\can}$ is always a
lattice $3$-tope.

There exist a
smooth projective toric variety $X_\Sigma$
defined by a fan $\Sigma$ whose $1$-dimensional cones
are generated
 by all lattice vectors from the finite set $\supp(\Delta^{\FI})$. Then
the minimal surface $\mathcal S_\Delta$ is a $K3$-surface which is
the Zariski closure of $\mathcal Z_\Delta$ in $ X_\Sigma$~\cite{Bat18}.
\end{ex}

\begin{ex}
Let us consider the (almost) reflexive canonical Fano $3$-tope
$\Delta = \conv(v_1,v_2,v_3,v_4) \subseteq \MQ$
(\href{http://www.grdb.co.uk/forms/toricf3c}{ID $547386$}, Figure \ref{fig:picturestrefct}\subref{fig:P(1,1,1,1)})
with vertices
$$v_1 \defeq (1,0,0), \,
v_2 \defeq (0,1,0), \,
v_3 \defeq(0,0,1), \, \text{and }
v_4 \defeq (-1,-1,-1)$$
and ${\Delta}^{\FI}=\{0\}$.
Moreover,
$$\Delta^{\refl}= \conv( 2\Delta^\circ \cap M) = \conv(\Delta \cap M) = \Delta$$
and
$${\Delta}^{\can} =[{\Delta}^*]^* = ({\Delta}^*)^*={\Delta}$$
because ${\Delta}$ is reflexive, {\emph i.e.}, $\Delta^{\refl}=\Delta^{\can}= \Delta$ reflexive
(Figure \ref{fig:picturestrefct}\subref{fig:P(1,1,1,1)}).
\end{ex}

\begin{ex}
Let us consider the almost reflexive canonical Fano $3$-tope
${\Delta} = \conv(v_1,v_2,v_3,v_4) \subseteq \MQ$
(\href{http://www.grdb.co.uk/forms/toricf3c}{ID $547385$},
Figure \ref{fig:picturestrefct}\subref{fig:P(1,1,1,2)})
with vertices
$$v_1 \defeq (1,0,0), \,
v_2 \defeq (0,1,0), \,
v_3 \defeq(0,0,1), \,\text{and }
v_4 \defeq (-1,-1,-2)$$
and ${\Delta}^{\FI}=\{0\}$.
Moreover,
$$\Delta^{\refl}= \conv((\Delta \cap M) \cup \{v_5\}) = \conv(v_1,v_2,v_3,v_4, v_5)$$
and
$${\Delta}^{\can} =[{\Delta}^*]^* = \conv(v_1,v_2,v_3,v_4, v_5)$$
with $v_5\defeq (0, 0, -1)$
because ${\Delta}$ is almost reflexive,
{\emph i.e.}, $\Delta^{\refl}=\Delta^{\can}= \Delta$ reflexive
(Figure \ref{fig:picturestrefct}\subref{fig:P(1,1,1,2)}).
\end{ex}

\section{Asymmetric Fine interior of dimension $1$} \label{sectionasymfine}

There exist exactly $9,\!020$ canonical Fano
$3$-topes $\Delta$ with $1$-dimensional Fine interior such that $0 \in N$ belongs to a facet
$\Theta \preceq [\Delta^*]$ of the lattice $3$-tope $[\Delta^*]$. This class of canonical Fano $3$-topes is characterized by
the property that the lattice $3$-tope $[2\Delta^\circ]$ has exactly $2$ interior
lattice points.

\begin{figure}
	\centering
	\begin{subfigure}{4cm}
	\centering
	\fbox{
	\begin{minipage}[c][3cm]{3cm}
		\centering \includegraphics[width=3cm]{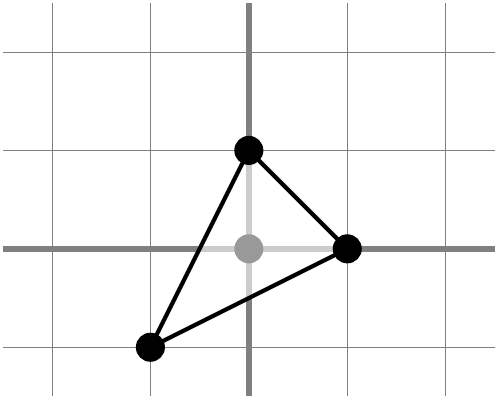}
	\end{minipage}
	}
		\caption{}
		\label{facet1}
	\end{subfigure}
	\hspace{1cm}
	\begin{subfigure}{4cm}
	\centering
	\fbox{
	\begin{minipage}[c][3cm]{3cm}
		\centering \includegraphics[width=3cm]{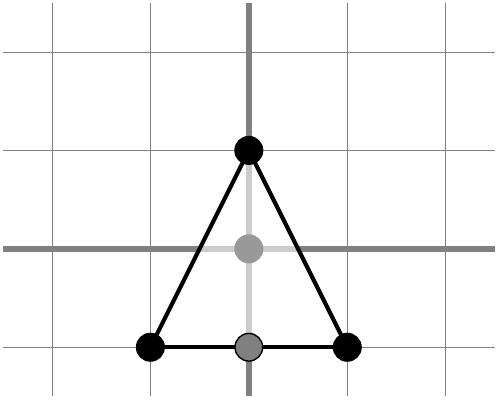}
	\end{minipage}
	}
		\caption{}
		\label{facet2}
	\end{subfigure}
	\hspace{1cm}
	\begin{subfigure}{4cm}
	\centering
	\fbox{
	\begin{minipage}[c][3cm]{3cm}
		\centering \includegraphics[width=3cm]{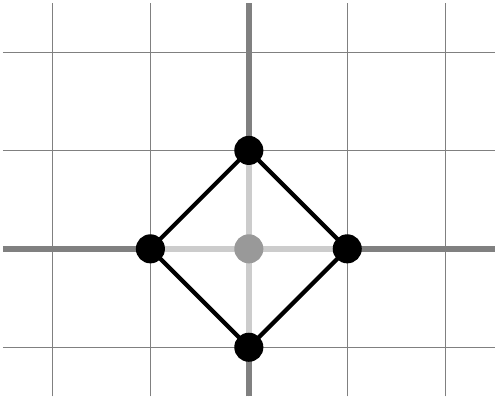}
	\end{minipage}
	}
		\caption{}
		\label{facet3}
	\end{subfigure}
	\caption[]
	{
	{\bf Reflexive Facets of {\boldmath $\Delta$} Containing {\boldmath$\pm v_\Delta$}.} Three types of reflexive facets $\theta_{\pm} \preceq \Delta$ of $\Delta$ containing $\pm v_\Delta$ for all $9,\!020+20$ canonical Fano $3$-topes $\Delta$ with
	$\dim(\Delta^{\FI})=1$. Vertices are coloured black, boundary points that are not vertices grey, and the origin light grey.
	{\boldmath $(a)$} $ \conv ((1,0),(0,1),(-1,-1))$.
 {\boldmath $(b)$} $\conv ((1,0),(-1,1),(-1,-1))$.
	{\boldmath $(c)$} $\conv ((\pm 1,0),(0,\pm 1))$.
}
	\label{fig:facet}
\end{figure}

The corresponding minimal surfaces
$\mathcal S_\Delta$ are
 {\em simply connected}
(\emph{i.e.}, have trivial fundamental group $\pi_1(\mathcal S_\Delta)$)
elliptic surfaces of Kodaira dimension $\kappa =1$.
We observed that the facet $\Theta \preceq [\Delta^*]$ is a reflexive $2$-tope
corresponding to one of the three types pictured in
Figure \ref{fig:facet}.
All $N$-lattice points on the boundary of $\Theta$ belong to
$\supp(\Delta^{\FI})$.
It was checked that for all these $3$-topes $\Delta$ the canonical hull
$\Delta^{\can}$ is again a lattice $3$-tope.
Moreover, the Fine interior $\Delta^{\FI}$ is contained in
the ray generated by the primitive lattice vector $v_\Delta \in M$
which is the primitive inward-pointing facet normal of
$\Theta$, \emph{i.e.}, $\langle x , y \rangle =0$ for all $x \in \Delta^{\FI}$, $y \in \Theta$.
 The lattice point $0 \in M$
is a vertex of $\Delta^{\FI}$. More precisely, one has
\[ \Delta^{\FI} = \conv( 0, \lambda v_\Delta ), \]
where $\lambda \in \{ 1/2, 2/3\} $.
The primitive lattice vector $v_\Delta$ is the unique interior lattice point
on a reflexive facet $\theta_{+} \preceq \Delta$ of $\Delta$ of one of the three possible types
pictured in Figure \ref{fig:facet}.
These three
reflexive polygons $\theta_{+}$
are characterized by the condition that the dual reflexive polygons $\theta^*_{+}$ are
obtained from $\theta_{+}$ (Figure \ref{fig:proj}) by enlarging
the lattice $\ZZ^2$ in the following ways:
$\ZZ^2 + \ZZ(1/3,2/3)$ (Figure \ref{fig:proj}\subref{proj1}),
$\ZZ^2 + \ZZ(1/2, 0)$ (Figure \ref{fig:proj}\subref{proj2}), and
$\ZZ^2 + \ZZ(1/2,1/2)$ (Figure \ref{fig:proj}\subref{proj3}).
Moreover, the reflexive facet $\theta_{+}$ of $\Delta$ is isomorphic to the facet
$\Theta$ of $[\Delta^*]$.
The projection
$M \to M/\ZZ v_\Delta$ of $\Delta$ or of $\theta_{+}$ along $v_\Delta$ is a reflexive
polygon of one of the three types pictured in
Figure \ref{fig:proj}, which is dual to $\theta_{+}$ and $\Theta$.
The lattice vector $v_\Delta$ defines a character
of the $3$-dimensional torus $\chi: \T^3 \to \CC^\times$.
For almost all $\alpha \in \CC^\times$, the
fiber $\chi^{-1} (\alpha)$ is an affine elliptic curve defined by a Laurent polynomial
with the reflexive Newton polytope $\Theta^* \cong \theta^*_{+}$
of one of the three types pictured in
Figure \ref{fig:proj} with the distribution shown in Table \ref{tbl:distr}.
So $\chi$ defines
birationally
an elliptic fibration.

 \begin{table}[h!]
 \centering
 \begin{tabular}{c|c|l|c|c}
 $\theta_{\pm}$ &$ \theta^*_{\pm}$ & enlarged lattice & $\#\Delta_{\text{asym}}$
 & $\# \Delta_{\text{sym}}$ \\ \toprule \midrule
Figure \ref{fig:facet}\subref{facet1} &Figure \ref{fig:proj}\subref{proj1}& $\ZZ^2 + \ZZ(1/3,2/3)$& $ 3,\!038$ & $7$ \\
Figure \ref{fig:facet}\subref{facet2} &Figure \ref{fig:proj}\subref{proj2}& $\ZZ^2 + \ZZ(1/2, 0)$ & $ 4,\!663 $ & $9$\\
Figure \ref{fig:facet}\subref{facet3} &Figure \ref{fig:proj}\subref{proj3}& $\ZZ^2 + \ZZ(1/2,1/2)$& $ 1,\!319$ & $4$
 \end{tabular}
 \vspace{1em}
 \caption[]
 {{\bf Distribution of the Reflexive Facets of {\boldmath $\Delta$} Containing {\boldmath$\pm v_\Delta$}.}
 Table contains: Type of the reflexive facet $\theta_{\pm}$ containing $\pm v_{\Delta}$, type of the dual reflexive facet $ \theta^*_{\pm}$, the enlarged lattice used to obtain $ \theta^*_{\pm}$ from ${\theta_{\pm}}$, 
 the number of canonical Fano $3$-topes $\Delta_{\text{asym}} \defeq \{\Delta \, \vert \, $1$\text{-dim. asym. } \Delta^{\FI}\}$, and 
 the number of canonical Fano $3$-topes $\Delta_{\text{sym}} \defeq \{\Delta \, \vert \, $1$\text{-dim. sym. } \Delta^{\FI}\}$
 with respect to the facet type of $\theta_{\pm}$ pictured in Figure~\ref{fig:facet}.
 }
 \label{tbl:distr}
 \end{table}

\begin{figure}
	\centering
	\begin{subfigure}{4cm}
	\centering
	\fbox{
	\begin{minipage}[c][3cm]{3cm}
		\centering \includegraphics[width=3cm]{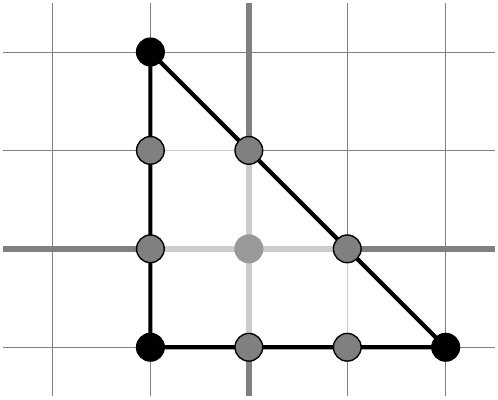}
	\end{minipage}
	}
		\caption{}
		\label{proj1}
	\end{subfigure}
	\hspace{1cm}
	\begin{subfigure}{4cm}
	\centering
	\fbox{
	\begin{minipage}[c][3cm]{3cm}
		\centering \includegraphics[width=3cm]{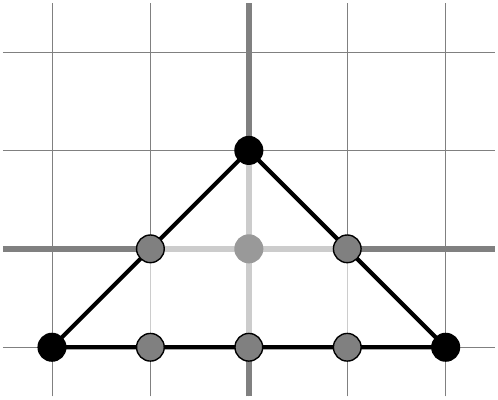}
	\end{minipage}
	}
		\caption{}
		\label{proj2}
	\end{subfigure}
	\hspace{1cm}
	\begin{subfigure}{4cm}
	\centering
	\fbox{
	\begin{minipage}[c][3cm]{3cm}
		\centering \includegraphics[width=3cm]{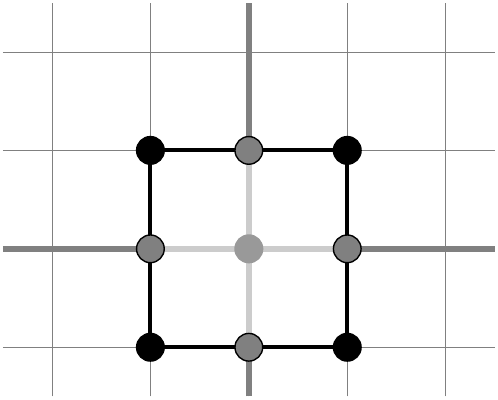}
	\end{minipage}
	}
		\caption{}
		\label{proj3}
	\end{subfigure}
	\caption[]
	{{\bf Reflexive Projection Polytopes.} Three types of reflexive polytopes
	obtained via a projection of $\Delta$ along $\pm v_\Delta$
	for all $9,\!020+20$ canonical Fano $3$-topes $\Delta$ with
	$\dim(\Delta^{\FI})=1$ Vertices are coloured black, boundary points that are not vertices grey, and the origin light grey. {\boldmath $(a)$} $\conv ((-1,2),(-1,-1),(2,-1))$. {\boldmath $(b)$} $\conv ((-2,-1),(0,1),$ $(2,-1))$. {\boldmath $(c)$} $\conv ((\pm 1,\pm 1))$.
}
	\label{fig:proj}
\end{figure}

\begin{ex}
Let $\Delta \subseteq \MQ$ be a canonical Fano $3$-tope given as the convex hull
$$v_1 \defeq (2,3,8), \,
v_2 \defeq (1,0,0), \,
v_3 \defeq (0,1,0), \,\text{and }
v_4 \defeq (-1,-1,-1)$$
(\href{http://www.grdb.co.uk/forms/toricf3c}{ID $547324$}, Figure \ref{fig:asymFine1}\subref{fig:547324}, Table \ref{tbl:corti} and \ref{tbl:corti1}).
 Then
\[ \Delta^{\FI} = {\conv}((0,0,0), (1/2,1/2,1)) = \conv( 0, 1/2 \cdot v_\Delta ), \]
where $v_\Delta = (1,1,2)$. One has $v_1 + 2v_2 + v_3 = 4 v_\Delta$. Therefore,
$v_\Delta$ is the interior lattice point of the reflexive facet $\theta_{+}$ of
$\Delta$
 (Figure \ref{fig:facet}\subref{facet2})
with vertices
$v_1,v_2, v_3$ and
the images $\overline{v}_1, \overline{v}_2, \overline{v}_3$ of $v_1,v_2,v_3$ in
$M/\ZZ v_\Delta$ are vertices of the dual reflexive triangle $\theta^*_{+}$
 (Figure \ref{fig:proj}\subref{proj2})
satisfying the relation
\[ \overline{v}_1 + 2 \overline{v}_2 + \overline{v}_3 =0. \]

To compute the canonical hull $\Delta^{\can}$ of $\Delta$, we obtain
$\supp(\Delta^{\FI}) = \{s_i \, \vert \, 1 \leq i \leq 18 \}$
with
$s_1 \defeq (-1, -1, 1)$,
$s_2 \defeq (-1, -1, 2)$,
$s_3 \defeq (-1, -1, 3)$,
$s_4 \defeq (-1, 0, 1)$,
$s_5 \defeq (-1, 0, 2)$,
$s_6 \defeq (-1, 1, 0)$,
$s_7 \defeq (-1, 1, 1)$,
$s_8 \defeq (-1, 2, 0)$,
$s_9 \defeq (-1, 3, -1)$,
$s_{10} \defeq (0, -1, 1)$,
$\ldots$,
$s_{18} \defeq (-2, -2, 1)$, which leads to
$$\Delta^{\can} =\conv(v_1,v_2,v_3,v_4, v_5)$$
with $v_5\defeq (0, 1,4)$ (Figure \ref{fig:asymFine1}\subref{fig:547324}).
\end{ex}

\begin{figure}
	\centering
	\begin{subfigure}{6.5cm}
	\fbox{
	\begin{minipage}[c][6cm]{6cm}
		\centering \includegraphics[height=6cm]{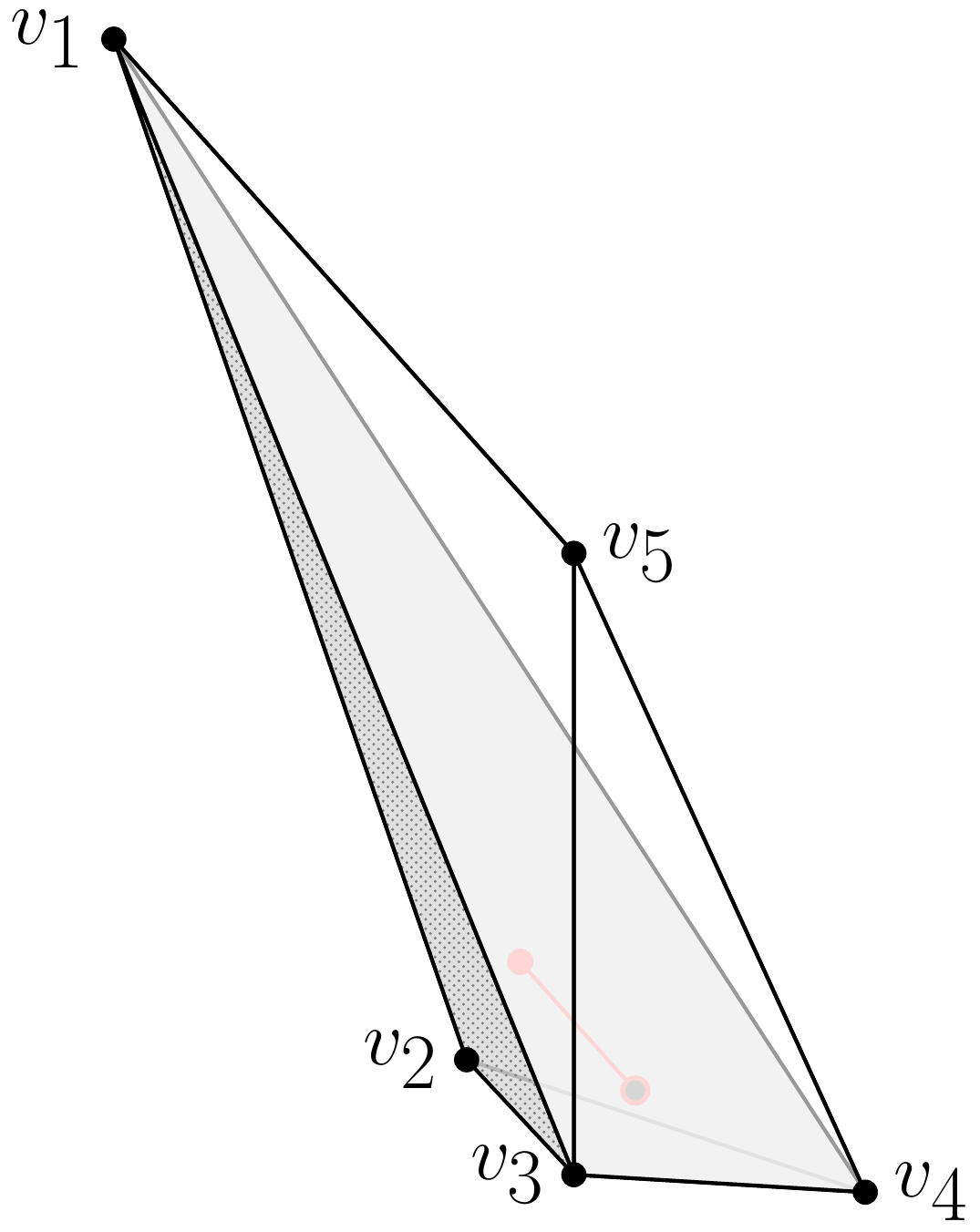}
	\end{minipage}
	}
	\caption{\href{http://www.grdb.co.uk/forms/toricf3c}{ID $547324$}}
	\label{fig:547324}
	\end{subfigure}
	\hspace{0.5cm}
	\begin{subfigure}{6.5cm}
	\fbox{
	\begin{minipage}[c][6cm]{6cm}
		\centering \includegraphics[height=6cm]{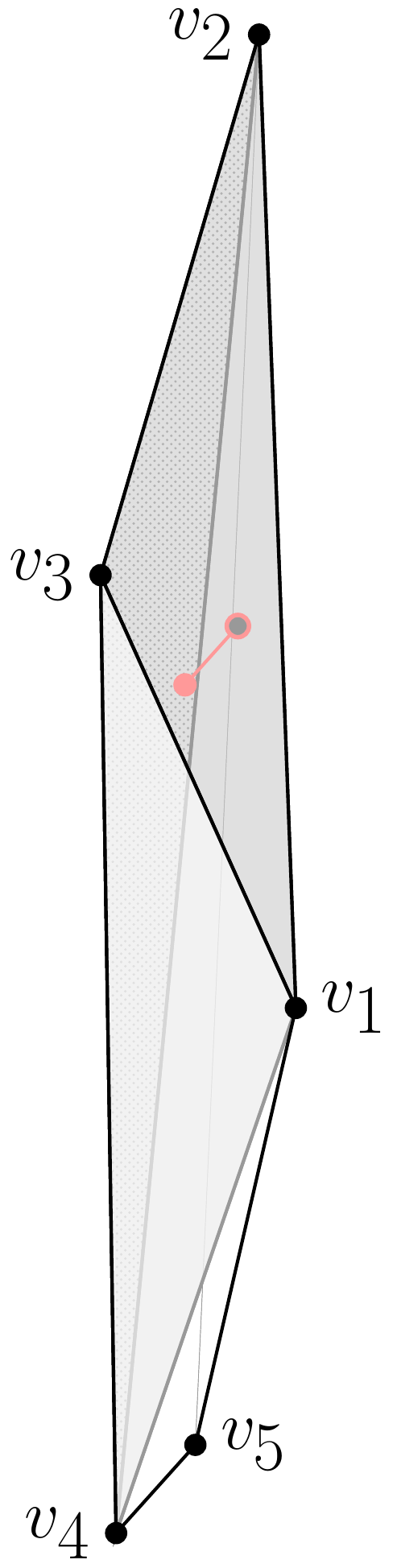}
	\end{minipage}
	}
	\caption{\href{http://www.grdb.co.uk/forms/toricf3c}{ID $547323$}}
	\label{fig:547323}
	\end{subfigure}

	\caption[]
	{{\bf Canonical Fano {\boldmath $3$}-topes with Asymmetric
Fine Interior of Dimension {\boldmath $1$.}} Shaded faces are occluded. The Fine interior is coloured red, the origin grey with a red margin, and the facet $\theta_{+}$ grey dotted.
	{\bf (a)} 	The whole polytope is $\Delta= \conv(v_1,v_2,v_3,v_4)$ with
	$v_1=(2,3,8)$, $v_2=(1,0,0)$, $v_3=(0,1,0)$, $v_4= (-1,-1,-1)$. Moreover,
	$\Delta^{\FI}= {\conv}((0,0,0), (1/2,1/2,1))$, $\theta_{+} = \conv(v_1,v_2,v_3)$, and
	$\Delta^{\can} = \conv(v_1, v_2, v_3, v_4, v_5)$ with $v_5 = (0, 1,4)$.
		{\bf (b)} 	The whole polytope is $\Delta= \conv(v_1,v_2,v_3,v_4)$ with
	$v_1=(-1,1,-2)$, $v_2=(1,-2,3)$, $v_3=(1,0,0)$, $v_4= (-2,5,-3)$. Moreover,
	$\Delta^{\FI}={\conv}((0,0,0), (0,2/3,0))$ and $\theta_{+}=\conv(v_2,v_3,v_4)$, and
	$\Delta^{\can} = \conv(v_1, v_2, v_3, v_4, v_5)$ with $v_5 = (-2, 4,-3)$.
	}
	\label{fig:asymFine1}
\end{figure}

\begin{ex}
Let $\Delta \subseteq \MQ$ be canonical Fano $3$-tope given as the convex hull
 $$v_1 \defeq (-1,1,-2), \,
v_2 \defeq (1,-2,3), \,
v_3 \defeq (1,0,0), \,\text{and }
v_4 \defeq (-2,5,-3)$$
(\href{http://www.grdb.co.uk/forms/toricf3c}{ID $547323$}, Figure \ref{fig:asymFine1}\subref{fig:547323}, Table \ref{tbl:corti} and \ref{tbl:corti1}).
 Then
\[ \Delta^{\FI} = {\conv}((0,0,0), (0,2/3,0)) = \conv( 0, 2/3 \cdot v_\Delta ), \]
where $v_\Delta = (0,1,0)$. One has $v_2 + v_3 + v_4 = 3 v_\Delta$. Therefore,
$v_\Delta$ is the interior lattice point of the reflexive facet $\theta_{+}$ of $\Delta$
 (Figure \ref{fig:facet}\subref{facet1})
with vertices
$v_2,v_3, v_4$ and
the images $\overline{v}_2, \overline{v}_3, \overline{v}_4$ of $v_2,v_3,v_4$ in
$M/\ZZ v_\Delta$ are vertices of the dual reflexive triangle $\theta^*_{+}$
 (Figure \ref{fig:proj}\subref{proj1})
satisfying the relation
\[ \overline{v}_2 + \overline{v}_3 + \overline{v}_4 =0. \]

To compute the canonical hull $\Delta^{\can}$ of $\Delta$, we obtain $\supp(\Delta^{\FI}) = \{s_i \, \vert \, 1 \leq i \leq 20 \}$
with
$s_1 \defeq (-3, -3, -2)$,
$s_2 \defeq (-1, 0, 0)$,
$s_3 \defeq (-1,0,1)$,
$s_4 \defeq (-1,1,1)$,
$s_5 \defeq (-1,2,2)$,
$s_6 \defeq (-1,3,2)$,
$s_7 \defeq (-1,4,3)$,
$s_8 \defeq (-1,6,4)$,
$s_9 \defeq (0,1,1)$,
$s_{10} \defeq (0,2,1)$,
$\ldots$,
$s_{20} \defeq (4,1,-1)$, which leads to
$$\Delta^{\can} =\conv(v_1,v_2,v_3,v_4, v_5)$$
with $v_5\defeq (-2,4,-3)$ (Figure \ref{fig:asymFine1}\subref{fig:547323}).
\end{ex}

\begin{rem}
The detailed information about a small selection of the $9,\!020$ canonical Fano $3$-topes with $\dim(\Delta^{\FI})=1$ and $0 \in \vertex(\Delta^{\FI})$ can be found in Appendix \ref{appendixA}.
To be precise, it is listed in Table \ref{tbl:corti}, \ref{tbl:corti0}, and \ref{tbl:corti1} and can be viewed in~\cite[Appendix \ref{appendixA}, Figure A$1$]{Sch18}.
\end{rem}

\section{Symmetric Fine interior of dimension $1$} \label{sectionsymfine}

There exist exactly $20$ canonical Fano $3$-topes $\Delta$ such that
$0$ is the center of $1$-dimensional Fine interior $\Delta^{\FI}$. In this case,
 $\mathcal S_\Delta$ is an elliptic surface of Kodaira dimension $\kappa = 1$ with non-trivial
 fundamental group $\pi_1(\mathcal S_\Delta)$ of order $2$ or $3$.
Computations show that one always has
 $\Delta =\Delta^{\can}$ and
\[ \Delta^{\FI} = \conv( -\lambda v_\Delta, \lambda v_\Delta ) \]
with $\lambda = \frac12 $ if and only if $\left \vert \pi_1(\mathcal S_\Delta) \right\vert =2$ and
\[ \Delta^{\FI} = \conv( -\mu v_\Delta, \mu v_\Delta ) \]
with $\mu = \frac{2}{3}$
if and only if $\left \vert \pi_1(\mathcal S_\Delta) \right\vert =3$.
The primitive lattice vectors $\pm v_\Delta$ are the two unique interior lattice points
in two reflexive facets $\theta_{\pm} \preceq \Delta$
of one of the three possible types
pictured in Figure \ref{fig:facet}.
The reflexive facets $\theta_{\pm}$ of $\Delta$ are isomorphic to the facet
$\Theta$ of $[\Delta^*]$.
The projections
$M \to M/\ZZ (\pm v_\Delta)$ of $\Delta$ or of $\theta_{\pm}$ along $\pm v_\Delta$ reveal a reflexive
polygon of one of the three types pictured in
Figure \ref{fig:proj}, which is dual to $\theta_{\pm}$ and $\Theta$.
The lattice vector $v_\Delta$ defines a character
of the $3$-dimensional torus $\chi: \T^3 \to \CC^\times$.
For almost all $\alpha \in \CC^\times$, the
fiber $\chi^{-1} (\alpha)$ is an affine elliptic curve defined by a Laurent polynomial
with the reflexive Newton polytope $\Theta^* \cong \theta^*_{\pm}$
of one of the three types pictured in
Figure \ref{fig:proj} with the distribution shown in Table \ref{tbl:distr}.
So $\chi$ defines birationally an elliptic fibration.
The vertex sets of $\Delta$ and these reflexive facets are related via $\vertex(\Delta) =
\vertex(\theta_{+}) \cup \vertex(\theta_{-})$. Moreover,
every edge of $\Delta$ is either an edge of $\theta_{+}$ or $\theta_{-}$ of these two facets or it is parallel to $v_\Delta$.

\begin{ex}
Let $\Delta \subseteq \MQ$ be canonical Fano $3$-tope given as the convex hull
$$v_1 \defeq (0,1,0), \,
v_2 \defeq (2,1,1), \,
v_3 \defeq (-2,-3,-5), \,\text{and }
v_4 \defeq (2,1,9)$$
(\href{http://www.grdb.co.uk/forms/toricf3c}{ID $547393$}, Figure \ref{fig:symFine1}\subref{fig:547393}, Table \ref{tbl:20} and \ref{tbl:200}).
Then
\[ \Delta^{\FI} = {\conv}((0,0,-1/2), (0,0, 1/2)) =
( -\lambda v_\Delta , \lambda v_\Delta) \]
with $\lambda = \frac12, $
where $v_\Delta = (0,0,1)$. One has
$2v_1 + v_3 + v_4 = 4 v_\Delta$ and
$2v_1 + v_2 + v_3 = 4(- v_\Delta)$. Therefore,
$v_\Delta$ is the interior lattice point of the reflexive facet $\theta_{+} = \theta_{134}$ of
$\Delta$ and $-v_\Delta$ is the interior lattice point of the
reflexive facet $\theta_{-}= \theta_{123}$ of
$\Delta$ (Figure \ref{fig:facet}\subref{facet2}).
The images
$\overline{v}_1, \overline{v}_3, \overline{v}_4$ of $v_1,v_3,v_4$ in $M/\ZZ v_\Delta$
and the images
$\overline{v}_1, \overline{v}_2, \overline{v}_3$ of $v_1,v_2,v_3$
in $M/\ZZ (-v_\Delta)$ are vertices of the dual reflexive triangle $\theta^*_{\pm}$
 (Figure \ref{fig:proj}\subref{proj2})
satisfying the relation
\[ 2 \overline{v}_1 + \overline{v}_3 + \overline{v}_4 =0 \]
and
\[ 2 \overline{v}_1 + \overline{v}_2 + \overline{v}_3 =0, \]
respectively.

To compute the canonical hull $\Delta^{\can}$ of $\Delta$, we obtain $\supp(\Delta^{\FI}) = \{s_i \, \vert \, 1 \leq i \leq 6 \}$
with
$s_1 \defeq (-1, -2, 2)$,
$s_2 \defeq (-1, 1, 0)$,
$s_3 \defeq (0, -1, 0)$,
$s_4 \defeq (1, -1, 0)$,
$s_5 \defeq (2, -1, 0)$, and
$s_6 \defeq (9, -2, -2)$,
which leads to
$\Delta^{\can} = \Delta.$
\end{ex}

\begin{ex}
Let $\Delta \subseteq \MQ$ be canonical Fano $3$-tope given as the convex hull
$$v_1 \defeq (-4,2,9), \,
v_2 \defeq (1,0,0), \,
v_3 \defeq (0,1,0), \,\text{and }
v_4 \defeq (7,-6,-18)$$
(\href{http://www.grdb.co.uk/forms/toricf3c}{ID $547409$}, Figure \ref{fig:symFine1}\subref{fig:547409}, Table \ref{tbl:20} and \ref{tbl:200}).
Then
\[ \Delta^{\FI} = {\conv}( (-2/3, 2/3, 2),(2/3, -2/3, -2)) =
( -\mu v_\Delta , \mu v_\Delta) \]
with $\mu = \frac23, $
where $v_\Delta = (1,-1,-3)$. One has
$v_1+ v_2 +v_3 = -3 v_\Delta$ and
$v_1 + v_3 + v_4 = -3(- v_\Delta)$. Therefore,
$v_\Delta$ is the interior lattice point of the reflexive facet $\theta_{+} = \theta_{123}$ of
$\Delta$ and $-v_\Delta$ is the interior lattice point of the
reflexive facet $\theta_{-}= \theta_{134}$ of
$\Delta$
 (Figure \ref{fig:facet}\subref{facet2}).
The images
$\overline{v}_1, \overline{v}_2, \overline{v}_3$ of $v_1,v_2,v_3$ in $M/\ZZ v_\Delta$
and the images
$\overline{v}_1, \overline{v}_3, \overline{v}_4$ of $v_1,v_3,v_4$
in $M/\ZZ (-v_\Delta)$ are vertices of the dual reflexive triangle $\theta^*_{\pm}$
 (Figure \ref{fig:proj}\subref{proj2})
satisfying the relation
\[ \overline{v}_1 + \overline{v}_2 + \overline{v}_3 =0, \]
and
\[ \overline{v}_1 + \overline{v}_3 + \overline{v}_4 =0, \]
respectively.

To compute the canonical hull $\Delta^{\can}$ of $\Delta$, we obtain $\supp(\Delta^{\FI}) = \{s_i \, \vert \, 1 \leq i \leq 5 \}$
with
$s_1 \defeq (-3,-3,-1)$,
$s_2 \defeq (-1, -1, 0)$,
$s_3 \defeq (-1,2, -1)$,
$s_4 \defeq (2, -1, 1)$, and
$s_5 \defeq (15,-3,7)$,
which leads to
$\Delta^{\can} = \Delta$.
\end{ex}

\begin{figure}
	\centering
	\begin{subfigure}{6.5cm}
	\fbox{
	\begin{minipage}[c][6cm]{6cm}
		\centering \includegraphics[height=6cm]{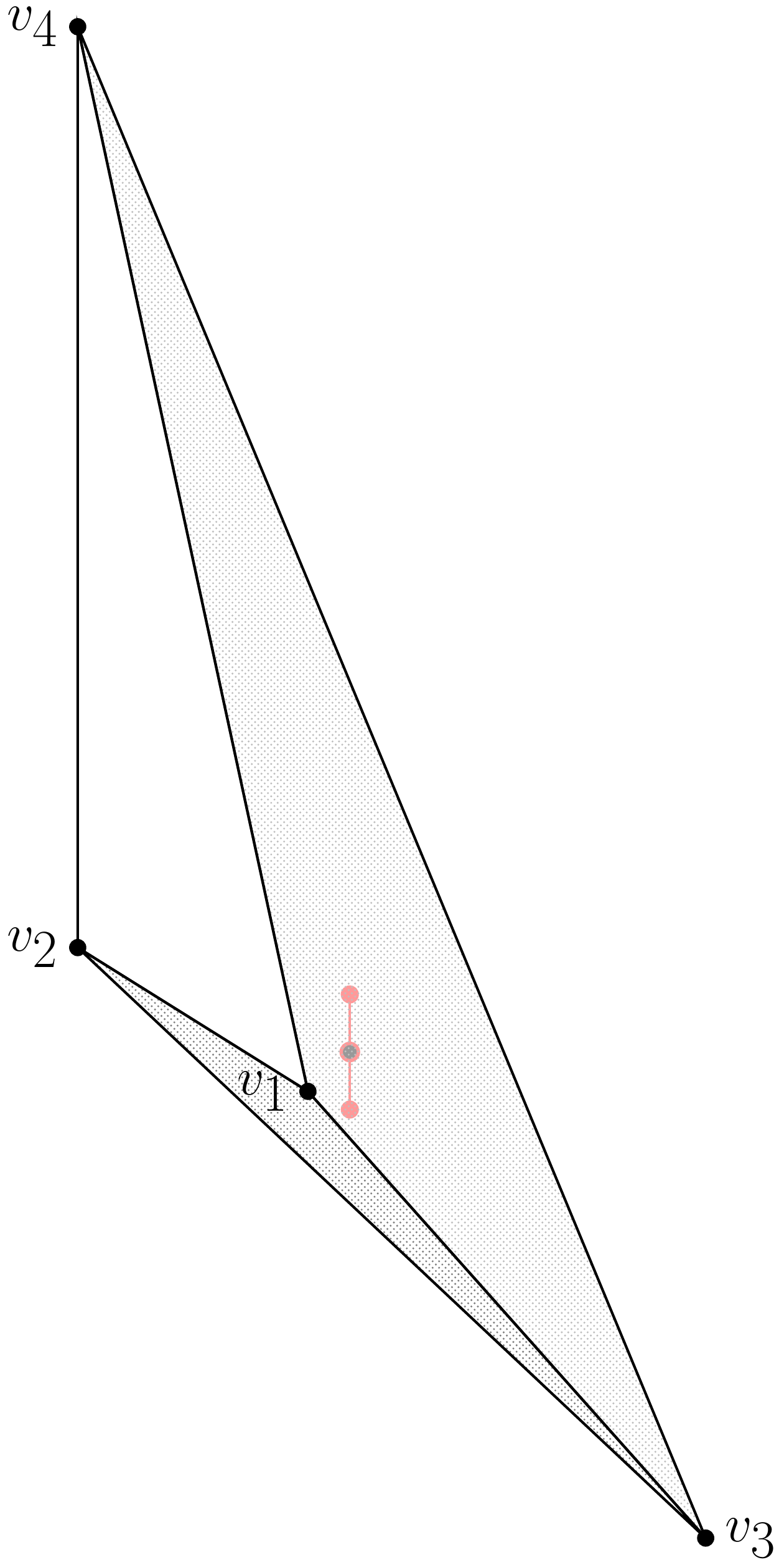}
	\end{minipage}
	}
	\caption{\href{http://www.grdb.co.uk/forms/toricf3c}{ID $547393$}}
	\label{fig:547393}
	\end{subfigure}
	\hspace{0.5cm}
	\begin{subfigure}{6.5cm}
	\fbox{
	\begin{minipage}[c][6cm]{6cm}
		\centering \includegraphics[height=6cm]{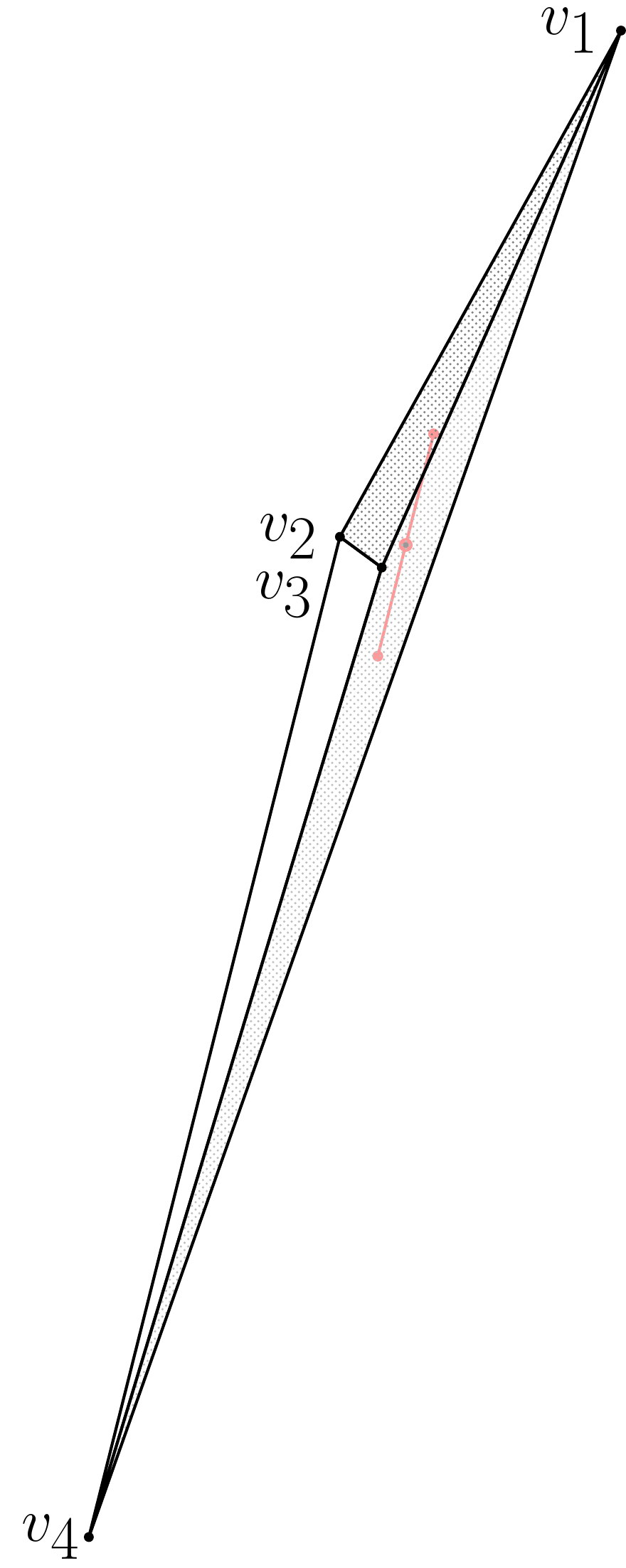}
	\end{minipage}
	}
	\caption{\href{http://www.grdb.co.uk/forms/toricf3c}{ID $547409$}}
	\label{fig:547409}
	\end{subfigure}

	\caption[]
	{ {\bf Canonical Fano {\boldmath $3$}-topes with Symmetric
Fine Interior of Dimension {\boldmath $1$.}} Shaded faces are occluded. The Fine interior is coloured red, the origin grey with a red margin, and the facets $\theta_{\pm}$ grey dotted.
	{\bf (a)} 	The whole polytope is $\Delta= \conv(v_1,v_2,v_3,v_4)$ with
	$v_1=(1,0,0)$, $v_2=(2,1,1)$, $v_3=(-2,-3,-5)$, $v_4= (2,1,9)$.
	Moreover,
	$\Delta^{\FI}= {\conv}((0,0,-1/2), (0,0,1/2))$, $\theta_{+} = \conv(v_1,v_3,v_4)$,
	$\theta_{-} = \conv(v_1,v_2,v_3)$, and $\Delta^{\can} = \Delta$.
		{\bf (b)} 	The whole polytope is $\Delta= \conv(v_1,v_2,v_3,v_4)$ with
	$v_1=(-4,2,9)$, $v_2=(1,0,0)$, $v_3=(0,1,0)$, $v_4= (7,-6,-18)$. Moreover,
	$\Delta^{\FI}= {\conv}((-2/3,2/3,2),(2/3,-2/3,-2))$, $\theta_{+} = \conv(v_1,v_2,v_3)$,
	$\theta_{-} = \conv(v_1,v_3,v_4)$, and $\Delta^{\can} = \Delta$.
	}
	\label{fig:symFine1}
\end{figure}

\begin{rem}
The detailed information about all $20$ canonical
Fano $3$-topes with $\dim(\Delta^{\FI})=1$ and $0 \in {(\Delta^{\FI})}^{\circ}$ can be found in
 Appendix \ref{appendixA}.
To be precise, it is listed in Table \ref{tbl:20} and \ref{tbl:200}
and can be viewed in~\cite[Appendix \ref{appendixA}, Figure A$2$]{Sch18}.
\end{rem}

\section{Fine interior of dimension $3$}

There exist $49$ canonical Fano $3$-topes $\Delta$ such that $\dim (\Delta^{\FI}) = 3$.
Exactly $3$ of these polytopes $\Delta$ define minimal surface $\mathcal S_\Delta$ with non-trivial
fundamental group of order $2$ and
$K^2 =2$.
For these $3$ polytopes one has $\Delta = \Delta^{\can}$. The surfaces
$\mathcal S_\Delta$ were investigated by Todorov~\cite{Tod81} as well as Catanese and Debarre~\cite{CD89}.

The remaining $46$ canonical Fano $3$-topes $\Delta$ define simply connected
minimal surfaces $\mathcal S_\Delta$ with $K^2 =1$.
These surfaces were investigated by Kanev~\cite{Kan77},
Catanese~\cite{Cat79}, and Todorov~\cite{Tod80}.
Among these $46$ canonical Fano $3$-topes there exist exactly $26$ polytopes
$\Delta$ such that $\Delta = \Delta^{\can}$.

\begin{exnopar}[{\cite{Kan77}}] {\em
Let $M \subseteq \QQ^4$ be the $3$-dimensional affine lattice defined by
 \[ M \defeq \{ ( m_1,m_2, m_3,m_4 ) \in \ZZ^4 \, \vert \, m_1 + m_2 + m_3 + 2 m_4 = 6, \;
 m_2 + 2 m_3 \equiv 0\ \modb{3}\} \]
and $\Delta' \subseteq \MQ$ be the convex hull of $4$ lattice points
\[ (6,0,0,0),\, (0,6,0,0), \,(0,0,6,0), \,\text{and }(0,0,0,3) \in M. \]
Then $(\Delta')^{\FI}$ is the $3$-dimensional rational simplex
\[ {\conv}((2,1,1,1), (1,2,1,1), (1,1,2,1), (1,1,1,3/2)) \]
and $(\Delta')^{\FI} \cap M = \{ (2,1,1,1) \}$.

The canonical Fano $3$-tope
$\Delta'$ is the Newton polytope of
the $\mu_3$-cyclic quotient $\overline{Z}_{\Delta'}$ of the projective surface of degree $6$ defined by the
polynomial $z_1^6 + z_2^6 + z_3^6 + z_4^3 =0$
in the weighted projective space
$\PP(1,1,1,2)$, where the cyclic group $\mu_3$ acts via $(z_1: z_2: z_3:z_4) \mapsto
(z_1 : \varepsilon z_2: \varepsilon^2 z_3: z_4)$. The single interior lattice point in $\Delta'$
corresponds to the monomial $z_1^2z_2z_3z_4$. The surface $\overline{Z}_{\Delta'}$ has
$3$ cyclic quotient singularities of type $A_2$. The minimal desingularization $\mathcal S_{\Delta'}$
of $\overline{Z}_{\Delta'}$ is simply connected
surface of general type with $K^2 =1$.

One can identify $\Delta'$ with
the canonical Fano $3$-simplex $\Delta$ given as the convex hull
$$v_1 \defeq (1,0,0), \,
v_2 \defeq (-2,-4,-5), \,
v_3 \defeq (1,2,4), \,\text{and }
v_4 \defeq (1,4,2)$$
(\href{http://www.grdb.co.uk/forms/toricf3c}{ID $547444$}, Figure \ref{fig:fine3}\subref{fig:547444}, Table \ref{tbl:491}, \ref{tbl:492}, and \ref{tbl:493}).
The primitive inward-pointing facet normals of the facets $\theta_{124}$, $\theta_{234}$, $\theta_{123}$, and $\theta_{134}\preceq \Delta$ of this simplex $\Delta$ are
$$n_1 \defeq (-2,-1,2), \,
n_2 \defeq (5,-1,-1), \,
n_3 \defeq (-1,2,-1), \,\text{and }
n_4 \defeq (-1,0,0),$$
respectively.
They satisfy the relation
\[ n_1 + n_2 + n_3 + 2n_4 = 0. \]

To compute the canonical hull $\Delta^{\can}$ of $\Delta$, we obtain $\supp(\Delta^{\FI}) = \{s_i \, \vert \, 1 \leq i \leq 6 \}$
with
$s_1 \defeq (-2,-1,2)$,
$s_2 \defeq (-1, 0,0)$,
$s_3 \defeq (-1,2,-1)$,
$s_4 \defeq (1,1,-1)$,
$s_5 \defeq (3,0,-1)$, and
$s_6 \defeq (5,-1,-1)$,
which leads to
$\Delta^{\can} = \Delta$.}
\end{exnopar}

\begin{figure}
	\centering
	\begin{subfigure}{6.5cm}
	\fbox{
	\begin{minipage}[c][6cm]{6cm}
		\centering \includegraphics[width=6cm]{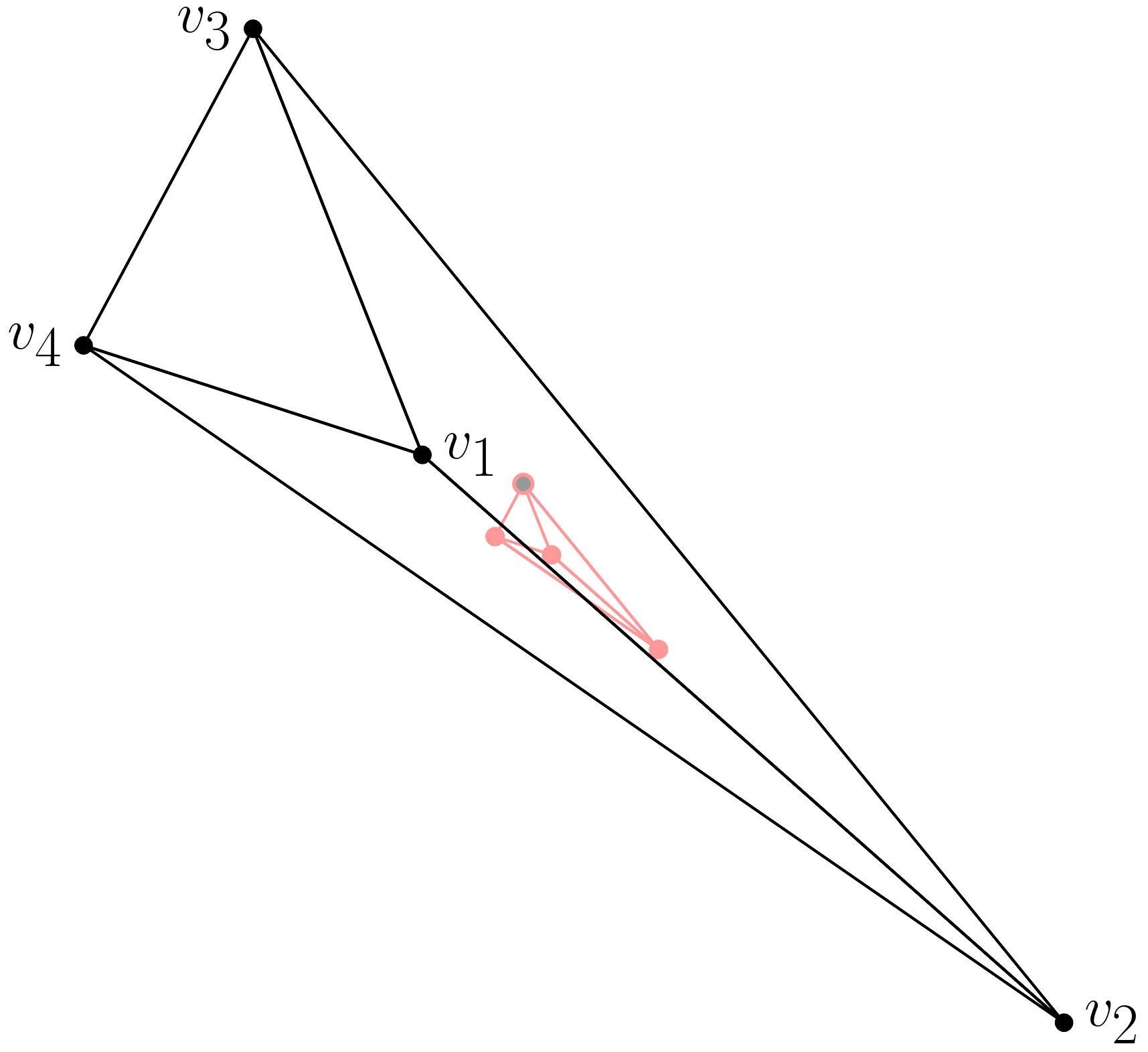}
	\end{minipage}
	}
	\caption{\href{http://www.grdb.co.uk/forms/toricf3c}{ID $547444$}}
	\label{fig:547444}
	\end{subfigure}
	\hspace{0.5cm}
	\begin{subfigure}{6.5cm}
	\fbox{
	\begin{minipage}[c][6cm]{6cm}
		\centering \includegraphics[width=6cm]{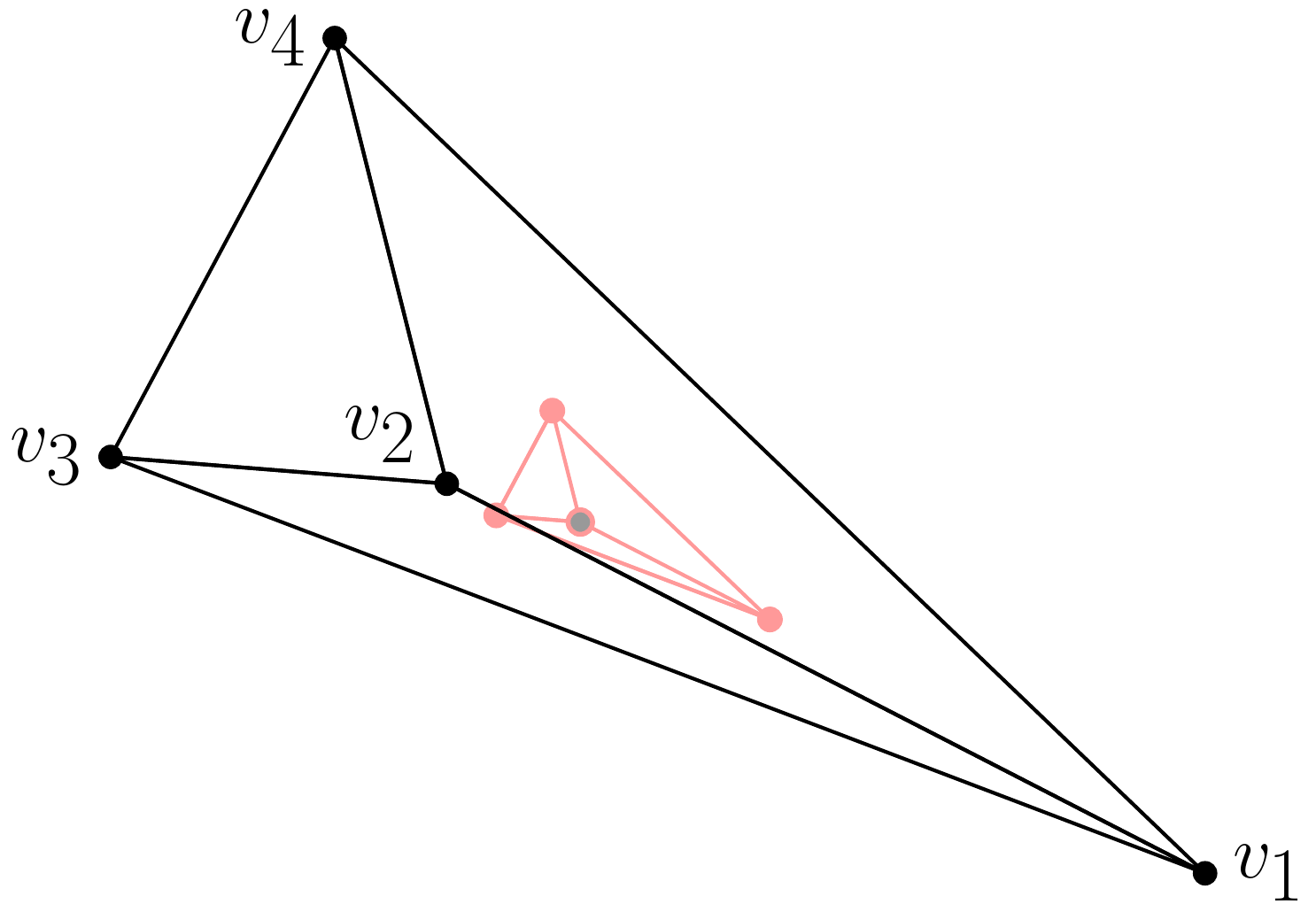}
	\end{minipage}
	}
	\caption{\href{http://www.grdb.co.uk/forms/toricf3c}{ID $547465$}}
	\label{fig:547465}
	\end{subfigure}

	\caption[]
	{{\bf Canonical Fano {\boldmath$3$}-topes with Fine Interior of Dimension {\boldmath$3$}.} Shaded faces are occluded. The Fine interior is coloured red and the origin grey with a red margin.
	{\bf (a)} 	The whole polytope is $\Delta= \conv(v_1,v_2,v_3,v_4)$ with
	$v_1= (1,0,0)$, $v_2 = (-2,-4,-5)$, $v_3 = (1,2,4)$, $v_4 = (1,4,2)$.
	Moreover,
	$\Delta^{\FI}= {\conv}( (0, 0, 0), (-1/2, -1, -3/2), (0, -1/3, -2/3), (0, 1/3, -1/3) )$ and
	$\Delta^{\can} = \Delta$.
		{\bf (b)} 	The whole polytope is $\Delta= \conv(v_1,v_2,v_3,v_4)$ with
	$v_1 = (-3,-2,-2)$, $v_2 = (1,0,0)$, $v_3 = (1,3,1)$, $v_4 = (1,1,3)$. Moreover,
	$\Delta^{\FI}= {\conv}( (0, 0, 0), (-1, -1/2, -1/2), (0, 3/4, 1/4), (0, 1/4, 3/4))$ and
	$\Delta^{\can} = \Delta$.
	}
	\label{fig:fine3}
\end{figure}

\begin{exnopar}[{\cite{Tod81}}] {\em
Let $M \subseteq \QQ^4$ be the $3$-dimensional affine lattice defined by
 \[ M \defeq \{ ( m_1,m_2, m_3,m_4 ) \in \ZZ^4 \, \vert \, m_1 + m_2 + 2m_3 + 2 m_4 = 8, \;
 3m_2 + m_3 +3 m_4 \equiv 0\ \modb{4}\} \]
and ${\Delta'} \subseteq \MQ$ be the convex hull of $4$ lattice points
\[ (8,0,0,0), \, (0,8,0,0), \,(0,0,4,0), \,\text{and }(0,0,0,4) \in M. \]
Then $({\Delta'})^{\FI}$ is the $3$-dimensional rational simplex
\[ {\conv}((3,1,1,1), (1,3,1,1), (1,1,2,1), (1,1,1,2)) \]
and $({\Delta'})^{\FI} \cap M = \{ (1,1,2,1) \}$.

The canonical Fano $3$-tope
${\Delta'}$ is the Newton polytope of
the $\mu_4$-cyclic quotient $\overline{Z}_{\Delta'}$ of the projective surface
of degree $8$ defined by the
polynomial $z_1^8 + z_2^8 + z_3^4 + z_4^4 =0$
in the weighted projective space
$\PP(1,1,2,2)$, where the cyclic group $\mu_4$ acts via $(z_1: z_2: z_3:z_4) \mapsto
(z_1 : i ^3z_2: i z_3: i ^3z_4)$. The single interior lattice point in this lattice simplex ${\Delta'}$
corresponds to the monomial $z_1z_2z^2_3z_4$. The projective surface $\overline{Z}_{\Delta'}$ has two Gorenstein cyclic quotient singularities of type $A_3$.
The minimal desingularization $\mathcal S_{\Delta'}$ of $\overline{Z}_{\Delta'}$ is a surface
of general type with $K^2$
and fundamental group $\pi_1(\mathcal S_\Delta)$ of order $2$.

One can identify ${\Delta'}$ with
the canonical Fano $3$-simplex $\Delta$ given as the convex hull
\[ v_1 \defeq (-3,-2,-2), \, v_2 \defeq (1,0,0), \,v_3\defeq(1,3,1), \,\text{and } v_4\defeq (1,1,3) \]
(\href{http://www.grdb.co.uk/forms/toricf3c}{ID $547465$}, Figure \ref{fig:fine3}\subref{fig:547465}, Table
 \ref{tbl:491}, \ref{tbl:492}, and \ref{tbl:493}).
The primitive inward-pointing facet normals of the facets $\theta_{123}$, $\theta_{124}$, $\theta_{234}$, $\theta_{134} \preceq \Delta$ of this simplex $\Delta$ are
\[ n_1 \defeq (-1,-1, 3), \, n_2 \defeq (-1,3,-1), \,n_3 \defeq (-1,0,0 ), \,\text{and }n_4 \defeq (2,-1,-1), \]
respectively.
They satisfy the relation
\[ n_1 + n_2 + 2n_3 + 2n_4 = 0. \]

To compute the canonical hull $\Delta^{\can}$ of $\Delta$, we obtain $\supp(\Delta^{\FI}) = \{s_i \, \vert \, 1 \leq i \leq 9 \}$
with
$s_1 \defeq (-1, -1, 3)$,
$s_2 \defeq (-1, 0,0)$,
$s_3 \defeq (-1,0,1)$,
$s_4 \defeq (-1,0,2)$,
$s_5 \defeq (-1,1,0)$,
$s_6 \defeq (-1,1,1)$,
$s_7 \defeq (-1,2,0)$,
$s_8 \defeq (-1,3,-1)$, and
$s_9 \defeq (2,-1,-1)$,
which leads to
$\Delta^{\can} = \Delta$.}
\end{exnopar}

\begin{rem}
The detailed information about all $49$ canonical Fano $3$-topes with $\dim(\Delta^{\FI})=3$ can be found in the Appendix \ref{appendixA}.
To be precise, it is listed in Table \ref{tbl:491}, \ref{tbl:492}, and \ref{tbl:493} and can be viewed in~\cite[Appendix \ref{appendixA}, Figure A$3$]{Sch18}.\end{rem}

\newpage

\appendix
\section{Computational data}
 \label{appendixA}

{In all the tables, the
canonical Fano $3$-topes $\Delta$ are given by
their \href{http://www.grdb.co.uk/forms/toricf3c}{ID}s used in the
\href{http://www.grdb.co.uk}{Graded Ring Database}.

\renewcommand{\thetable}{\Alph{section}.\arabic{table}}
\setcounter{table}{0}

 \hspace{0pt}
\vfill
 \begin{table}[h!]
 \centering
 \rotatebox{0}{%
 \scalebox{1}{
 \begin{tabular}{l|l|l|l|l}
 \multicolumn{1}{c|}{ID} &
 \multicolumn{1}{c|}{$\vertex(\Delta)$} &
 \multicolumn{1}{c|}{$\vertex(\Delta^{\FI})$} &
 \multicolumn{1}{c|}{$v_\Delta$} & \multicolumn{1}{c}{$(w_i)_{0 \leq i \leq 3}$} \\ \toprule \midrule
 $ 547324 $& $ (2,3,8), (1,0,0), (0,1,0), (-1,-1,-1) $&$0,1/2 \cdot v_{\Delta}$&$ (1,1,2) $& $(1,5,6,8)$\\
 $547323$ &$ (-1,1,-2), (1,-2,3), (1,0,0), (-2,5,-3) $& $0,2/3 \cdot v_{\Delta}$&$ (0,1,0) $& $(1,4,7,9)$\\
 $547311 $&$ (-1,4,2), (-1,-1,0), (0,0,-1), (2,0,1) $& $0,2/3 \cdot v_{\Delta}$&$ (0,1,1) $& $(2,5,8,9)$\\
 $547490 $&$ (1,2,4), (1,0,0), (1,-2,3), (-1,1,-2) $& $0,1/2 \cdot v_{\Delta}$&$ (0,1,0) $& $(1,5,8,14)$\\
 $547321 $&$ (1,-2,3), (0,1,0), (1,0,0), (-6,3,-8) $&$0,1/2 \cdot v_{\Delta}$ &$ (-1,1,-2)$ & $(3,7,8,10)$ \\
 $547305 $& $(0,1,0), (1,0,0), (1,2,4), (-4,-6,-7) $ &$0,2/3 \cdot v_{\Delta}$&$ (-1,-1,-1) $& $(4,7,9,10)$ \\
 $547526 $&$ (1,0,0), (0,1,0), (-2,1,5), (2,-4,-9) $& $0,2/3 \cdot v_{\Delta}$ &$ (1,-1,-3)$ & $(5,9,8,11)$\\
 $547454 $&$ (2,1,7), (1,0,0), (0,1,0), (-2,-3,-3) $& $0,1/2 \cdot v_{\Delta}$&$ (0,0,1) $& $(3,7,8,18)$\\
 $547446$ & $(0,1,1), (-6,7,-15), (1,-2,3), (1,0,0) $&$0,1/2 \cdot v_{\Delta}$&$ (-1,1,-2)$& $(5,8,9,22)$
 \end{tabular}
 }
 }
 \vspace{1em}
  \caption[]
 {{\bf {\boldmath $9$} Canonical Fano {\boldmath $3$}-topes with Asymmetric
Fine Interior of Dimension {\boldmath $1$}.} Table contains:
vertices $\vertex(\Delta)$ of $\Delta$,
vertices $\vertex(\Delta^{\FI})$ of the Fine interior
$\Delta^{\FI}$, unique primitive lattice point $v_{\Delta} \in \theta_+$
in the reflexive facet $\theta_+ \preceq \Delta$, and
weights $(w_i)_{0 \leq i \leq 3}$ of the weighted projective $3$-space $\PP(w_0,\ldots,w_3)$
appearing in~\cite{CG11}.}
 \label{tbl:corti}
 \end{table}

 \vfill
 \hspace{0pt}

 \hspace{0pt}
\vfill
 \begin{table}[h!]
 \centering
 \rotatebox{0}{%
 \scalebox{1}{
 \begin{tabular}{l|l|l|l}
 \multicolumn{1}{c|}{ID} &
 \multicolumn{1}{c|}{ $(n_{i})_{1 \leq i \leq 4}$} &
 \multicolumn{1}{c|}{$\vertex(\theta_{+})$} &
 \multicolumn{1}{c}{$n_{\theta_+}$}
 \\ \toprule \midrule
 $ 547324 $& $(-2, -2, 1), (-1, -1, 3), $& $ (2,3,8), (1,0,0), (0,1,0)$& $(-2,-2,1)$ \\
 & $(-1, 3, -1), (7, -3, -1)$& & \\
 $547323$ &$(-3, -3, -2), (-1, 0, 1),$ &$(1,-2,3), (1,0,0), (-2,5,-3)$&$(-3,-3,-2)$ \\
 & $ (-1, 6, 4), (17, 3, -5)$& & \\
 $547311 $& $(-1, -1, 1), (-1, 2, 1), $& $ (-1,4,2), (-1,-1,0), (2,0,1)$& $(1,2,-5)$ \\
 & $(1, 2, -5), (7, -2, 5)$& & \\
 $547490 $& $(-2, -2, 1), (-1, 0, 0), $& $ (1,2,4), (1,0,0), (-1,1,-2) $&$(-2, -2, 1)$ \\
 & $(-1, 6, 4), (23, 2, -8)$& & \\
 $547321 $& $(-3, -3, -2), (-2, -2, 1), $&$ (0,1,0), (1,0,0), (-6,3,-8) $& $(-2, -2, 1)$ \\
 & $(-1, 3, 2), (9, -5, -8)$& & \\
 $547305 $& $(-7, -7, 11), (-2, -2, 1), $ &$(0,1,0), (1,2,4), (-4,-6,-7)$&$(7, -3, -1)$ \\
 & $(-1, 2, -1), (7, -3, -1)$& & \\
 $547526 $& $(-5, -5, -2), (-3, -3, 1), $& $ (1,0,0), (0,1,0), (2,-4,-9)$&$ (-3,- 3, 1)$ \\
 & $(-1, 2, -1), (25, -8, 10)$& & \\
 $547454 $&$(-7, -7, 2), (-1, -1, 2),$& $ (2,1,7), (0,1,0), (-2,-3,-3)$& $(7, -2, -2)$ \\
 & $ (-1, 1, 0), (7, -2, -2)$& & \\
 $547446$ & $(-9, 21, 14), (-5, -3, -2),$&$(0,1,1), (-6,7,-15), (1,-2,3)$& $(9, 1, -3)$ \\
 &$ (-1, -1, 0), (9, 1, -3)$ & &
 \end{tabular}
 }
 }
 \vspace{1em}
 \caption[]
 {{\bf {\boldmath $9$} Canonical Fano {\boldmath $3$}-topes with Asymmetric
Fine Interior of Dimension {\boldmath $1$}.} Table contains:
primitive inward-pointing facet normals $(n_{i})_{1 \leq i \leq 4}$ of $\Delta$,
vertices $\vertex(\theta_{+})$ of the reflexive facet $\theta_{+}\preceq \Delta$, and
primitive inward-pointing facet normal $n_{\theta_+}$ of the reflexive facet $\theta_+ \preceq \Delta$.}
 \label{tbl:corti0}
 \end{table}
 }

 \newpage

 \renewcommand{\thetable}{\Alph{section}.\arabic{table}}
\setcounter{table}{2}

 \hspace{0pt}
\vfill
{
 \begin{table}[h!]
 \centering
\begin{threeparttable}
 \rotatebox{0}{%
 \scalebox{1}{
 \begin{tabular}{l|l|l|l}
 \multicolumn{1}{c|}{ID} &
 \multicolumn{1}{c|}{$\vertex(\Theta)$} &
 \multicolumn{1}{c|}{$\supp(\Delta^{\FI})$} &
 \multicolumn{1}{c}{$\vertex(\Delta^{\can})$}\\ \toprule \midrule
 $ 547324 $ &$(-1, 3, -1), (-1, -1, 1), (1, -1, 0)$ &$(-2, -2, 1),(-1, -1, 1),S_1$& $\vertex(\Delta), (0,1,4)$\\
 $ 547323 $ & $(-1, 0, 1), (-1, 0, 0), (2, 0, -1)$ &$(-3, -3, -2),(-1, 0, 0),S_2$& $ \vertex(\Delta), (-2,4,-3)$\\
 $547311 $& $(-1, -1, 1), (0, 1, -1), (1, 0, 0)$&$(-1, -1, 1),(-1, 0, 1),S_3$& $\vertex(\Delta), (-1, 2, 0)$\\
$ 547490 $&$(-1, 0, 0), (-1, 0, 1), (3, 0, -1)$ &$(-2, -2, 1),(-1, 0, 0),S_4$& $\vertex(\Delta), (1, -1, 4)$\\
 $ 547321 $& $(-1, -1, 0), (-1, 3, 2), (1, -1, -1)$&$(-2, -2, 1),(-1, -1, 0),$& $\vertex(\Delta), (1, 0, 1), $\\
 & & $(-1, 0, 0),(-1, 1, 1),$& $(0, -3, 4)$ \\
 & & $(-1, 3, 2),(0, -1, -1),$ &\\
 && $(1, -1, -1)$&\\
 $ 547305 $& $(-1, 2, -1), (1, -1, 0), (0, -1, 1)$ &$(-1, -1, 1),(-1, 0, 0),$& $\vertex(\Delta), (0, -2, -3), $\\
 & & $(-1, 2, -1),(0, -1, 1),$ & $(1, 2, 2)$\\
 & & $(1, -1, 0),(7, -3, -1)$ & \\
 $ 547526 $&$(-1, -1, 0), (-1, 2, -1), (2, -1, 1)$ &$(-3, -3, 1),(-1, -1, 0),$& $\vertex(\Delta)\setminus \{(-2,1,5)\}, $\\
 & & $(-1, 2, -1),(0, -1, 0),$& $ (0, 1, 3), (-3, 1, 6) $ \\
 & & $(2, -1, 1)$&\\
 $ 547454 $&$(-1, 1, 0), (0, -1, 0), (2, -1, 0)$ &$(-1, -1, 1),(-1, -1, 2),S_5$& $\vertex(\Delta), (2, 1, 2)$\\
 $547446 $& $(-1, -1, 0), (0, 2, 1), (2, 0, -1)$&$(-1, -1, 0),(-1, 0, 0),$& $\vertex(\Delta), (1, 0, -1), $ \\
 & & $(0, 2, 1),(1, 1, 0),$& $(1, 0, 3)$\\
 & & $(2, 0, -1),(9, 1, -3)$&\\
 \end{tabular}
 }
 }
\end{threeparttable}
\vspace{1em}
 \caption[]
 {{\bf {\boldmath $9$} Canonical Fano {\boldmath $3$}-topes with Asymmetric
Fine Interior of Dimension {\boldmath $1$}.} Table contains:
vertices $\vertex(\Theta)$ of the reflexive facet $\Theta \preceq [\Delta^*]$,
support $\supp(\Delta^{\FI})$ of the Fine interior $\Delta^{\FI}$, and
vertices $\vertex(\Delta^{\can})$ of the canonical hull $\Delta^{\can} $.}
 \label{tbl:corti1}
 \end{table}
 }

 \vfill
 \noindent
 {
 \label{S_i}
 \begin{align*}
 S_1 \defeq & (-1, -1, 2),(-1, -1, 3),(-1, 0, 1),(-1, 0, 2),(-1, 1, 0),(-1, 1, 1),(-1, 2, 0),(-1, 3, -1),(0, -1, 1),\\
 &(0, -1, 2),(0, 0, 1), (0, 1, 0),(1, -1, 0),(1, -1, 1),(1, 0, 0),(2, -1, 0) \\
S_2 \defeq & (-1, 0, 1),(-1, 1, 1),(-1, 2, 2),(-1, 3, 2),(-1, 4, 3),(-1, 6, 4),(0, 1, 1),(0, 2, 1),(0, 3, 2),(0, 5, 3), \\
&(1, 1, 0),(1, 2, 1),(1, 4, 2),(2, 0, -1),(2, 1, 0),(2, 3, 1),(3, 2, 0),(4, 1, -1) \\
S_3 \defeq &(-1, 1, 1),(-1, 2, 1),(0, 0, 1),(0, 1, -1),(0, 1, 0),(0, 1, 1),(1, 0, 0),(1, 0, 1),(1, 2, -5) \\
S_4 \defeq &(-1, 0, 1),(-1, 1, 1),(-1, 2, 2),(-1, 3, 2),(-1, 4, 3),(-1, 6, 4),(0, 1, 1),(0, 2, 1),(0, 3, 2),(0, 5, 3),\\
&(1, 0, 0),(1, 1, 0),(1, 2, 1),(1, 4, 2),(2, 1, 0),(2, 3, 1),(3, 0, -1),(3, 2, 0),(4, 1, -1) \\
S_5 \defeq &(-1, 0, 1),(-1, 1, 0),(0, -1, 0),(0, -1, 1),(1, -1, 0),(2, -1, 0),(7, -2, -2)
 \end{align*}
 }

 \vfill
\hspace{0pt}

 \newpage
 \hspace{0pt}
\vfill
{
 \begin{table}[h!]
 \centering
 \rotatebox{0}{%
 \scalebox{1}{
 \begin{tabular}{l|l}
 \multicolumn{1}{c|}{ID} & \multicolumn{1}{c}{$\vertex(\Delta)$} \\\toprule \midrule
$547393 $&$ (0,1,0),(2,1,1),(-2,-3,-5),(2,1,9) $ \\
$547409 $&$ (-4,2,9),(1,0,0),(0,1,0),(7,-6,-18) $\\
$547461 $&$ (0,1,0),(2,1,1),(-2,-3,-5),(0,1,4) $\\ \midrule
$544442 $&$ (1,0,0),(0,1,0),(3,-6,8),(1,-4,4),(-5,6,-12) $ \\
$544443 $&$ (-1,-2,0),(3,-6,8),(0,1,0),(1,0,0),(-3,4,-8) $ \\
$544651 $&$ (-4,1,-3),(4,-2,3),(0,1,0),(1,-2,3),(-1,1,-3) $ \\
$544696 $&$(5,-4,-15),(1,0,0),(0,1,0),(-4,2,9),(-3,1,6) $ \\
$544700 $&$ (-2,-3,-3),(0,1,0),(1,0,0),(-1,-4,-6),(2,5,9) $ \\
$544749 $&$ (-6,-5,-8),(0,1,0),(1,0,0),(-2,-1,0),(3,2,4) $ \\ \midrule
$520925 $&$ (0,1,0),(0,0,1),(-2,-1,0),(-2,0,-1),(8,2,3),(-2,-3,-2) $ \\
$520935 $&$ (3,4,6),(2,1,2),(-3,-2,-2),(1,0,0),(0,1,0),(-6,-5,-8) $ \\
$522056 $&$ (-1,-1,0),(0,1,0),(1,0,0),(-1,-1,-3),(-5,-3,-6),(6,4,9) $ \\
$522059 $&$ (2,5,6),(-2,-3,-3),(0,1,0),(1,0,0),(-1,-4,-6),(0,1,3) $ \\
$522087 $&$ (1,0,-3),(1,0,0),(0,1,0),(-4,2,9),(-3,1,6),(5,-4,-12) $ \\
$522682 $&$ (2,1,4),(-3,-2,-4),(-2,-3,-4),(1,2,4),(1,0,0),(0,1,0) $ \\
$522684 $&$ (-2,-1,-4),(3,2,4),(-2,-1,0),(1,0,0),(0,1,0),(-4,-3,-4) $\\
$526886 $&$ (-3,4,-6),(1,0,0),(0,1,0),(3,-6,8),(0,1,-2),(2,-5,6) $ \\ \midrule
$439403 $&$ (1,2,2),(-1,0,0),(-1,1,-1),(1,0,0),(-1,-2,-2),(1,1,3),$\\
&$(1,-3,-1) $ \\ \midrule
$275525 $&$ (4,1,2),(0,1,0),(-2,-1,0),(1,1,2),(-3,-1,-2),(-2,-1,-2),$ \\
&$(1,1,0),(1,-1,0)$\\
$275528 $&$(-1,0,-1),(-3,-2,1),(-2,-1,2),(0,-1,0),(0,1,0),(1,0,1),$ \\
&$(2,1,-2), (3,2,-1)$
 \end{tabular}
 }
 }
 \vspace{1em}
  \caption[]
 {{\bf {\boldmath $20$} Canonical Fano {\boldmath $3$}-topes with Symmetric
Fine Interior of Dimension {\boldmath $1$}.} Table contains:
vertices $\vertex(\Delta)$ of $\Delta$.}
 \label{tbl:20}
 \end{table}

 \vfill
\hspace{0pt}
}

\newpage

 {
 \hspace{0pt}
\vfill
 \begin{table}[h!]
 \centering
 \rotatebox{0}{%
 \scalebox{0.8}{
 \begin{tabular}{l|l|l|l|l}
 \multicolumn{1}{c|}{ID} & \multicolumn{1}{c|}{$\vertex(\Delta^{\FI})$} & \multicolumn{1}{c|}{$\pm v_{\Delta}$} & \multicolumn{1}{c|}{$\vertex(\theta_{\pm})$} & \multicolumn{1}{c}{$\supp(\Delta^{\FI})$} \\ \toprule \midrule
$547393 $&$\pm 1/2 \cdot v_\Delta$&$\pm(0,0,1) $ &$ (0,1,0),(2,1,1),(-2,-3,-5)$&$(-1, -2, 2),(-1, 1, 0),(0, -1, 0),$\\
& & &$(0,1,0),(-2,-3,-5),(2,1,9) $ & $(1, -1, 0),(2, -1, 0),(9, -2, -2) $ \\
$547409 $& $\pm 2/3 \cdot v_\Delta$&$ \pm(1,-1,-3) $&$(-4,2,9),(1,0,0),(0,1,0) $&$(-3, -3, -1),(-1, -1, 0),(-1, 2, -1),$\\
& & &$ (-4,2,9),(0,1,0),(7,-6,-18)$& $(2, -1, 1),(15, -3, 7)$ \\
$547461 $& $\pm 1/2 \cdot v_\Delta$&$\pm (0,0,1) $&$ (0,1,0),(2,1,1),(-2,-3,-5)$&$(-3, 6, -2),(-1, -2, 2),(-1, 1, 0),$\\
& & & $(2,1,1),(-2,-3,-5),(0,1,4) $& $(0, -1, 0),(1, -1, 0),(2, -1, 0)$\\ \midrule
$544442 $& $\pm 1/2 \cdot v_\Delta$&$\pm(1,-1,2) $ &$(0,1,0),(1,-4,4),(-5,6,-12) $&$ (-2, -2, -1),(-1, -1, 0),(-1, 1, 1),$\\
& & &$(1,0,0),(0,1,0),(3,-6,8) $&$(1, -1, -1),(3, -1, -2),(10, -2, -5)$ \\
$544443 $& $\pm 1/2 \cdot v_\Delta$&$ \pm(1,-1,2) $ &$(-1,-2,0),(0,1,0),(-3,4,-8) $&$ (-2, -2, -1),(-1, -1, 0),(-1, 1, 1),$\\
& & & $(3,-6,8),(0,1,0),(1,0,0) $&$(1, -1, -1),(3, -1, -2),(6, -2, -3)$\\
$544651 $& $\pm 2/3 \cdot v_\Delta$&$ \pm(1,0,0) $&$ (-4,1,-3),(0,1,0),(1,-2,3)$&$(-3, -3, 1),(0, -1, -1),(0, -1, 0),$\\
& & & $(4,-2,3),(0,1,0),(-1,1,-3) $&$(0, 2, 1),(3, -3, -4) $\\
$544696 $& $\pm 2/3 \cdot v_\Delta$&$\pm (1,-1,-3) $ &$ (1,0,0),(0,1,0),(-4,2,9)$&$(-3, -3, -1),(-3, 12, -4),(-1, -1, 0), $\\
& & &$(5,-4,-15),(1,0,0),(-3,1,6) $& $(-1, 2, -1),(2, -1, 1)$\\
$544700 $& $\pm 2/3 \cdot v_\Delta$&$ \pm(1,2,3) $ &$(-2,-3,-3),(0,1,0),(-1,-4,-6) $&$(-3, -3, 2),(-1, -1, 1),(-1, 2, -1),$\\
& & &$ (0,1,0),(1,0,0),(2,5,9)$& $(2, -1, 0),(3, -3, 2) $ \\
$544749 $& $\pm 1/2 \cdot v_\Delta$&$\pm (1,1,2) $&$ (-6,-5,-8),(0,1,0),(1,0,0)$&$(-2, -2, 3),(-1, -1, 1),(-1, 1, 0), $\\
& & & $(0,1,0),(-2,-1,0),(3,2,4) $&$(-1, 3, -1),(1, -1, 0),(2, -2, -1)$\\\midrule
$520925 $& $\pm 1/2 \cdot v_\Delta$&$\pm(2,1,1) $ &$ (-2,-1,0),(-2,0,-1),(-2,-3,-2)$&$(-1, -1, 3),(0, -1, 1),(0, 1, -1),$\\
& & &$ (0,1,0),(0,0,1),(8,2,3)$& $(1, -2, -2),(1, -1, -1),(1, 0, 0)$ \\
$520935 $& $\pm 1/2 \cdot v_\Delta$&$ \pm(1,1,2) $ &$(1,0,0),(0,1,0),(-6,-5,-8) $&$(-2, -2, 3),(-1, -1, 1),(-1, 1, 0),$\\
& & & $(3,4,6),(2,1,2),(-3,-2,-2) $& $(-1, 3, -1),(0, 4, -3),(1, -1, 0)$\\
$522056 $& $\pm 2/3 \cdot v_\Delta$&$ \pm (2,1,3)$ &$ (0,1,0),(-1,-1,-3),(-5,-3,-6)$&$(-3, 6, -1),(-1, -1, 1),(-1, 2, 0),$\\
& & &$ (-1,-1,0),(1,0,0),(6,4,9)$ & $(0, -3, 2),(2, -1, -1)$\\
$522059 $& $\pm 2/3 \cdot v_\Delta$&$\pm (1,2,3) $ &$ (-2,-3,-3),(0,1,0),(-1,-4,-6)$&$(-3, 3, -2),(-1, -1, 1),(-1, 2, -1),$\\
& & &$ (2,5,6),(1,0,0),(0,1,3)$& $(2, -1, 0),(3, -3, 2) $\\
$522087 $& $\pm 2/3 \cdot v_\Delta$&$ \pm(1,-1,-3) $&$ (1,0,0),(0,1,0),(-4,2,9)$&$ (-3, -3, -1),(-1, -1, 0),(-1, 2, -1),$\\
& & &$ (1,0,-3),(-3,1,6),(5,-4,-12)$& $(2, -1, 1),(9, 0, 4)$ \\
$522682 $& $\pm 1/2 \cdot v_\Delta$&$\pm (1,1,2) $ &$(-3,-2,-4),(-2,-3,-4),(1,0,0),(0,1,0) $&$(-2, -2, 1),(-2, -2, 3),(-1, -1, 1),$\\
& & &$ (2,1,4),(1,2,4),(1,0,0),(0,1,0)$ & $(-1, 1, 0),(1, -1, 0),(1, 1, -1)$\\
$522684 $& $\pm 1/2 \cdot v_\Delta$&$\pm(1,1,2) $ &$(-2,-1,-4),(1,0,0),(-4,-3,-4) $&$(-2, 2, 1),(-1, -1, 1),(-1, 1, 0),$\\
& & & $(3,2,4),(-2,-1,0),(0,1,0) $&$(-1, 3, -1),(1, -1, 0),(2, -2, -1)$\\
$526886 $& $\pm 1/2 \cdot v_\Delta$&$ \pm (1,-1,2) $ &$ (-3,4,-6),(0,1,-2),(2,-5,6)$&$ (-2, -2, -1),(-1, -1, 0),(-1, 1, 1),$\\
& & &$(1,0,0),(0,1,0),(3,-6,8) $& $(0, 4, 3),(1, -1, -1),(3, -1, -2)$ \\ \midrule
$439403 $& $\pm 1/2 \cdot v_\Delta$&$\pm (0,1,1) $ &$(-1,1,-1),(1,0,0),(-1,-2,-2),(1,-3,-1) $&$(-2, -1, 3),(-1, -1, 1),(-1, 0, 0),$\\
& & & $ (1,2,2),(-1,0,0),(-1,1,-1),(1,1,3)$& $(1, 0, 0),(1, 1, -1),(2, -1, -1)$ \\\midrule
$275525 $& $\pm 1/2 \cdot v_\Delta$&$\pm(1,0,0) $ &$ (0,1,0),(-2,-1,0),(1,1,2),(-3,-1,-2)$&$(-2, 0, 3),(0, -1, 0),(0, -1, 1),$\\
& & &$(4,1,2),(-2,-1,-2),(1,1,0),(1,-1,0) $ & $\pm(0, 1, -1),(0, 1, 0),(2, -2, -1) $ \\
$275528 $& $\pm 1/2 \cdot v_\Delta$&$\pm(1,1,-1) $ &$(-3,-2,1),(-2,-1,2),(0,-1,0),(1,0,1)$&$(-1, 1, 0),(-1, 2, -1),(0, -1, -1),$\\
& & &$ (-1,0,-1),(0,1,0),(2,1,-2),(3,2,-1)$ & $(0, 1, 1),(1, -2, 1),(1, -1, 0)$
 \end{tabular}
 }
 }
 \vspace{1em}
  \caption[]
 {{\bf {\boldmath $20$} Canonical Fano {\boldmath $3$}-topes with Symmetric
Fine Interior of Dimension {\boldmath $1$}.} Table contains:
vertices $\vertex(\Delta^{\FI})$ of the Fine interior $\Delta^{\FI}$,
unique primitive lattice points $\pm v_{\Delta} \in \theta_\pm$
in the reflexive facets $\theta_\pm \preceq \Delta$,
vertices $\vertex(\theta_{{\pm}})$ of the reflexive facets $\theta_{{\pm}}
\preceq \Delta$, and
 support $\supp(\Delta^{\FI})$
of the Fine interior $\Delta^{\FI}$ (here: $\Delta^{\can} = \Delta$). }
 \label{tbl:200}
 \end{table}

 \vfill
\hspace{0pt}
}

\newpage
 \hspace{0pt}
\vfill
 {
 \begin{table}[h!]
 \centering
 \rotatebox{0}{%
 \scalebox{0.85}{
 \begin{tabular}{l|l}
 \multicolumn{1}{c|}{ID} & \multicolumn{1}{c}{$\vertex(\Delta)$} \\ \toprule \midrule
$547444 $&$ (1,0,0),(-2,-4,-5),(1,2,4),(1,4,2) $\\
$547465 $&$ (-3,-2,-2),(1,0,0),(1,3,1),(1,1,3) $\\
$547524 $&$ (0,2,1),(-2,-3,-5),(2,1,1),(0,0,1) $\\
$547525 $&$ (0,0,1),(0,1,0),(2,1,1),(-2,-5,-7) $\\\midrule
$545317 $&$ (-3,4,-6),(0,1,0),(1,0,0),(1,-2,4),(3,-5,6) $\\
$545932 $&$ (0,-1,-1),(1,-1,-3),(-2,1,5),(1,0,0),(1,2,-2) $\\
$546013 $&$ (3,-5,6),(1,-2,4),(1,0,0),(-1,1,-2),(-1,3,-2) $\\
$546062 $&$ (0,1,3),(-2,1,-1),(0,1,0),(1,0,0),(-1,-2,-2) $\\
$546070 $&$ (0,-2,-3),(0,2,1),(-2,-3,-5),(2,1,1),(0,0,1) $\\
$546205 $&$ (1,2,-2),(-1,0,2),(1,0,0),(-2,1,5),(1,-1,-3) $\\
$546219 $&$ (1,1,1),(-3,-2,-2),(1,0,0),(1,3,1),(-1,-1,1) $\\
$546663 $&$ (2,-3,-1),(1,0,0),(0,1,0),(0,0,1),(-2,-3,-3) $\\
$546862 $&$ (1,0,0),(0,1,0),(-2,1,5),(1,-1,-3),(1,2,-2) $\\
$546863 $&$ (-1,-1,1),(1,3,1),(0,0,1),(1,0,0),(-3,-2,-2) $\\
$547240 $&$ (-1,1,-2),(0,1,0),(1,0,0),(1,-2,4),(3,-5,6) $\\
$547246 $&$ (0,-2,-3),(-2,-3,-5),(2,1,1),(0,1,0),(0,0,1) $\\\midrule
$532384 $&$ (1,-1,-3),(-2,1,5),(1,0,0),(1,-1,-2),(0,-1,-1),(1,2,-2) $\\
$532606 $&$ (0,-1,2),(-1,-1,0),(0,1,0),(1,0,0),(2,2,-3),(-2,0,-3) $\\
$533513 $&$ (-1,1,2),(1,0,0),(0,1,0),(1,1,2),(-1,-2,-4),(-2,-3,-4) $\\
$534667 $&$ (1,0,3),(-1,-1,-1),(0,1,0),(1,0,0),(-1,-1,0),(5,2,3) $\\
$534669 $&$ (1,3,0),(5,3,2),(-1,-1,-1),(0,0,1),(1,0,0),(-1,-1,0) $\\
$534866 $&$ (-1,-1,-3),(1,0,0),(0,1,0),(1,1,1),(-1,-1,0),(-3,-5,-3) $\\
$535952 $&$ (3,-5,6),(1,-2,4),(1,0,0),(0,1,0),(-1,1,-2),(-1,2,-2) $\\
$536013 $&$ (0,1,1),(0,0,1),(0,1,0),(2,1,1),(-2,-3,-5),(0,-2,-3) $\\
$536498 $&$ (1,2,-2),(1,-1,-2),(1,0,0),(0,1,0),(-2,1,5),(1,-1,-3) $ \\
$537834 $&$ (0,0,1),(1,0,0),(0,1,0),(-2,1,5),(1,-1,-3),(1,2,-2) $\\
$538356 $&$ (-2,-3,-3),(-1,-3,-1),(0,0,1),(0,1,0),(1,0,0),(-1,-1,-3) $\\
$539063 $&$ (-1,1,-1),(1,1,3),(-3,-2,-2),(1,0,0),(0,1,0),(1,1,2) $\\
$539304 $&$ (1,0,1),(-3,-1,-2),(1,1,2),(-2,-1,0),(1,0,0),(1,2,0) $\\
$539313 $&$ (1,-1,-2),(1,1,-1),(-1,2,2),(1,-1,-3),(-2,1,5),(1,0,0) $\\
$540602 $&$ (0,0,1),(1,0,0),(-2,1,5),(1,-1,-3),(-1,2,2),(1,1,-1) $\\
$540663 $&$ (1,0,0),(0,1,0),(1,1,2),(-3,-1,-2),(1,1,1),(-3,-2,0) $\\\midrule
$474457 $&$ (-1,2,-3),(1,0,2),(0,0,1),(0,1,0),(1,0,0),(-1,-1,0),(-3,-2,-3) $\\
$481575 $&$ (3,2,4),(-1,-1,-2),(-3,-1,-2),(-2,-1,0),(0,1,0),(1,0,0),(0,0,-1) $\\
$483109 $&$ (3,0,2),(1,-2,-2),(0,0,-1),(-1,-1,0),(1,1,1),(0,1,0),(-1,0,0) $\\
$490478 $&$ (1,-1,-2),(1,1,-1),(-1,2,2),(1,-1,-3),(-2,1,5),(1,0,0),(-1,0,2) $\\
$490481 $&$ (-3,-2,0),(-5,-3,-2),(1,0,0),(0,1,0),(1,1,2),(-1,-1,-1),(2,1,1) $\\
$490485 $&$ (-1,-1,0),(1,2,0),(1,0,0),(-2,-1,0),(1,1,2),(-3,-1,-2),(1,0,1) $\\
$490511 $&$ (1,0,0),(0,1,0),(-2,-1,0),(1,1,2),(2,1,1),(1,0,1),(-5,-2,-4) $\\
$495687 $&$ (0,0,-1),(1,1,-1),(-1,2,2),(1,-1,-3),(-2,1,5),(1,0,0),(0,0,1) $\\
$499287 $&$ (1,1,1),(-1,-1,-3),(1,0,0),(0,1,0),(0,0,1),(-1,-3,-1),(-2,-3,-3) $\\
$499291 $&$ (-1,-1,-1),(-1,-1,-3),(1,0,0),(0,1,0),(0,0,1),(-1,-3,-1),$\\
& $(-2,-3,-3) $\\
$499470 $&$ (1,0,0),(0,1,0),(-2,-1,0),(1,1,2),(0,0,1),(-5,-2,-4),(2,1,1) $\\
$501298 $&$ (3,-6,8),(-1,1,-2),(1,-2,3),(0,1,0),(1,0,0),(0,1,-1),(3,-5,6) $\\
$501330 $&$ (1,0,0),(0,1,0),(-2,-1,0),(1,1,2),(1,1,1),(0,0,1),(-5,-2,-4) $\\\midrule
$354912 $&$ (3,1,2),(1,0,0),(0,1,0),(-2,-1,0),(1,1,2),(2,1,1),(1,0,1),(-5,-2,-4) $\\
$372528 $&$ (2,1,1),(-1,-1,-1),(1,1,2),(0,1,0),(1,0,0),(-5,-3,-2),(-3,-2,0),$\\
& $(1,1,0) $\\
$372973 $&$ (-5,-2,-4),(1,0,1),(2,1,1),(1,1,2),(-2,-1,0),(0,1,0),(1,0,0),(2,1,2) $\\
$388701 $&$ (1,1,1),(-2,-3,-3),(-1,-3,-1),(0,0,1),(0,1,0),(1,0,0),(-1,-1,-3),$ \\
& $(-1,-1,-1) $\\
 \end{tabular}
 }
 }
 \vspace{1em}
  \caption[]
 {{\bf {\boldmath $49$} Canonical Fano {\boldmath $3$}-topes with
Fine Interior of Dimension {\boldmath $3$}.} Table contains:
vertices $\vertex(\Delta)$ of $\Delta$.}
 \label{tbl:491}
 \end{table}
 \vfill
 \hspace{0pt}
 }

 \hspace{0pt}
\vfill
 {
 \begin{table}[h!]
 \centering
 \rotatebox{0}{%
 \scalebox{0.85}{
 \begin{tabular}{l|l}
 \multicolumn{1}{c|}{ID} & \multicolumn{1}{c}{$\vertex(\Delta^{\FI})$} \\ \toprule \midrule
$547444 $&$ 0, (-1/2, -1, -3/2), (0, -1/3, -2/3), (0, 1/3, -1/3) $\\
$547465 $&$ 0, (-1, -1/2, -1/2), (0, 3/4, 1/4), (0, 1/4, 3/4)$\\
$547524 $&$ 0, (0, 1/2, 0), (1/3, 1/3, 0), (-1/3, -1/3, -1) $\\
$547525 $&$ 0, (0, 0, -1/2), (1/3, 0, -1/3), (-1/3, -1, -5/3)$ \\\midrule
$545317 $&$ 0, (1, -3/2, 2), (2/3, -2/3, 1), (1/2, -1/2, 1), (2/3, -1, 5/3) $\\
$545932 $&$ 0, (-1/2, 1/2, 3/2), (0, 1/3, 2/3), (0, 2/3, 1/3) $\\
$546013 $&$0, (1, -3/2, 2), (0, 1/2, 0), (1/2, -1/4, 1/2), (1/2, -3/4, 3/2)$\\
$546062 $&$ 0, (-1/2, -1/2, -1/2), (-2/3, 0, -1/3), (-1/3, 0, 1/3) $\\
$546070 $&$ 0, (0, 1/2, 0), (1/2, 1/4, 0), (0, -1/2, -1), (-1/2, -3/4, -3/2)$\\
$546205 $&$ 0, (-1/2, 1/2, 3/2), (0, 1/3, 2/3), (0, 2/3, 1/3) $\\
$546219 $&$ 0, (-1, -1/2, -1/2), (-1/3, 1/3, 0),(-2/3, -1/3, 0)$\\
$546663 $&$ 0, (0, -1/2, 0), (1/3, -1, -1/3), (-1/3, -1, -2/3)$ \\
$546862 $&$ 0, (-1/2, 1/2, 3/2), (0, 1/3, 2/3), (0, 2/3, 1/3) $\\
$546863 $&$0, (-1, -1/2, -1/2), (-1/3, 1/3, 0), (-2/3, -1/3, 0) $\\
$547240 $&$0, (1, -3/2, 2), (2/3, -2/3, 1), (1/2, -1/2, 1), (2/3, -1, 5/3) $\\
$547246 $&$ 0, (0, 0, -1/2), (1/3, 0, -1/3), (0,-1/2, -1), (-1/3, -2/3, -4/3) $\\\midrule
$532384 $&$0, (-1/2, 1/2, 3/2), (0,1/3, 2/3), (0, 2/3, 1/3) $\\
$532606 $&$ 0, (0, 1/2, -1/2), (1/3, 2/3,-1), (-1/3, 1/3, -1)$\\
$533513 $&$ 0, (-1/2, -1/2, -1), (-1/2, 0, 0), (-1/3, 0, -1/3), (-2/3, -2/3, -1)$\\
$534667 $&$ 0, (1/2, 1/2, 1/2), (4/3, 2/3, 1), (2/3, 1/3, 1) $\\
$534669 $&$ 0, (1/2, 1/2, 1/2), (4/3, 1, 2/3), (2/3, 1, 1/3) $\\
$534866 $&$0, (0, -1/2, -1/2), (-1/3, -2/3, -1), (-2/3, -4/3, -1) $\\
$535952 $&$0, (1, -3/2, 2), (2/3, -2/3, 1), (1/2, -1/2, 1), (2/3, -1, 5/3) $\\
$536013 $&$ 0, (0, 0, -1/2),(1/3, 0, -1/3), (0, -1/2, -1), (-1/3, -2/3, -4/3) $\\
$536498 $&$0, (-1/2, 1/2, 3/2), (0, 1/3, 2/3), (0, 2/3, 1/3) $\\
$537834 $&$0, (-1/2, 1/2, 3/2), (0, 1/3, 2/3), (0, 2/3, 1/3) $\\
$538356 $&$0, (0, -1/2,-1/2), (-1/3, -2/3, -1), (-1/3, -1, -2/3), (-1/2, -1, -1)$\\
$539063 $&$ 0, (-1, -1/2, -1/2), (-2/3, 0, -1/3), (-1/3, 0, 1/3)$\\
$539304 $&$0, (0, 1/2, 0), (-1/2, 0, 0), (0, 1/3, 1/3), (-2/3, 0, -1/3) $\\
$539313 $&$ 0, (-1/2, 1/2, 3/2), (0, 1/3, 2/3), (0, 1/2, 1/2), (-1/3, 2/3, 1) $\\
$540602 $&$ 0, (-1/2, 1/2, 3/2), (0, 1/3, 2/3), (0, 1/2, 1/2), (-1/3,2/3, 1)$ \\
$540663 $&$0, (-1/2, 0, 0), (-1, -1/2, 0), (-1/3, 0, 1/3), (-1, -1/3, -1/3) $\\\midrule
$474457 $&$0, (0, 0, -1/2), (-1/3, 1/3, -1), (-2/3, -1/3, -1) $\\
$481575 $&$ 0, (-1/2, 0, 0), (1/2, 1/2, 1), (0, 1/3, 1/3), (-1/3, 0, 1/3)$\\
$483109 $&$ 0, (0, -1/2, 0), (2/3, -1/3, 1/3), (1/3, -2/3, -1/3) $\\
$490478 $&$0, (-1/2, 1/2, 3/2), (0, 1/3, 2/3), (0, 1/2, 1/2), (-1/3, 2/3, 1) $\\
$490481 $&$0, (-1/2, 0, 0), (-1, -1/2, 0), (-1/3, 0, 1/3), (-4/3, -2/3, -1/3) $ \\
$490485 $&$0, (0,1/2, 0), (-1/2, 0, 0), (0, 1/3, 1/3), (-2/3, 0, -1/3)$\\
$490511 $&$ 0, (-3/2, -1/2, -1), (-1/2, 0, 0), (-2/3, 0, -1/3), (-1, -1/3, -1/3) $\\
$495687 $&$0, (-1/2, 1/2, 3/2), (0, 1/3, 2/3), (0, 1/2, 1/2), (-1/3, 2/3, 1) $\\
$499287 $&$0, (0, -1/2, -1/2), (-1/3, -2/3, -1), (-1/3, -1,-2/3), (-1/2, -1, -1) $\\
$499291 $&$ 0, (0, -1/2, -1/2), (-1/3, -2/3,-1), (-1/3, -1, -2/3), (-1/2, -1, -1)$\\
$499470 $&$ 0, (-3/2, -1/2,-1), (-1/2, 0, 0), (-2/3, 0, -1/3), (-1, -1/3, -1/3)$\\
$501298 $&$ 0, (1/2, -1/2, 1), (2/3, -2/3, 1), (1, -3/2, 2), (1, -5/3, 7/3) $\\
$501330 $&$ 0, (-3/2, -1/2, -1), (-1/2, 0, 0), (-2/3, 0, -1/3), (-1, -1/3, -1/3)$ \\\midrule
$354912 $&$ 0, (-3/2, -1/2, -1), (-1/2, 0, 0), (-2/3, 0, -1/3), (-1, -1/3, -1/3)$ \\
$372528 $&$ 0, (-1/2, 0, 0), (-1, -1/2, 0), (-1/3,0, 1/3), (-4/3, -2/3, -1/3) $\\
$372973 $&$0, (-3/2, -1/2, -1), (-1/2, 0, 0), (-2/3, 0, -1/3), (-1, -1/3, -1/3) $ \\
$388701 $&$0, (0, -1/2, -1/2), (-1/3, -2/3, -1), (-1/3, -1, -2/3), (-1/2, -1, -1) $
 \end{tabular}
 }
 }
 \vspace{1em}
  \caption[]
 {{\bf {\boldmath $49$} Canonical Fano {\boldmath $3$}-topes with
Fine Interior of Dimension {\boldmath $3$}.} Table contains:
vertices $\vertex(\Delta^{\FI})$ of the Fine interior $\Delta^{\FI}$.}
 \label{tbl:492}
 \end{table}
 \vfill
 \hspace{0pt}

 }
\newpage

\hspace{0pt}
\vfill
 {
 \begin{table}[h!]
 \centering
 \rotatebox{0}{%
 \scalebox{0.85}{
 \begin{tabular}{l|l|l|c}
 \multicolumn{1}{c|}{ID} & \multicolumn{1}{c|}{$\supp(\Delta^{\FI})$} & \multicolumn{1}{c|}{$\vertex(\Delta^{\can} )$}& $\left\vert \pi_1(\mathcal S_\Delta) \right\vert$ \\ \toprule \midrule
$547444 $&$(-2, -1, 2),(-1, 0, 0),(-1, 2, -1),(1, 1, -1),(3, 0, -1),$&$\vertex(\Delta)$ & $1$\\
& $(5, -1, -1)$ & & \\
$547465 $& $(-1, -1, 3),(-1, 0, 0),(-1, 0, 1),(-1, 0, 2),(-1, 1, 0),$& $\vertex(\Delta)$ & $2$\\
& $(-1, 1, 1),(-1, 2, 0),(-1, 3, -1),(2, -1, -1)$ & & \\
$547524 $&$(-1, -2, 2),(-1, 1, 0),(-1, 2, -1),(0, 0, -1),(0, 1, -1),$&$\vertex(\Delta), (0,-1,-1)$& $1$\\
& $(0, 2, -1),(1, 0, -1),(1, 1, -1),(2, 0, -1),(3, 0, -1)$& & \\
$547525 $&$(-1, -2, 2),(-1, 2, -1),(0, -1, 0),(0, 0, -1),(0, 1, -1),$&$\vertex(\Delta),(1,1,1),$& $1$\\
& $(1, -1, -1),(1, -1, 0),(1, 0, -1),(1, 1, -1),(2, -1, -1),$& $(-1,-2,-3)$ & \\
& $(2, -1, 0),(2, 0, -1),(3, -1, -1),(3, -1, 0),(3, 0, -1),$& & \\
& $(4, -1, -1),(4, 0, -1),(5, -1, -1),(6, -1, -1)$& & \\ \midrule
$545317 $& $(-2, -2, -1),(-1, -1, 0),(-1, 2, 2),(1, -1, -1),(1, 2, 1),$&$\vertex(\Delta)$ & $1$\\
& $(3, 2, 0)$& & \\
$545932 $& $(-2, -1, -1),(-1, 0, 0),(-1, 2, -1),(0, 1, 0),(1, 1, 0),$ & $\vertex(\Delta), (1,-1,-2),$& $1$\\
& $(2, 0, 1),(3, 0, 1),(5, -1, 2)$& $(1,0,-3)$& \\
$546013 $ & $(-2, -2, -1),(-1, 0, 1),(-1, 2, 2),(0, 1, 1),(1, 0, 0),$& $\vertex(\Delta)$ & $2$ \\
& $(1, 2, 1),(2, 1, 0),(3, 0, -1),(3, 2, 0)$& & \\
$546062 $ & $(-1, -1, 0),(-1, -1, 1),(-1, -1, 2),(-1, 0, 0),(-1, 0, 1),$ &$\vertex(\Delta)$ & $1$\\
& $(-1, 1, 0),(-1, 2, -1),(0, -1, 0),(2, 1, -1)$& & \\
$546070 $& $(-1, -2, 2),(-1, 2, -1),(0, 0, -1),(0, 1, -1),(0, 2, -1),$ &$\vertex(\Delta)$ & $2$\\
& $(1, 0, -1),(1, 1, -1),(2, 0, -1),(3, 0, -1)$& & \\
$546205 $& $(-2, -1, -1),(-1, 0, 0),(-1, 2, -1),(-1, 3, -1),(0, 1, 0),$&$\vertex(\Delta)$ & $1$\\
& $(1, 1, 0),(1, 2, 0),(2, 0, 1),(3, 0, 1),(3, 1, 1),$& & \\
& $(5, -1, 2),(5, 0, 2),(7, -1, 3)$& & \\
$546219 $ & $(-1, -1, 3),(-1, 0, 0),(-1, 0, 1),(-1, 0, 2),(-1, 1, -1),$&$\vertex(\Delta)$ & $1$\\
& $(-1, 1, 0),(-1, 1, 1),(-1, 2, 0),(0, 0, -1),(2, -1, -1)$& & \\
$546663 $ & $(-1, -1, -1),(-1, -1, 0),(-1, -1, 1),(-1, -1, 2),(-1, 0, -1),$ &$\vertex(\Delta),(-1,-1,-1)$ & $1$\\
&$(0, -1, -1),(0, -1, 0),(0, -1, 1),(0, 0, -1), (1, -1, -1),$ & & \\
& $(1, -1, 0),(1, 0, -1),(1, 2, -2),(2, -1, -1),(2, -1, 0),$& & \\
& $(2, 0, -1),(3, -1, -1)$& & \\
$546862 $ &$(-2, -1, -1),(-1, 0, 0),(-1, 2, -1),(-1, 3, -1),(0, 1, 0),$& $\vertex(\Delta),(0,0,1)$& $1$\\
& $(0, 4, -1),(1, 1, 0),(1, 2, 0),(2, 0, 1),(2, 3, 0),$& & \\
& $(3, 0, 1),(3, 1, 1),(4, 2, 1),(5, -1, 2),(5, 0, 2),$& & \\
& $(6, 1, 2),(7, -1, 3),(8, 0, 3),(10, -1, 4)$& & \\
$546863 $& $(-1, -1, 3),(-1, 0, 0),(-1, 0, 1),(-1, 0, 2),(-1, 1, -1),$& $\vertex(\Delta) \setminus \{(0,0,1)\},$ & $1$\\
& $(-1, 1, 0),(-1, 1, 1),(-1, 2, 0),(0, 0, -1),(2, -1, -1)$& $(1,1,1)$& \\
$547240 $& $(-2, -2, -1),(-1, -1, 0),(-1, 0, 1),(-1, 2, 2),(0, -1, 0),$& $\vertex(\Delta),(0,1,-1),$& $1$\\
& $(0, 1, 1),(1, -1, -1),(1, 0, 0),(1, 2, 1),(2, -1, -1),$&$(0,0,1)$ & \\
& $(2, 1, 0),(3, 0, -1),(3, 2, 0)$& & \\
$547246 $&$(-1, -2, 2),(-1, 2, -1),(0, -1, 0),(0, 0, -1),(0, 1, -1),$&$\vertex(\Delta),(1,1,1),$ & $1$ \\
& $(0, 2, -1),(1, -1, -1),(1, -1, 0),(1, 0, -1),(1, 1, -1),$& $(-1,-1,-2)$ & \\
& $(2, -1, -1),(2, -1, 0),(2, 0, -1),(3, -1, -1),(3, 0, -1),$& & \\
& $(4, -1, -1)$& & \\\midrule
$532384 $ & $(-2, -1, -1),(-1, 0, 0),(-1, 2, -1),(0, 1, 0),(1, 1, 0),$ & $\vertex(\Delta),(1,0,-3)$& $1$ \\
&$(2, 0, 1),(3, 0, 1),(5, -1, 2)$ & & \\
 \end{tabular}
 }
 }
 \vspace{1em}
  \caption[]
 {{\bf {\boldmath $49$} Canonical Fano {\boldmath $3$}-topes with
Fine Interior of Dimension {\boldmath $3$}.} Table contains:
 support $\supp(\Delta^{\FI})$ of the Fine interior $\Delta^{\FI}$,
vertices $\vertex(\Delta^{\can})$ of the canonical hull $\Delta^{\can}$,
 and order of fundamental group $\left\vert \pi_1(\mathcal S_\Delta) \right\vert$ of the minimal model $\mathcal S_{\Delta}$.}
 \label{tbl:493}
 \end{table}
 \vfill
 \hspace{0pt}
 }
\newpage

\hspace{0pt}
\vfill
 {
 \begin{table}[h!]
 \ContinuedFloat
 \centering
 \caption[]
 {}
 \rotatebox{0}{%
 \scalebox{0.85}{
 \begin{tabular}{l|l|l|c}
 \multicolumn{1}{c|}{ID} & \multicolumn{1}{c|}{$\supp(\Delta^{\FI})$} & \multicolumn{1}{c|}{$\vertex(\Delta^{\can} )$} & $\left\vert \pi_1(\mathcal S_\Delta) \right\vert$\\ \toprule \midrule
$532606 $ & $(-1, -1, -1),(-1, 1, 0),(-1, 2, 1),(0, -1, -1),(0, 1, 0),$ &$\vertex(\Delta),(0,-1,1)$ & $1$ \\
& $(1, -2, 0),(1, -1, -1),(2, -1, -1)$& & \\
$533513 $ & $(-1, -1, 1),(-1, 0, 0),(-1, 1, 0),(-1, 2, -1),(0, -1, 0),$&$\vertex(\Delta),(1,1,1),$ & $1$\\
&$(0, 1, -1),(0, 3, -2),(2, -2, 1)$ & $(0,-1,-1)$ & \\
$534667 $ &$(-1, -1, 2),(-1, -1, 3),(-1, 0, 2),(-1, 1, 1),(-1, 2, 0),$& $\vertex(\Delta)$& $1$\\
& $(0, -1, 1),(0, -1, 2),(0, 0, 1),(0, 1, 0),(1, -2, -1),$& & \\
& $(1, -1, 0),(1, -1, 1),(1, 0, 0),(2, -1, -1),(2, -1, 0)$& & \\
$534669 $& $(-1, 0, 2),(-1, 1, 1),(-1, 2, -1),(-1, 2, 0),(0, 0, 1),$&$\vertex(\Delta)$& $1$\\
& $(0, 1, -1),(0, 1, 0),(1, -1, -2),(1, 0, -1),(1, 0, 0),$& & \\
& $(2, -1, -1),(2, -1, 0)$& & \\
$534866 $& $(-2, 1, 1),(-1, -1, 1),(-1, 0, 0),(-1, 1, -1),(-1, 2, -2),$& $\vertex(\Delta)$& $1$\\
& $(0, -1, 0),(0, 0, -1),(0, 1, -2),(1, -1, -1),(1, -1, 0),$& & \\
& $(1, 0, -2),(1, 0, -1),(2, -1, -2),(2, -1, -1),(2, -1, 0)$& & \\
$535952 $ & $(-2, -2, -1),(-1, -1, 0),(-1, 0, 1),(-1, 2, 2),(0, 1, 1),$& $\vertex(\Delta)$& $1$\\
& $(1, -1, -1),(1, 0, 0),(1, 2, 1),(2, 1, 0),(3, 0, -1),$& & \\
& $(3, 2, 0)$& & \\
$536013 $& $(-1, -2, 2),(-1, 2, -1),(0, -1, 0),(0, 0, -1),(0, 1, -1),$ &$\vertex(\Delta)$ & $1$\\
& $(0, 2, -1),(1, -1, 0),(1, 0, -1),(1, 1, -1),(2, -1, 0),$& & \\
& $(2, 0, -1),(3, 0, -1)$& & \\
$536498 $& $(-2, -1, -1),(-1, 0, 0),(-1, 2, -1),(0, 1, 0),(1, 1, 0),$& $\vertex(\Delta)$& $1$\\
& $(1, 2, 0),(2, 0, 1),(2, 3, 0),(3, 0, 1),(3, 1, 1),$& & \\
& $(4, 2, 1),(5, -1, 2),(5, 0, 2),(6, 1, 2),(7, -1, 3),$& & \\
& $(8, 0, 3),(10, -1, 4)$& & \\
$537834 $& $(-2, -1, -1),(-1, 0, 0),(-1, 2, -1),(-1, 3, -1),(0, 1, 0),$ & $\vertex(\Delta)$& $1$\\
& $(0, 4, -1),(1, 1, 0),(1, 2, 0),(2, 0, 1),(2, 3, 0),$& & \\
& $(3, 0, 1),(3, 1, 1),(4, 2, 1),(5, -1, 2),(5, 0, 2),$& & \\
& $(6, 1, 2),(7, -1, 3),(8, 0, 3),(10, -1, 4)$& & \\
$538356 $ & $(-2, 1, 1),(-1, -1, -1),(-1, -1, 0),(-1, -1, 1),(-1, 0, -1),$& $\vertex(\Delta),(-1,-1,-1)$&$1$\\
& $(-1, 0, 0),(-1, 1, -1),(0, -1, -1),(0, -1, 0), (0, 0, -1),$& & \\
& $(1, -1, -1),(1, -1, 0),(1, 0, -1),(2, -1, -1),(2, -1, 0),$& & \\
& $(2, 0, -1),(3, -1, -1)$& & \\
$539063 $ & $(-1, -1, 1),(-1, 0, 0),(-1, 0, 1),(-1, 0, 2),(-1, 1, 0),$ & $\vertex(\Delta) \setminus \{(0,1,0),(1,1,2)\},$& $1$\\
& $(-1, 1, 1),(-1, 2, 0),(-1, 3, -1),(0, -1, 0),(2, -1, -1)$& $(1,1,1)$& \\
$539304 $& $(-1, 0, 0),(-1, 0, 1),(-1, 0, 2),(-1, 1, 0),(-1, 1, 1),$ &$\vertex(\Delta),(-2,-1,-1)$ & $1$\\
& $(-1, 2, 0),(-1, 2, 1),(-1, 3, 0),(0, 1, -1),(0, 1, 0),$& & \\
& $(2, -2, -1)$& & \\
$539313 $ & $(-2, -1, -1),(-1, 0, 0),(-1, 2, -1),(0, 1, 0),(1, -1, 1),$ &$\vertex(\Delta),(-1,1,2)$ & $1$\\
& $(1, 1, 0),(1, 2, 0),(2, 0, 1),(2, 3, 0),(3, 0, 1), $ & & \\
& $(3, 1, 1),(4, 2, 1),(5, 0, 2),(6, 1, 2)$ & & \\
$540602 $ & $(-2, -1, -1),(-1, 0, 0),(-1, 2, -1),(-1, 3, -1),(0, 1, 0),$&$\vertex(\Delta),(-1,1,2)$ & $1$\\
& $(0, 4, -1),(1, -1, 1),(1, 1, 0),(1, 2, 0),(2, 0, 1),$ & & \\
& $(2, 3, 0),(3, 0, 1),(3, 1, 1),(4, 2, 1),(5, 0, 2),$ & & \\
& $(6, 1, 2)$ & & \\
$540663 $ & $(-1, -1, 1),(-1, -1, 2),(-1, 0, 0),(-1, 0, 1),(-1, 0, 2),$&$\vertex(\Delta),(-1,0,-1)$ & $1$\\
& $(-1, 1, 0),(-1, 1, 1),(-1, 2, -1),(-1, 2, 0),(-1, 2, 1),$ & & \\
& $(0, -1, 0),(0, -1, 1),(2, -2, -1)$ & & \\ \midrule
$474457 $& $(-2, 1, 2),(-1, -1, 0),(-1, 0, 0),(-1, 1, 0),(-1, 2, 0),$ & $\vertex(\Delta)$& $1$\\
& $(1, -1, -1),(1, 0, -1),(2, -1, -1)$ & & \\
$481575 $ &$(-1, -1, 1),(-1, 0, 1),(-1, 1, 0),(-1, 2, 0),(-1, 3, -1),$ &$\vertex(\Delta),(-1,-1,-1)$ & $1$\\
& $(0, -1, 1),(0, 1, 0),(2, -2, -1)$ & &
 \end{tabular}
 }
 }
 \end{table}
 \vfill
 \hspace{0pt}
}

\newpage

\hspace{0pt}
\vfill
 {
 \begin{table}[h!]
 \ContinuedFloat
 \centering
 \caption[]
 {}
 \rotatebox{0}{%
 \scalebox{0.85}{
 \begin{tabular}{l|l|l|c}
 \multicolumn{1}{c|}{ID} & \multicolumn{1}{c|}{$\supp(\Delta^{\FI})$} & \multicolumn{1}{c|}{$\vertex(\Delta^{\can} )$} & $\left\vert \pi_1(\mathcal S_\Delta) \right\vert$\\ \toprule \midrule
$483109 $&$(-1, -1, 1),(0, -1, 0),(0, -1, 1),(0, 2, -1),(1, -1, -1),$ & $\vertex(\Delta)$& $1$\\
& $(1, -1, 0),(1, -1, 1),(1, 0, -2),(1, 0, -1),(1, 0, 0),$ & & \\
& $(1, 0, 1)$ & & \\
$490478 $&(-2, -1, -1),(-1, 0, 0),(-1, 2, -1),(0, 1, 0),(1, -1, 1),&$\vertex(\Delta)$& $1$\\
& $(1, 1, 0),(1, 2, 0),(2, 0, 1),(3, 0, 1),(3, 1, 1), $ & & \\
& $(5, 0, 2)$ & & \\
$490481 $ &$(-1, -1, 2),(-1, -1, 3),(-1, 0, 1),(-1, 0, 2),(-1, 1, 0),$ &$\vertex(\Delta)$& $1$\\
& $(-1, 1, 1),(-1, 2, -1),(-1, 2, 0),(0, -1, 0),(0, -1, 1),$ & & \\
& $(0, -1, 2),(2, -2, -1)$ & & \\
$490485 $& $(-1, 0, 0),(-1, 0, 1),(-1, 0, 2),(-1, 1, 0),(-1, 1, 1),$&$\vertex(\Delta)$& $1$\\
& $(-1, 2, 0),(-1, 2, 1),(0, 1, -1),(0, 1, 0),(2, -2, -1)$ & & \\
$490511 $ &$(-1, -1, 2),(-1, 0, 1),(-1, 1, 0),(-1, 1, 1),(-1, 2, 0),$& $\vertex(\Delta),(2,1,2)$& $1$\\
& $(-1, 3, 0),(0, -1, 0),(0, 1, -1),(2, -2, -1)$ & & \\
$495687 $&$(-2, -1, -1),(-1, 0, 0),(-1, 2, -1),(-1, 3, -1),(0, 1, 0),$&$\vertex(\Delta)$ & $1$\\
& $(0, 4, -1),(1, -1, 1),(1, 1, 0),(1, 2, 0),(2, 0, 1),$ & & \\
& $(2, 3, 0),(3, 0, 1),(3, 1, 1),(4, 2, 1)$ & & \\
$499287 $ &$(-2, 1, 1),(-1, -1, 1),(-1, 0, 0),(-1, 1, -1),(0, -1, 0),$&$\vertex(\Delta),(-1,-1,-1)$ & $1$\\
& $(0, 0, -1),(1, -1, -1),(1, -1, 0),(1, 0, -1),(2, -1, -1),$ & & \\
& $(2, -1, 0),(2, 0, -1),(3, -1, -1)$ & & \\
$499291 $ & $(-2, 1, 1),(-1, -1, -1),(-1, -1, 0),(-1, -1, 1),(-1, 0, -1),$&$\vertex(\Delta)$ &$1$\\
& $(-1, 0, 0),(-1, 1, -1),(0, -1, -1),(0, -1, 0), (0, 0, -1),$ & & \\
& $(1, -1, -1),(1, -1, 0),(1, 0, -1),(2, -1, -1),(2, -1, 0),$ & & \\
& $(2, 0, -1),(3, -1, -1)$ & & \\
$499470 $ &$(-1, -1, 2),(-1, 0, 1),(-1, 1, 0),(-1, 1, 1),(-1, 2, -1),$& $\vertex(\Delta)$& $1$\\
& $(-1, 2, 0),(-1, 3, -1),(-1, 3, 0),(0, -1, 0),(0, 1, -1),$ & & \\
& $(2, -2, -1)$ & & \\
$501298 $& $(-2, -2, -1),(-1, -1, 0),(-1, 0, 1),(-1, 2, 2),(0, -1, 0),$&$\vertex(\Delta)$& $1$\\
& $(0, 1, 1),(1, -1, -1),(1, 0, 0),(1, 2, 1),(2, -1, -1),$ & & \\
& $(2, 1, 0),(3, -1, -2),(3, 0, -1),(4, -1, -2),(5, 0, -2),$ & & \\
& $(6, -1, -3)$ & & \\
$501330 $&$(-1, -1, 1),(-1, -1, 2),(-1, 0, 0),(-1, 0, 1),(-1, 1, 0),$& $\vertex(\Delta)$& $1$\\
& $(-1, 1, 1),(-1, 2, -1),(-1, 2, 0),(-1, 3, -1),(-1, 3, 0),$ & & \\
& $(0, -1, 0),(0, 1, -1),(2, -2, -1)$ & & \\\midrule
$354912 $ & $(-1, -1, 2),(-1, 0, 1),(-1, 1, 1),(-1, 2, 0),(-1, 3, 0),$& $\vertex(\Delta)$& $1$\\
& $(0, -1, 0),(0, 1, -1),(2, -2, -1)$ & & \\
$372528 $& $(-1, 0, 1),(-1, 0, 2),(-1, 1, 0),(-1, 1, 1),(-1, 2, -1),$ & $\vertex(\Delta)$& $1$\\
& $(-1, 2, 0),(0, -1, 0),(0, -1, 1),(0, -1, 2),(2, -2, -1)$ & & \\
$372973 $ & $(-1, -1, 2),(-1, 0, 1),(-1, 1, 0),(-1, 1, 1),(-1, 2, 0),$&$\vertex(\Delta)$ & $1$\\
& $(-1, 3, 0),(0, -1, 0),(0, 1, -1),(2, -2, -1)$ & & \\
$388701 $& $(-2, 1, 1),(-1, -1, 1),(-1, 0, 0),(-1, 1, -1),(0, -1, 0),$&$\vertex(\Delta)$ & $1$\\
& $(0, 0, -1),(1, -1, -1),(1, -1, 0),(1, 0, -1),(2, -1, -1),$ & & \\
& $(2, -1, 0),(2, 0, -1),(3, -1, -1)$ & &
 \end{tabular}
 }
 }
 \end{table}
 \vfill
 \hspace{0pt}
 }
\newpage

\section{Hollow $3$-topes with Non-empty Fine
 Interior} \label{appendixB}

\renewcommand{\thetable}{\Alph{section}.\arabic{table}}
\setcounter{table}{0}

\renewcommand{\thefigure}{\Alph{section}.\arabic{figure}}
\setcounter{figure}{0}

A lattice polytope $\Delta \subseteq \MQ$ is called {\em hollow} if it has no interior
lattice points in its relative interior, \emph{i.e.}, $\Delta^\circ \cap M = \emptyset$.
By a Theorem 1.3 in~\cite{Tre08}, any
$3$-dimensional hollow lattice polytope
can be projected to the unimodular
$1$-simplex,
to the double unimodular
$2$-simplex, or is an exceptional hollow $3$-tope,
whereas up to unimodular transformation there exist only a
finite number of these. This theorem implies that a hollow
$3$-tope with non-empty Fine interior has to be exceptional because the
unimodular $1$-simplex and the double unimodular $2$-simplex have empty
Fine interior.
Treutlein
has found $9$ maximal exceptional hollow polytopes,
which was not an complete
list. Averkov et al.~\cite{AWW11,AKW17} have found the complete list consisting of $12$
maximal exceptional hollow $3$-topes $\Delta_i$ $(1 \leq i \leq 12)$ (Table \ref{tbl:hollow1}, Figure \ref{fig:h}).
Computations show that exactly $9$ of $12$
maximal exceptional hollow $3$-topes $\Delta_i$ have
non-empty Fine interior $\Delta_i^{\FI}$ (Table \ref{tbl:hollow1}). Moreover, no one of these $9$ polytopes
contains a proper lattice $3$-subpolytope with non-empty Fine interior.
Thus, there exist exactly $9$ hollow $3$-topes $\Delta_i$ with non-empty
Fine interior $\Delta_i^{\FI}$.

 \begin{table}[h!]
 \centering
  \rotatebox{0}{%
 \scalebox{0.9}{
 \begin{tabular}{c|l|c|c|l|c}
 \multicolumn{1}{c|}{$i$} &
 \multicolumn{1}{c|}{$\vertex(\Delta_i)$} &
 \multicolumn{1}{c|}{$\text{w}(\Delta_i)$} &
 \multicolumn{1}{c|}{$\dim(\Delta_i^{\FI})$} &
 \multicolumn{1}{c|}{$\vertex(\Delta_i^{\FI})$} & \multicolumn{1}{c}{ $\left\vert \pi_1(\mathcal S_\Delta) \right\vert$} \\ \toprule \midrule
$1$ & $(0,0,0),(6,0,0),(3,3,0),(4,0,2)$ &$2$& $-1$ & $\emptyset$& $1$ \\
$2$ & $(0,0,0),(4,0,0),(0,4,0),(2,0,2)$&$2$ &$-1$ &$\emptyset$ & $1$ \\
$3$ & $(0,0,0),(3,0,0),(0,3,0),(3,0,3)$& $3$& $-1$ &$\emptyset$ &$1$ \\ \midrule
$4$ &$ (0,0,0),(4,0,0),(2,4,0),(3,0,2)$&$2$& $0$ & $1/2 \cdot (5, 1, 2)$& $2$\\
$5$ &$(0,0,0),(2,2,0),(1,1,2),(3,-1,2) $&$2$& $0$ & $ 1/2 \cdot (3, 1, 2)$&$2$ \\
$6$ &$(0,0,0),(2,2,0),(4,0,0),(2,-2,0),$ &$2$& $0$ & $1/2 \cdot (5, 1, 2)$&$2$ \\
& $(3,1,2)$&& & & \\
$7$ & $(0,0,0),(1,1,0),(2,-2,0),(3,-1,0),$ &$2$& $0$ & $ 1/2 \cdot (3,- 1, 2) $&$2$ \\
& $(1,-1,2),(2,0,2)$& & & & \\
$8$ &$(0,0,0),(1,1,0),(1,-1,0),(2,0,0),$ &$2$& $0$ & $ 1/2 \cdot (3,- 1, 2)$&$2$ \\
& $(1,-1,2),(2,0,2),(2,-2,2),(3,-1,2)$ && & & \\ \midrule
$9$ &$(0,0,0),(3,0,0),(1,3,0),(2,0,3)$ &$3$& $1$& $ (4/3, 1, 1), (5/3, 1, 1) $&$3$ \\
$10$ & $(0,0,0),(1,2,0),(1,-1,0),(3,0,0),$ &$3$& $1$& $ (4/3, 2/3, 1), (5/3, 1/3, 1 ) $& $3$\\
& $(2,1,3)$& & & & \\
$11$ &$(0,0,0),(1,1,0),(3,0,0),(2,-1,0),$ &$3$& $1$& $ (5/3, 2/3, 1),(7/3, 1/3, 1) $&$3$ \\
& $(4,1,3),(2,2,3)$ && & & \\ \midrule
$12$ &$(-1,0,0),(0,1,-2),(1,2,1),(2,-2,-1)$ &$3$& $3$ & $ (1/5, 1/5, -2/5),
 (2/5, 2/5, -4/5), $ & $5$ \\
 & & && $ (3/5, 3/5, -1/5), (4/5, -1/5, -3/5) $&
 \end{tabular}
 }
 }
 \vspace{1em}
 \caption[]
 {{\bf {\boldmath $12$} Maximal Hollow {\boldmath $3$}-topes.}
 Table contains:
 index $i$ of the maximal hollow $3$-tope $\Delta_i$,
vertices $\vertex(\Delta_i)$ of $\Delta_i$, lattice width $\text{w}(\Delta_i)$\footnotemark
of $\Delta_i$,
dimension $\dim(\Delta_i^{\FI})$ of Fine interior $\Delta_i^{\FI}$,
vertices $\vertex(\Delta_i^{\FI})$ of $\Delta_i^{\FI}$, and order of fundamental group
$\left\vert \pi_1(\mathcal S_\Delta) \right\vert$ of the minimal model $\mathcal S_{\Delta_{i}}$.}
 \label{tbl:hollow1}
 \end{table}

\footnotetext{The {\em lattice width} $\text{w}(\Delta)$ of a lattice polytope $\Delta$ is defined
as the infimum of $\max_{ x \in \Delta} \langle x, u \rangle - \min_{ x \in \Delta} \langle x, u \rangle $
 for all non-zero integer lattice directions $u$.}

It is remarkable that all minimal surfaces
$\mathcal S_{\Delta_i}$ corresponding to these $9$ hollow $3$-topes $\Delta_i$
have non-trivial fundamental group $\pi_1(\mathcal S_\Delta)$ of order $2$, $3$, or $5$ (Table \ref{tbl:hollow1}).
There exist exactly $5$ hollow $3$-topes $\Delta_i$ with $0$-dimensional
Fine interior $\Delta_i^{\FI} = \{R\}$, where $R \in \frac{1}{2}M \setminus M$ is a
rational point (Table \ref{tbl:hollow1}).
The normal fans $\Sigma^{\Delta_i}$ of these $5$ hollow polytopes $\Delta_i$
define $5$ toric Fano threefolds $X_{\Sigma^{\Delta_i}}$
with at worst canonical singularities (Table \ref{tbl:hollow3}). 
These Fano threefolds can be
obtained as quotients of
Gorenstein toric Fano threefolds $X_{\Sigma_{\Delta''}}$ in the following $5$ ways:

\begin{itemize} \label{quotient}
\item[\mylabel{itm:i}{(i)}] $\PP(1,1,2,4)$ with a $\mu_2$-action  given by 
$(x_0,x_1,x_2,x_3) \mapsto (x_0,-x_1,-x_2,-x_3)$;
\item[\mylabel{itm:ii}{(ii)}] $\PP^3$ with a $\mu_4$-action  given by 
$(x_0,x_1,x_2,x_3) \mapsto (x_0,ix_1,-x_2,-ix_3)$;
\item[\mylabel{itm:iii}{(iii)}] $\{ x_1x_2 - x_3x_4 =0 \} \subseteq \PP(2,1,1,1,1)$ 
with a $\mu_2$-action  given by 
  $(x_i)_{0 \leq i \leq 4} \mapsto (-x_0,-x_1,-x_2,x_3,x_4)$;
\item[\mylabel{itm:iv}{(iv)}] $\PP^1 \times \PP(1,1,2)$ with a $\mu_2$-action  given by 
$(x_0,x_1,y_0,y_1,y_2) \mapsto (x_0,-x_1,y_0,-y_1,-y_2)$;
\item[\mylabel{itm:v}{(v)}] $\PP^1 \times \PP^1 \times \PP^1$ with a $\mu_2$-action  given by 
 $(x_0,x_1,y_0,y_1, z_0,z_1) \mapsto (x_0,-x_1, y_0,-y_1, z_0, -z_1)$.
\end{itemize}

 \begin{table}[h!]
 \centering
 \rotatebox{0}{%
 \scalebox{1}{
 \begin{tabular}{c|l|l|l|l|c}
 \multicolumn{1}{c|}{$i$} &
 \multicolumn{1}{c|}{$\Sigma^{\Delta_i}(1)$} &
 \multicolumn{1}{c|}{ID($\Delta'$)} &
 \multicolumn{1}{c|}{$\Sigma_{\Delta'}(1)$} &
 \multicolumn{1}{c|}{ID($\Delta''$)} &
 \multicolumn{1}{c}{$X_{\Sigma_{\Delta''}}$} \\ \toprule \midrule
$4$ & $(2, -1, -3),(0,0,1),(0,1,0),$& $547354$& $(-2, -3, -5), (2, 1, 1), (0, 1, 0),$ &$547363$ &(i)\\
& $(-2,-1,-1)$ & &$(0, 0, 1)$&&\\
$5$ & $(1,-1,0),(1,1,-1),(-1,1,2),$ & $547364$ &$ (0, 0, 1),(0, 2, 1),(2, 1, 0),$ & $547367$ &(ii) \\
& $(-1,-1,-1)$ & &$ (-2, -3, -2)$&&\\
$6$ & $(0,0,1),(1,1,-2),(1,-1,-1),$ & $544353$ &$ (1, 0, 0),(0, 1, 0),(-2, -1, 0), $ & $544357$ &(iii) \\
&$(-1,-1,0),(-1,1,-1)$ & &$(1, 1, 2),(-3, -1, -2)$&&\\
$7$ & $(0,0,1),(-1,1,-1),(1,-1,-1),$ & $544310$ &$(-1, -1, -2),(1, 1, 2),(-2, -1, 0),$& $544342$ &(iv)\\
&$(1,1,0),(-1,-1,0)$ & &$(0, 1, 0),(1, 0, 0)$&&\\
$8$ & $(-1,1,1),(0,0,-1),(-1,-1,0),$ & $520134$&$ (1, 0, 0),(0, 1, 0),(1, 1, 2), $& $520140$ &(v)\\
&$(1,1,0),(1,-1,-1),(0,0,1)$ & &$(-1, 0, 0), (0, -1, 0),(-1, -1, -2)$&&
 \end{tabular}
 }
 }
 \vspace{1em}
  \caption[]
 {{\bf {\boldmath $5$} Hollow {\boldmath $3$}-topes with {\boldmath$0$}-dimensional Fine
 Interior.}
 Table contains: index $i$ of the maximal hollow $3$-tope $\Delta_i$,
 rays of the normal fan $\Sigma^{\Delta_i}$ corresponding to
 $\Delta_i$,
 \href{http://www.grdb.co.uk/forms/toricf3c}{ID}\footnotemark
 of the canonical Fano $3$-tope
 $\Delta'$ such that $\Sigma^{\Delta_i} \cong \Sigma_{\Delta'}$,
 rays of the spanning fan $\Sigma_{\Delta'}$,
 \href{http://www.grdb.co.uk/forms/toricf3c}{ID}
 of the reflexive canonical Fano $3$-tope
 $\Delta''$ used to construct the Gorenstein toric Fano threefold $X_{\Sigma_{\Delta''}}$ and obtain the toric Fano threefold $X_{\Sigma_{\Delta'}}$
 with at worst canonical singularities as a $\mu_2$ quotient, and reference
 to the corresponding Gorenstein toric Fano threefold
$X_{\Sigma_{\Delta''}}$ including the precise $\mu_2$
action on page \pageref{quotient}.
 }
 \label{tbl:hollow3}
 \end{table}

 In addition, Table \ref{tbl:hollow2} contains the support $\supp(\Delta_i^{\FI})$ of the Fine
 interior $\Delta_i^{\FI}$ and the vertices of the canonical hull $\Delta_i^{\can}$ for all
 $9$ hollow polytopes $\Delta_i$ with non-empty Fine interior $\Delta_i^{\FI}$.

 \begin{table}[h!]
 \centering
 \rotatebox{0}{%
 \scalebox{1}{
 \begin{tabular}{c|l|c}
 \multicolumn{1}{c|}{$i$} &
 \multicolumn{1}{c|}{$\supp(\Delta_i^{\FI})$} &
 \multicolumn{1}{c}{$\vertex(\Delta_i^{\can})$} \\ \toprule \midrule
$4$ & $ (-2, -1, -1),
 (0, -1, -2),
 (2, -1, -3),
 (0, 0, 1),
 (0, 0, -1),
 (0, 1, 0)$& $\vertex(\Delta_i)$ \\
$5$ & $ (1, -1, 0),
 (1, 1, -1),
 (0, 0, 1),
 (0, 0, -1),
 (-1, -1, -1),
 (-1, 1, 2)$& $\vertex(\Delta_i)$ \\
$6$ & $ (1, 1, -2),
 (1, -1, -1),
 (-1, -1, 0),
 (-1, 1, -1),
 (0, 0, 1),
 (0, 0, -1)$& $\vertex(\Delta_i)$ \\
$7$ & $(1, 1, 0),
 (1, -1, -1),
 (-1, -1, 0),
 (-1, 1, -1),
 (0, 0, 1),
 (0, 0, -1)$& $\vertex(\Delta_i)$\\
$8$ & $(1, 1, 0),
 (1, -1, -1),
 (-1, -1, 0),
 (0, 0, 1),
 (0, 0, -1),
 (-1, 1, 1)$& $\vertex(\Delta_i)$ \\ \midrule
$9$ & $(0, -1, -1),
 (0, 0, 1),
 (3, -1, -2),
 (0, 1, 0),
 (-3, -2, -1)$& $\vertex(\Delta_i)$ \\
$10$ & $ (-1, 2, -1),
 (1, 1, -1),
 (-1, -1, 0),
 (2, -1, -1),
 (0, 0, 1)$ & $\vertex(\Delta_i)$\\
$11$ & $ (1, -1, 0),
 (0, 0, 1),
 (-1, -2, 1),
 (-1, 1, 0),
 (1, 2, -2)$& $\vertex(\Delta_i)$ \\ \midrule
$12$ & $(1, 1, 1),
 (1, -1, 0),
 (-2, -1, 1),
 (0, 1, -2)$& $\vertex(\Delta_i)$
 \end{tabular}
 }
 }
 \vspace{1em}
  \caption[]
 {{\bf {\boldmath $9$} Hollow {\boldmath $3$}-topes with Non-empty Fine Interior.}
 Table contains:
 index $i$ of the maximal hollow $3$-tope $\Delta_i$,
support $\supp(\Delta_i^{\FI})$ of $\Delta_i^{\FI}$, and
vertices of the canonical hull $\Delta_i^{\can}$.}
 \label{tbl:hollow2}
 \end{table}

 \footnotetext{ID used in the \href{http://www.grdb.co.uk}{Graded Ring Database}. \label{foot1}}
 
 \newpage

 \begin{figure}
	\centering
	\begin{subfigure}{6.5cm}
	\fbox{
	\begin{minipage}[c][6cm]{6cm}
		\centering \includegraphics[width=6cm]{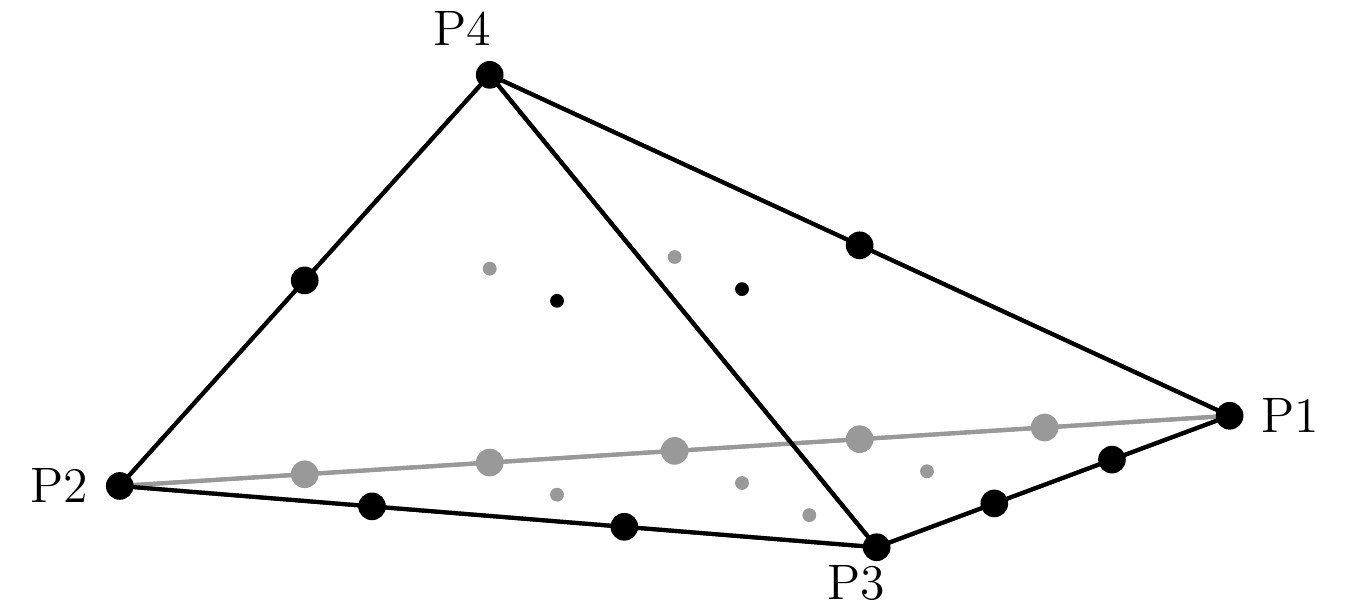}
	\end{minipage}
	}
	\caption{$\Delta_{1}$.}
	\label{fig:h1}
	\end{subfigure}
\hspace{0.5cm}
	\begin{subfigure}{6.5cm}
	\fbox{
	\begin{minipage}[c][6cm]{6cm}
	\centering \includegraphics[width=6cm]{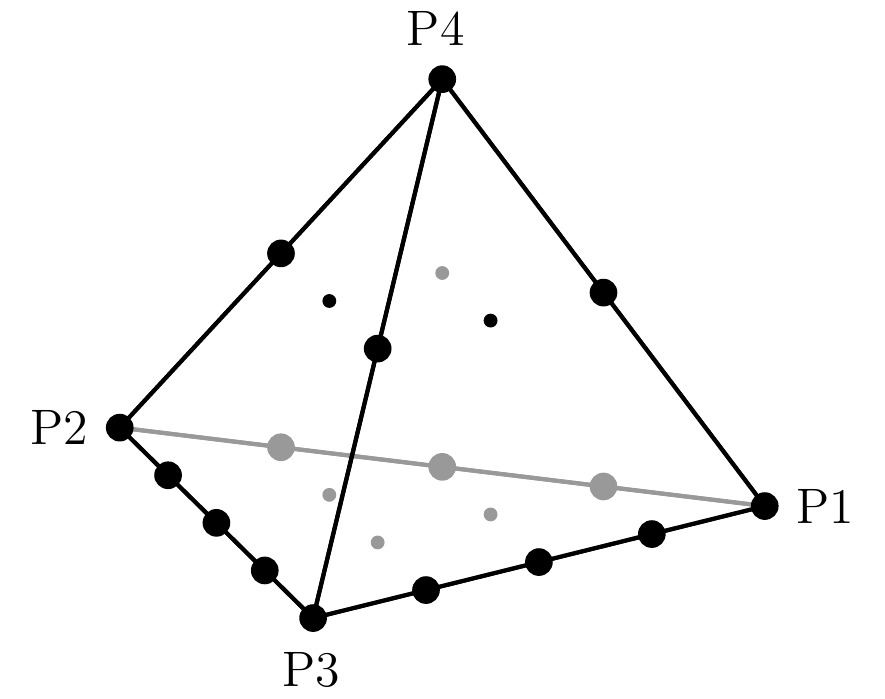}
	\end{minipage}
	}
	\caption{$\Delta_{2}$.}
	\label{fig:h2}
	\end{subfigure}
	\vspace{0.5cm}\\
		\begin{subfigure}{6.5cm}
	\fbox{
	\begin{minipage}[c][6cm]{6cm}
		\centering \includegraphics[width=6cm]{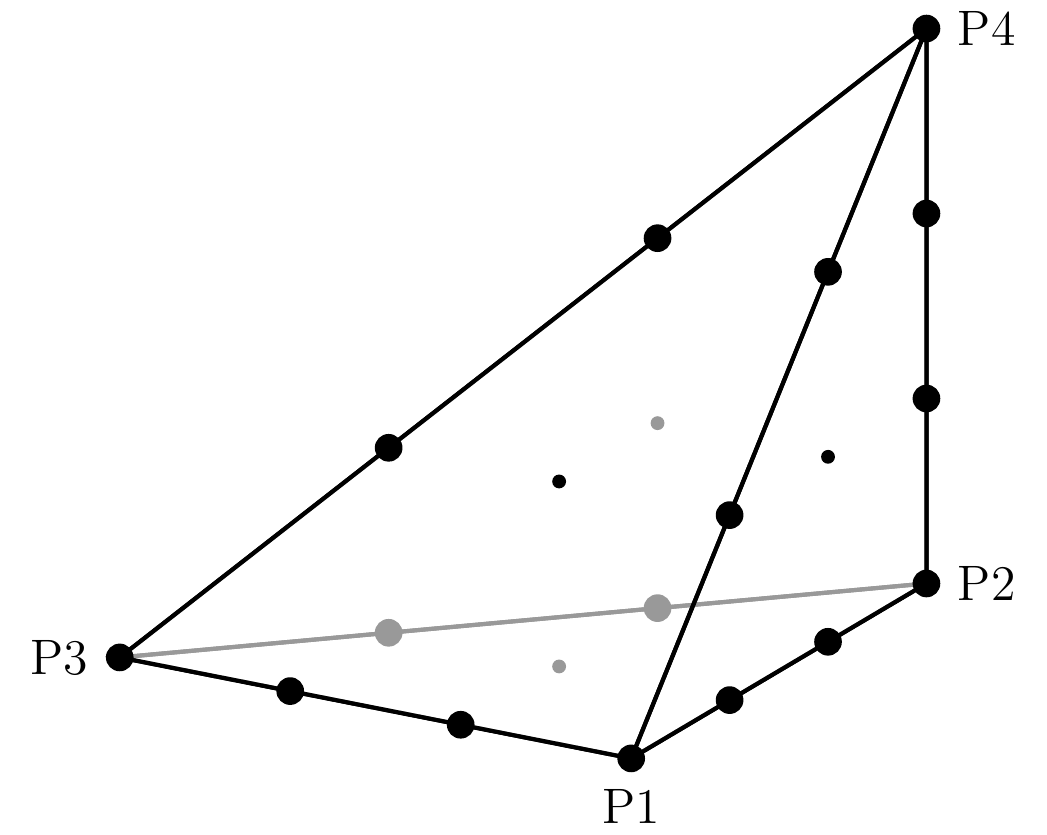}
	\end{minipage}
	}
	\caption{$\Delta_{3}$.}
	\label{fig:h3}
	\end{subfigure}
\hspace{0.5cm}
	\begin{subfigure}{6.5cm}
	\fbox{
	\begin{minipage}[c][6cm]{6cm}
	\centering \includegraphics[width=6cm]{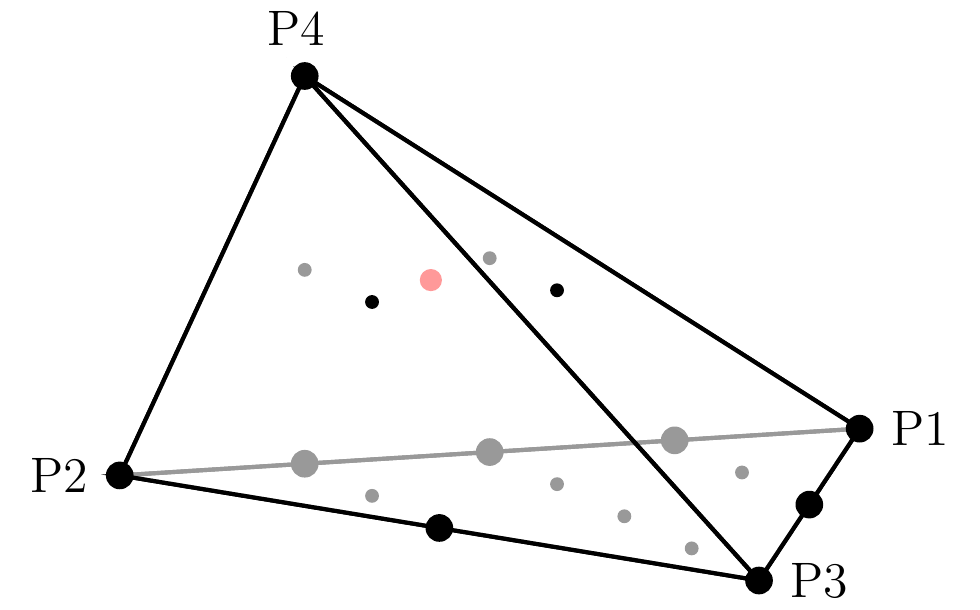}
	\end{minipage}
	}
	\caption{$\Delta_{4}$.}
	\label{fig:h4}
	\end{subfigure}
	\vspace{0.5cm}\\
		\begin{subfigure}{6.5cm}
	\fbox{
	\begin{minipage}[c][6cm]{6cm}
		\centering \includegraphics[width=6cm]{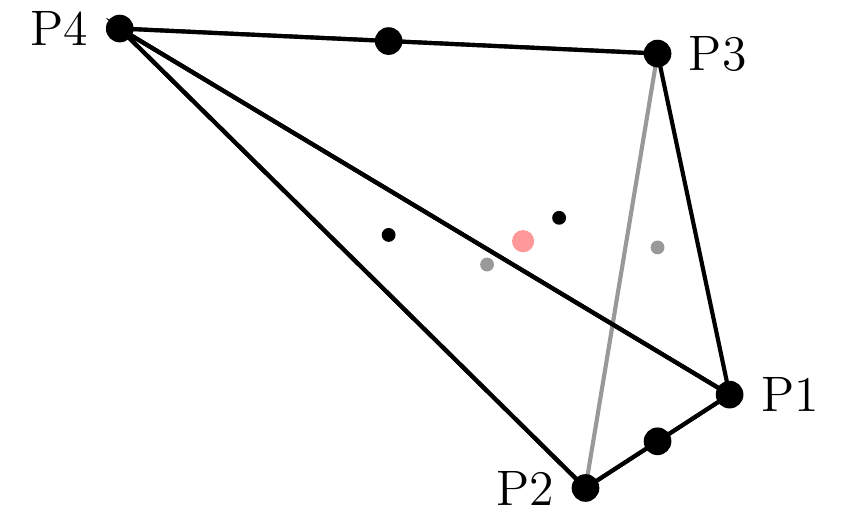}
	\end{minipage}
	}
	\caption{$\Delta_{5}$.}
	\label{fig:h5}
	\end{subfigure}
\hspace{0.5cm}
	\begin{subfigure}{6.5cm}
	\fbox{
	\begin{minipage}[c][6cm]{6cm}
	\centering \includegraphics[width=6cm]{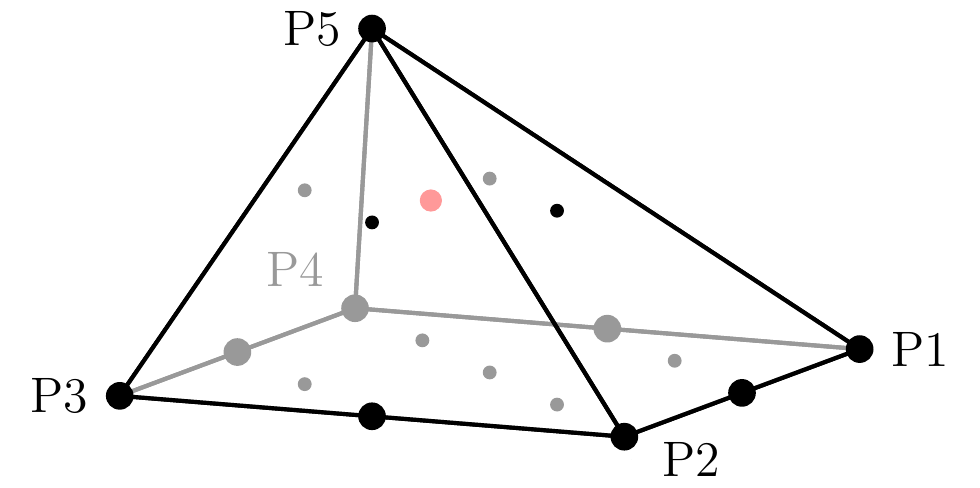}
	\end{minipage}
	}
	\caption{$\Delta_{6}$.}
	\label{fig:h6}
	\end{subfigure}
	\\
	\caption[]
	{ {\bf {\boldmath $12$} Maximal Hollow {\boldmath $3$}-topes.}
	Shaded faces are occulded. The Fine interior is coloured red.}
	\label{fig:h}
\end{figure}
\phantom{.}

\newpage

 \begin{figure}
 \ContinuedFloat
	\centering
	\begin{subfigure}{6.5cm}
	\fbox{
	\begin{minipage}[c][6cm]{6cm}
		\centering \includegraphics[width=6cm]{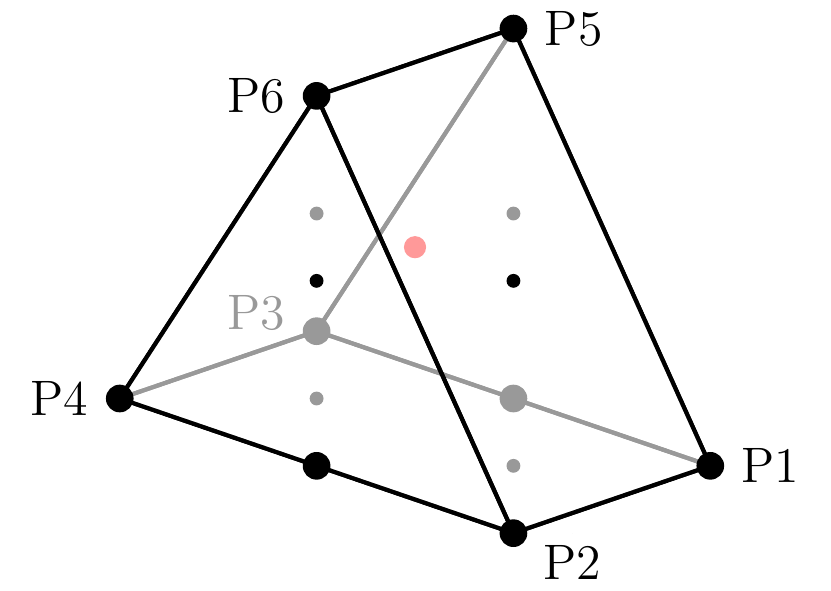}
	\end{minipage}
	}
	\caption{$\Delta_{7}$.}
	\label{fig:h7}
	\end{subfigure}
\hspace{0.5cm}
	\begin{subfigure}{6.5cm}
	\fbox{
	\begin{minipage}[c][6cm]{6cm}
	\centering \includegraphics[width=6cm]{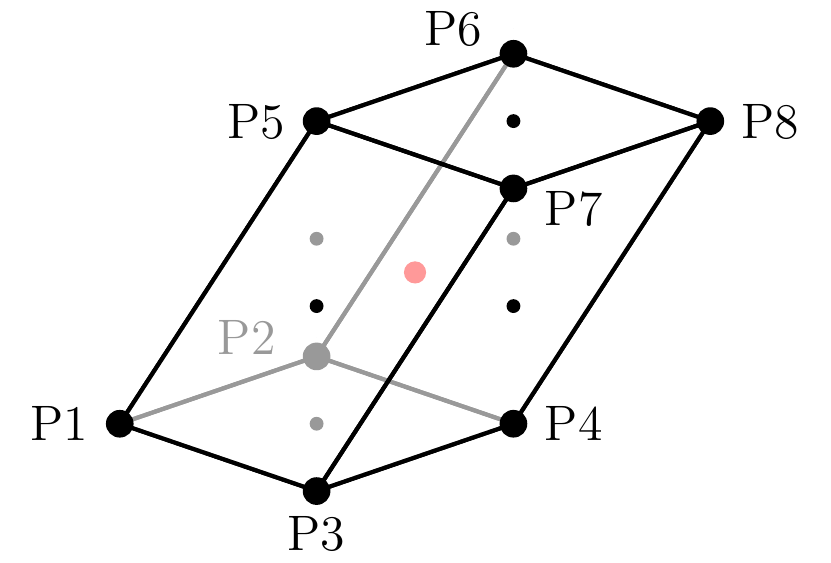}
	\end{minipage}
	}
	\caption{$\Delta_{8}$.}
	\label{fig:h8}
	\end{subfigure}
	\vspace{0.5cm}\\
		\begin{subfigure}{6.5cm}
	\fbox{
	\begin{minipage}[c][6cm]{6cm}
		\centering \includegraphics[width=6cm]{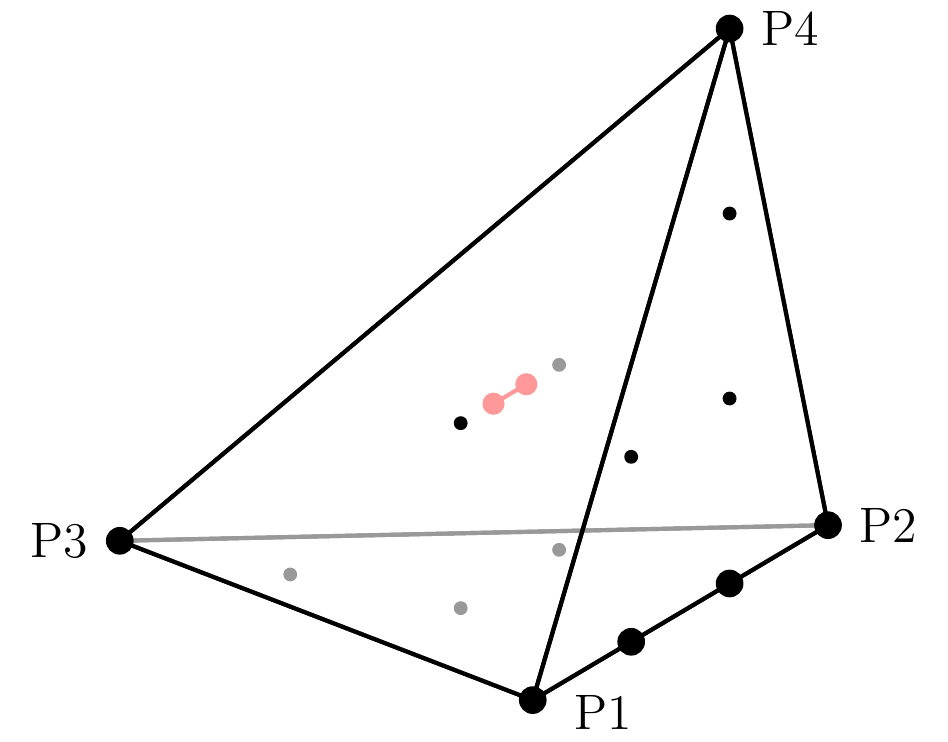}
	\end{minipage}
	}
	\caption{$\Delta_{9}$.}
	\label{fig:h9}
	\end{subfigure}
\hspace{0.5cm}
	\begin{subfigure}{6.5cm}
	\fbox{
	\begin{minipage}[c][6cm]{6cm}
	\centering \includegraphics[width=6cm]{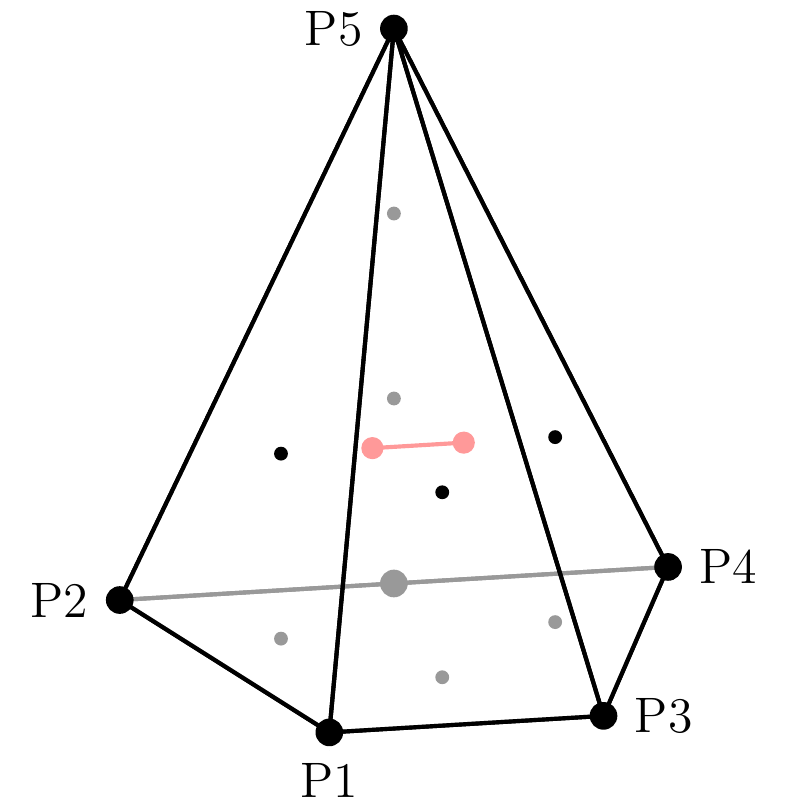}
	\end{minipage}
	}
	\caption{$\Delta_{10}$.}
	\label{fig:h10}
	\end{subfigure}
	\vspace{0.5cm}\\
		\begin{subfigure}{6.5cm}
	\fbox{
	\begin{minipage}[c][6cm]{6cm}
		\centering \includegraphics[width=6cm]{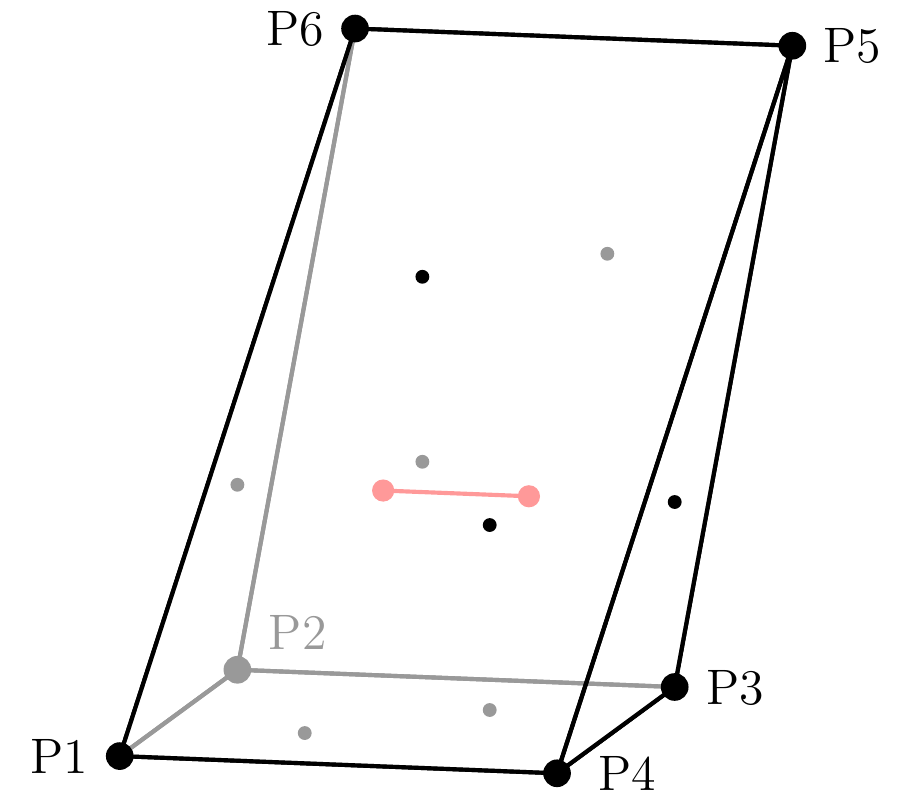}
	\end{minipage}
	}
	\caption{$\Delta_{11}$.}
	\label{fig:h11}
	\end{subfigure}
\hspace{0.5cm}
	\begin{subfigure}{6.5cm}
	\fbox{
	\begin{minipage}[c][6cm]{6cm}
	\centering \includegraphics[width=6cm]{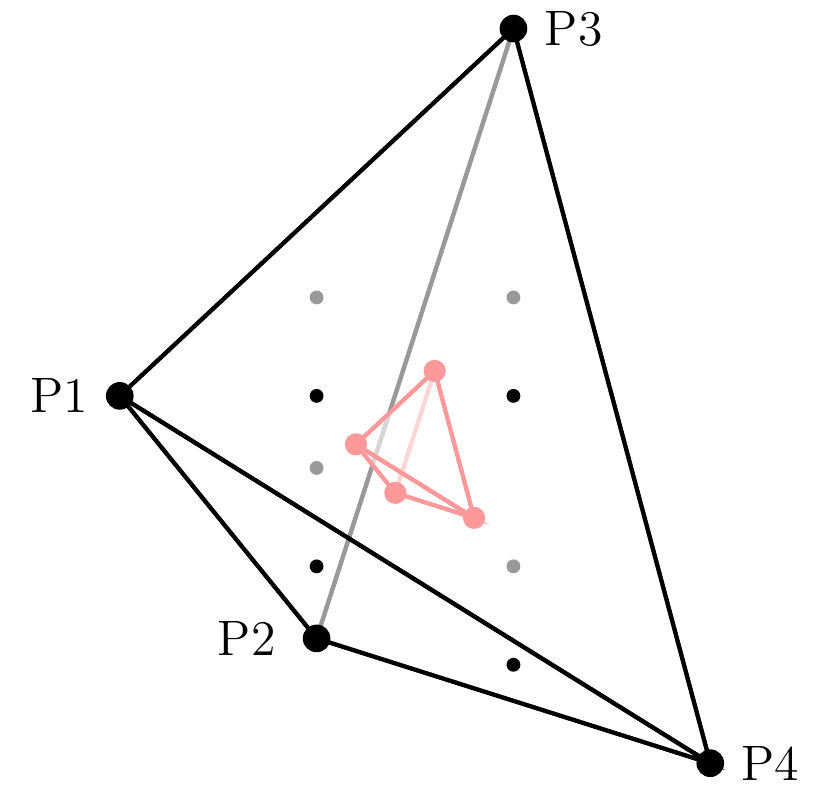}
	\end{minipage}
	}
	\caption{$\Delta_{12}$.}
	\label{fig:h12}
	\end{subfigure}
	\\
	\caption[]
	{ }
\end{figure}
\phantom{.}

\newpage

\bibliographystyle{amsalpha}

\end{document}